\documentclass[11pt]{article}
\usepackage[left=1in,top=1in,right=1in,bottom=1in,head=.1in,nofoot]{geometry}

\usepackage{setspace,url,bm,amsmath} %

\usepackage{titlesec} %
\titlelabel{\thetitle.\quad} %

\titleformat*{\section}{\bf\Large\center}

\usepackage{graphicx} %
\usepackage{bbm}
\usepackage{latexsym}
\usepackage{caption}
\usepackage[margin=20pt]{subcaption}
\usepackage{hyperref}
\usepackage{booktabs}
\usepackage{multirow}

\usepackage{enumitem}

\usepackage[table]{xcolor}

\newcommand{\GG}[1]{}

\usepackage{amsthm}
\usepackage{amssymb}
\usepackage{amsmath}
\usepackage{color}

\usepackage{comment}
\theoremstyle{definition}
\newtheorem{assumption}{Assumption}
\newtheorem*{theorem*}{Theorem}
\newtheorem{theorem}{Theorem}
\newtheorem*{rmk*}{Remark}

\newtheorem{lemma}{Lemma}

\newtheorem{condition}{Condition}

\newtheorem*{corollary*}{Corollary}

\usepackage{natbib} %
\bibpunct{(}{)}{;}{a}{}{,} %

\usepackage{etoolbox} %
\apptocmd{\sloppy}{\hbadness 10000\relax}{}{} %

\usepackage{color}
\usepackage{listings}

\def\I{\mathbbm{1}}
\def\pr{\text{pr}}
\def\LR{\text{LR}}
\def\SLR{\text{SLR}}
\def\Var{\text{Var}}

\def\E{\mathbb{E}}
\def\TST{\mathfrak{T}}
\def\ind{\begin{picture}(9,8)
	\put(0,0){\line(1,0){9}}
	\put(3,0){\line(0,1){8}}
	\put(6,0){\line(0,1){8}}
	\end{picture}
}
\def\converged{\stackrel{d}{\longrightarrow}}
\def\convergep{\stackrel{p}{\longrightarrow}}
\def\convergeas{\stackrel{a.s.}{\longrightarrow}}
\def\converge{\stackrel{}{\longrightarrow}}
\def\ED{M}
\def\Binomial{\text{Bin}}
\def\Hypergeometric{\text{HGeom}}
\def\Bernoulli{\text{Bern}}
\def\exponential{\text{Expo}}

\newcommand\independent{\protect\mathpalette{\protect\independenT}{\perp}}
\def\independenT#1#2{\mathrel{\rlap{$#1#2$}\mkern2mu{#1#2}}}

\usepackage[textsize=scriptsize, textwidth = 2cm, shadow]{todonotes}

\RequirePackage[normalem]{ulem}

\providecommand{\del}[1]{{\protect\color{red}\sout{#1}}}
\providecommand{\del}[1]{}

\allowdisplaybreaks

\numberwithin{equation}{section}
\theoremstyle{plain}

\def\rev{\color{black}}

\begin{document}

\onehalfspacing

\title{\bf Randomization-based Test for Censored Outcomes: A New Look at the Logrank Test
}
\author{
	Xinran Li and Dylan S. Small
\footnote{
    Xinran Li, Department of Statistics, University of Illinois, Champaign, IL 61820 (E-mail: \href{mailto:xinranli@illinois.edu}{xinranli@illinois.edu}).
Dylan S. Small, Department of Statistics, University of Pennsylvania, Philadelphia, PA 19104 (E-mail: \href{mailto:dsmall@wharton.upenn.edu}{dsmall@wharton.upenn.edu}).
}
}
\date{}
\maketitle

\begin{abstract}
Two-sample tests with censored outcomes 
are a classical topic in statistics with wide use even in cutting edge applications. 
There are 
at least 
two modes of inference used to justify two-sample tests. 
One is usual superpopulation inference assuming that units are independent and identically distributed (i.i.d.) samples from some superpopulation; 
the other is finite population inference that relies on the random assignments of units into different groups.  
When randomization is actually implemented, 
the latter has the advantage of avoiding distributional assumptions on the outcomes.
In this paper, we focus on finite population inference for censored outcomes, which has been less explored in the literature. Moreover, we allow the censoring time to depend on treatment assignment, under which exact permutation inference is unachievable. 
We find that, surprisingly, the usual logrank test can also be justified by randomization. 
Specifically, 
under a Bernoulli randomized experiment with 
non-informative i.i.d.\ censoring, 
the logrank test is asymptotically valid for testing Fisher's null hypothesis 
of no treatment effect
on any unit. 
The asymptotic validity of the logrank test does not require any distributional assumption on the potential event times.
We further extend the theory to the stratified logrank test, which is useful for randomized block designs and when censoring mechanisms vary across strata. 
In sum, 
the developed theory for the logrank test from finite population inference supplements its classical theory from usual superpopulation inference, 
and helps provide a broader justification for the logrank test. 
\end{abstract}

{\bf Keywords}: 
Logrank test; Randomization; Non-informative censoring; Potential outcome; Stratified logrank test

\section{Introduction}

\subsection{Superpopulation versus finite population inference}

The two-sample test has been one of the most classical and important topics in statistics.  
Most theoretical justification for two-sample tests relies on the assumption of independent and identically distributed (i.i.d.) sampling of units from {\rev some} population. 
However, in the context of randomized experiments, which are widely used in industry \citep[e.g.,][]{box2005statistics, dasgupta2014causal}, clinical trials \citep[e.g.,][]{rosenberger2015randomization}, and more recently in social science \citep[e.g.,][]{ATHEY201773} and technology companies 
\citep[e.g.,][]{basse2019minimax, Bojinov2019, bojinov2020design},  
it is often more natural to conduct inference based on randomization of the treatment assignments. 
This is often called randomization-based or design-based inference, as well as finite population inference \citep{Fisher:1935, Neyman:1923}. 
Specifically, 
under finite population inference, each unit's potential outcome 
\citep{Neyman:1923,Rubin:1974} 
under either treatment is viewed as a fixed constant (or equivalently being conditioned on), 
and the randomness in the observed data comes solely from the physical randomization of treatment assignments, which acts as the ``reasoned basis'' for inference \citep{Fisher:1935}.

Compared to superpopulation inference, finite population inference has the following two advantages. 
First, finite population inference focuses explicitly on the experimental units by conditioning on their potential outcomes, 
which 
in some sense makes our inference more relevant for the units in hand. 
On the contrary, 
the superpopulation inference assumes that the units are i.i.d.\ samples from an often hypothetical superpopulation, and focuses on inference about the superpopulation distribution. 
Second, 
the validity of superpopulation inference generally relies crucially on the i.i.d.\ sampling of units. 
However, the validity of finite population inference relies on the randomization of treatment assignments, which can be guaranteed by design, 
and does not require the units' potential outcomes to be either independent or identically distributed. 
Therefore, 
finite population inference 
can avoid distributional assumptions on the units
and thus provides inference procedures with broader justifications, 
at least in the context of randomized experiments. 
For a more detailed comparison between these two types of inference, see, e.g.,  \citet{Little2004} with an emphasis on survey sampling, 
\citet[][Chapter 2.4.5]{rosenbaum2002observational} with an emphasis on causal inference 
and \citet{Imbens2020} with an emphasis on regression analysis. 
 
\subsection{Two-sample tests for censored outcomes}

In many clinical trials, the primary outcome is time to a certain event. 
The outcome may be censored, %
rendering the usual two-sample tests inapplicable. 
In the presence of censoring, 
\citet{Gehan1965} proposed a permutation test using a generalized Wilcoxon statistic, and \citet{mantel1966} proposed the logrank test by combining the Mantel--Haenszel statistics %
for all contingency tables during the study period. 
\cite{Prentice1978} extended the two tests to more general linear rank tests. 
\cite{Aalen1978} and \citet{Gill1980} 
provided a full and rigorous study of the asymptotic properties of these %
tests. 
Most theoretical justifications for these two-sample tests have assumed that the units 
are i.i.d.\  
samples from some superpopulation \citep{Harrington1991}. 
However, this i.i.d.\ assumption can be violated 
when the units' outcomes are inhomogeneous or dependent on each other due to some (hidden)
risk factors. 
For example, in a multicenter clinical trial, 
each patient's outcome may depend on his/her own characteristics such as age and sex, 
and the outcomes for patients within the same center may be correlated due to the same environment or other center-specific factors \citep{Cai1999}.

There have been fewer studies of two-sample tests for censored outcomes under finite population inference. 
To test Fisher's null hypothesis of no treatment effect for any unit, 
\citet{rosenbaum2002observational} proposed randomization tests for censored outcomes using a partial ordering, 
and 
\citet{zhang2005} established asymptotic Normality of 
the randomization distribution 
of the logrank statistic. 
Both approaches require the assumption
of identical
potential censoring times under 
treatment and control. 
Under this identical censoring assumption, Fisher's null hypothesis is sharp, 
the null distributions of test statistics (e.g., logrank statistic) are known exactly, and the asymptotics provides
mainly a 
numerical approximation. 
However, as \citet{zhang2005} commented, 
the assumption of identical censoring can be violated due to toxicity or a side effect of either treatment. 
When the censoring times can be different
under the two treatment arms, 
Fisher's null hypothesis of no effect becomes composite.
In this case, understanding the randomization distribution of the logrank statistic becomes more challenging, and in the meanwhile more critical for inference, 
because the exact randomization distribution is no longer known.

\subsection{Our contribution}

\citet{Neyman:1923} first studied the randomization distribution of 
the difference-in-means statistic 
under finite population inference, and provided an asymptotic valid test for a non-sharp null hypothesis, for which no exact permutation tests are available and large-sample approximations become essential. 
In particular, 
he 
showed that the classical $t$-test is still asymptotically valid under randomization, although in a generally conservative way.
In the presence of censoring where the censoring can depend on the treatment, 
Fisher's null is composite,
and 
there has not been any 
test
available that can be justified by randomization. 
As \citet{zhang2005} commented, it does not appear that randomization-based techniques can be used to perform the analysis.
In this paper, we will show that it is still possible to conduct randomization-based tests in large samples, and our approach is based on the logrank test, one of the most popular two-sample tests for censored outcomes. 

From \citet{Neyman:1923}, 
test statistics that are justified under i.i.d.\  sampling from a superpopulation can sometimes also be justified under randomization of a finite population \citep{ding2017}.  %
However, this is not always true
\citep[see, e.g.,][]{freedman2008, freedman2008regression_a, freedman2008regression_b}. 
Our question is, can the logrank test 
for studying censored time-to-event outcomes 
be justified under randomization of a finite population?

We will 
show that as long as 
the censoring is non-informative and 
the potential censoring times 
are i.i.d.\ across all units, 
the logrank test 
can 
be justified by randomization, without requiring any distributional assumption on the potential event times.  
Specifically, under a Bernoulli randomized experiment and 
non-informative 
i.i.d.\ censoring, 
if the treatment has no effect on any unit's event time, then the 
randomization distribution of the 
logrank statistic, i.e., the distribution conditional on all the potential event times, is asymptotically standard Gaussian. 
Our proof makes use of a novel martingale construction, the martingale central limit theorem, and Gaussian approximation of hypergeometric distributions.
Our result shows that 
in practice we can be more confident in using the logrank test, when the units are randomly assigned to the two treatment arms and the censoring is i.i.d. across units 
under each treatment arm. 
The result is particularly useful if the experimenter is able to control both the treatment assignment and the censoring mechanisms, 
because the logrank test can then be justified by physical implementation.

Furthermore, we extend the theory to the stratified logrank test. 
Note that under superpopulation inference, 
the stratified logrank test is often preferred when we want to adjust covariates that 
may affect the event times. 
Because our finite-population justification of the logrank test allow arbitrary distribution of the event times, 
the original motivation for stratification seems unnecessary, at least in terms of the validity of the logrank test. 
However, as demonstrated later, such an extension is still useful, especially in randomized block designs and when the censoring mechanism varies across strata. 

The paper proceeds as follows. 
Section \ref{sec:framework} introduces the potential outcome framework, 
and briefly reviews the superpopulation and finite population inference. 
Section \ref{sec:review_logrank} reviews usual superpopulation inference for the logrank test, and 
Section \ref{sec:simu_logrank} 
compares it to finite population inference by simulation. 
Section \ref{sec:rand_dist_logrank} studies the exact randomization distribution of the logrank statistic, 
and Section \ref{sec:asym_dist} studies its large-sample approximation. 
Section \ref{sec:violation} 
conducts simulations to investigate violation of assumptions, and motivates the stratified logrank test.  
Section \ref{sec:strata} 
studies 
the stratified logrank test. 
Section \ref{sec:conclusion} concludes with a short discussion. 

\section{Framework, notation and assumptions}\label{sec:framework}

\subsection{Potential event/censoring time and treatment assignment}\label{sec:potential_outcome_treatment_assignment}

We consider an experiment on $n$ units with two treatment arms: an active treatment and a control. 
The outcome of interest is the time to a pre-specified event, e.g., occurrence of a certain disease or death. 
Under the potential outcome framework, 
for each unit $i$, 
let $T_i(1)$ and $T_i(0)$ denote the potential event times under treatment and control, respectively. 
Let $\bm{T}(1) = (T_1(1), \ldots, T_n(1))'$ and 
$\bm{T}(0) = (T_1(0), \ldots, T_n(0))'$ 
be the vectors of potential event times for all $n$ units under treatment and control, respectively.  
The treatment effects are then characterized by the comparison between potential outcomes $\bm{T}(1)$ and $\bm{T}(0)$. %

Under either treatment arm, we may not be able to observe the event time due to censoring, which itself may be related to the treatment. 
We introduce $C_i(1)$ and $C_i(0)$ to denote the potential censoring times for unit $i$ under treatment and control, respectively. 
Let $\bm{C}(1) = (C_1(1), \ldots, C_n(1))'$ and 
$\bm{C}(0) = (C_1(0), \ldots, C_n(0))'$ be the vectors of all units' potential censoring times under treatment and control, respectively. 
Both the $T_i(z)$'s and $C_i(z)$'s are nonnegative for $z=0,1$. 
If $C_i(z) = \infty$, then there is no censoring for unit $i$ under treatment arm $z$. 
We say a unit is at risk if its event has not happened and it has not been censored. For $z=0,1,$ $W_i(z) = \min\{T_i(z), C_i(z)\} $ is then the potential time at risk for unit $i$ under treatment arm $z$, and $\Delta_i(z) = \I\{T_{i}(z) \leq C_{i}(z)\}$ is the corresponding potential event indicator.  

Let $Z_i \in \{0,1\}$ be the treatment assignment for unit $i$, which equals 1 if unit $i$ receives the active treatment and 0 otherwise, 
and $\bm{Z} = (Z_1, \ldots, Z_n)'$ be the treatment assignment vector for all units. 
Then $T_i = Z_i \cdot T_i(1) + (1-Z_i)\cdot T_i(0)$ and $C_i = Z_i \cdot C_i(1) + (1-Z_i) \cdot C_i(0)$ are the realized event and censoring times for unit $i$.  
Note 
that these realized variables, $T_i$ and $C_i$, may not be observed in practice. %
Instead, 
for each unit $i$, 
we can observe only the minimum of the realized event and censoring times, or equivalently the time at risk, 
i.e., $W_i = \min \{T_i, C_i\}= Z_iW_i(1) + (1-Z_i) W_i(0)$, 
and the corresponding event indicator for whether the realized event happens before censoring,  i.e., 
$
\Delta_i  = \I(T_{i} \leq C_{i}) =  Z_i \Delta_i(1) + (1-Z_i)\Delta_i(0).
$

To facilitate the discussion and ease understanding, 
we list all the notation with a short explanation in Table \ref{tab:notation}, 
and introduce the Studies of Left Ventricular  Dysfunction \citep[SOLVD,][]{SOLVD1990} as an illustrative example. 
The SOLVD is an extensive research program that contains a double-blinded randomized trial conducted in 23 medical centers. 
The trial entrolled 2569 eligible patients, and randomly assigned about half of them to receive enalapril 
or placebo. 
The patients were followed up for 41.4 months on average. 
Suppose that the event of interest 
is death; see \citet{SOLVD1991} and \citet{Cai1999} for other events of interest. 
Then in our notation, $Z$ denotes the treatment indicator for whether the patient receives enalapril or placebo, 
$T(1)$ and $T(0)$ denote the potential survival times under treatment and control, 
and $C(1)$ and $C(0)$ denote the potential follow-up times under treatment and control.

\begin{table}[htb]
    \centering
    \caption{Notation and explanation.}
    \label{tab:notation}
    {\rev \resizebox{\columnwidth}{!}{%
    \begin{tabular}{ll}
        \toprule
        Notation & Explanation \\
        \midrule
        $Z$ & treatment assignment indicator 
        \\
        $T(z)$ & potential event time under treatment arm $z$ 
        \\
        $C(z)$ & potential censoring time under treatment arm $z$ 
        \\
        $W(z) = \min\{Y(z), C(z)\}$ & potential time at risk under treatment arm $z$
        \\
        $\Delta(z) = \I\{T_{i}(z) \leq C_{i}(z)\}$ & potential event indicator under treatment arm $z$ 
        \\
        $T, C, W, \Delta$ & realized event time, censoring time, time at risk and event indicator
        \\
        $t_1 < t_2 < \ldots < t_K$ & distinct values of the control potential event times $\{T_i(0): 1\le i \le n\}$
        \\
        $d_k = \sum_{i=1}^{n} \I (T_i(0) = t_k)$ & number of units potentially having events at time $t_k$ under control
        \\
        $n_k = \sum_{i=1}^{n} \I (T_i(0) \ge t_k)$ & number of units potentially having events no earlier than time $t_k$ under control
        \\
        $h_k = d_k/n_k$ & hazard at time $t_k$ for all units under control
        \\
        $F(\cdot), \lambda(\cdot), \Lambda(\cdot)$ & survival, hazard and integrated hazard functions for empirical distribution of $T_i(0)$
        \\
        \bottomrule
    \end{tabular}%
    }}
\end{table}

\subsection{Treatment assignment and censoring mechanisms}\label{sec:two_mechanism}
We study finite population inference 
in the presence of censoring, 
where 
both $T_i(1)$'s and $T_i(0)$'s are viewed as fixed constants.
This is equivalent to conducting inference conditioning on all the potential event times $\bm{T}(1)$ and $\bm{T}(0)$. 
Consequently, the inference relies crucially on the randomness of the treatment assignments as well as potential censoring times. 
In the following, we will write the conditioning on $(\bm{T}(1), \bm{T}(0))$ explicitly to emphasize that all the potential event times are viewed as fixed constants.  
On the contrary, 
in usual superpopulation inference, the units' potential event times $(T_i(1), T_i(0))$'s are assumed to be i.i.d.\ samples from some distribution, 
the treatment assignments are conditioned on and thus viewed as constants, 
and the inference relies crucially on the randomness of the potential event times. 
To ease understanding, 
we review and compare 
the superpopulation and finite population inferences 
for the classical two-sample $t$-test when there is no censoring; see Section \ref{sec:review_t_test}.

Under finite population inference, 
because 
the potential event times $\bm{T}(1)$ and $\bm{T}(0)$ have been conditioned on or equivalently viewed as constants, 
the randomness in the observed data $W_i$'s and $\Delta_i$'s comes solely from the random treatment assignments $Z_i$'s and the random censoring times $C_i(1)$'s and $C_i(0)$'s. %
Therefore, 
the distributions of the treatment assignments and the potential censoring times govern the data-generating process and are crucial for statistical inference. 
First, we consider the distribution of $\bm{Z}$, also called the treatment assignment mechanism.  
We focus on the Bernoulli randomized experiment 
with equal probability of receiving active treatment for all units. 
Let $\Bernoulli(p)$ denote a Bernoulli distribution that has probability $p$ at 1 and $1-p$ at 0, for $p\in (0,1)$.
\begin{assumption}\label{asmp:rbe}
	Conditional on all the potential event and censoring times, 
	the treatment assignments are i.i.d.\ across all units, i.e.,   
	$$
	Z_i \mid \bm{T}(1), \bm{T}(0), \bm{C}(1), \bm{C}(0) \ \overset{\text{i.i.d.}}{\sim} \  \Bernoulli(p_1), \quad (1\le i \le n)
	$$
	for some $p_1 \equiv 1- p_0 \in (0,1).$
\end{assumption}

Second, we consider the distribution of $(\bm{C}(1), \bm{C}(0))$, also called the censoring mechanism. 
We focus on the case of non-informative i.i.d.\ censoring. 
\begin{assumption}\label{asmp:noninformative}
	The potential censoring times 
	for all units 
    are independent of the potential event times for all units, 
	i.e., 
	$(\bm{C}(1), \bm{C}(0)) \ind (\bm{T}(1), \bm{T}(0)),$
	and 
	the %
	$n$ two-dimensional random vectors, 
	$(C_1(1), C_1(0))', (C_2(1), C_2(0))',$ $\ldots,$ $(C_n(1), C_n(0))'$,  
	are i.i.d.
\end{assumption}

Note 
that Assumption \ref{asmp:noninformative} does not require $C_i(1)$ and $C_i(0)$ for the same unit $i$ to be independent or identically distributed. Instead, $C_i(1)$ and $C_i(0)$ can have  different marginal distributions and arbitrary dependence. 
Under Assumption \ref{asmp:noninformative}, 
let $G_z(c) \equiv \pr(C_i(z)\ge c)$ 
for $z=0,1$. 
In the context of the SOLVD, Assumption \ref{asmp:noninformative} assumes that the treatment and control potential follow-up times are i.i.d.\ across all units and importantly that they are independent of the potential survival times. 
It can be violated when the censoring is informative about the survival time. 
For example, in the SOLVD, the blinded medication may be terminated when the patient remains symptomatic despite maximum therapy with such medications \citep{SOLVD1991}, under which the censoring might indicate a shorter survival time; see also \citet{leung1997censoring} for more examples and discussions. 

In our finite population inference, Assumptions \ref{asmp:rbe} and \ref{asmp:noninformative} are crucial and in some sense necessary for justifying the logrank test or identifying the treatment effects. 
Note that we relax the i.i.d.\ assumption on the potential event times and allow them to be arbitrary constants. 
If, for instance, Assumption \ref{asmp:noninformative} fails and the potential censoring times are not identically distributed for all units, 
then the censoring times can be related to the event times in which case identification of the treatment effect is lost without further assumptions \citep{Tsiatis1975}.  
We conduct a simulation study in Section \ref{sec:violation} to investigate the violation of Assumptions \ref{asmp:rbe} and  \ref{asmp:noninformative}, showing their necessity for the validity of the logrank test under finite population inference. 

Finally, we give two technical remarks regarding the assumptions. 
First, we can weaken Assumptions \ref{asmp:rbe} and \ref{asmp:noninformative}
by allowing dependence between the treatment assignment and the potential censoring times, but we still require 
$(Z_i, C_i(1),C_i(0))$'s to be i.i.d.\ and independent of the potential event times; see 
Appendix A1 of the the Supplementary Material
for details. 
Second, similar assumptions are also invoked in usual superpopulation inference, but they can have weaker forms; 
see 
Appendix A2 of the Supplementary Material
for a more detailed comparison.

\subsection{Fisher's null hypothesis of no treatment effect}\label{sec:null}

We are interested in testing whether the treatment has any effect on 
the units' event times. 
We consider Fisher's null hypothesis \citep{Fisher:1935}
of no 
treatment 
effect on any unit:
\begin{align}\label{eq:H0}
H_0: \bm{T}(1) = \bm{T}(0). %
\end{align}
In the absence of censoring, Fisher's null $H_0$ 
is sharp in the sense that all the potential event times are known from the observed data. However, this is generally not true in the presence of censoring. 
For example, for a unit $i$ under control with $Z_i=0$, $\Delta_i=0$ and $W_i \ge 0$, under $H_0$, we know that $T_i(1) = T_i(0) > C_i(0) = W_i$, 
but we do not know the exact values of $T_i(1)$, $C_i(1)$, $W_i(1)$ and $\Delta_i(1)$ under the alternative treatment.
The non-sharpness of $H_0$ implies that we cannot know the null distribution of the logrank statistic exactly, and 
the usual randomization or permutation tests %
may not be valid.
To conduct valid 
inference, 
it is crucial to understand the null distribution of the logrank statistic under $H_0$, which generally depends on 
the treatment assignment and censoring mechanisms, as well as all the potential event times. 
Here the null distribution of the logrank statistic refers to its conditional distribution given $\bm{T}(1)$ and $\bm{T}(0)$ under Fisher's null $H_0$.

\subsection{Superpopulation and finite population inferences}\label{sec:review_t_test}

We 
give a brief comparison between 
superpopulation and finite population inferences using the classical two-sample $t$-test as an illustration. 
Such a comparison may help the reader
understand the difference between these two types of inference, 
and make the purpose of the paper less obscure. 
We refer the reader to \citet{Imbens2020} 
for a more comprehensive comparison.

Suppose that there is no censoring and the realized event times $T_i$'s can always be observed. 
Let $\TST$ denote the two-sample $t$-statistic calculated from $(Z_i, T_i)$'s. 
Assume that the treatment assignments are completely randomized and Fisher's null $H_0$ in \eqref{eq:H0} holds. Below we discuss two types of justification for the asymptotic Gaussianity of $\TST$. 
First, under the superpopulation inference, the potential event times $T_i(1) = T_i(0)$'s are assumed to be i.i.d.\ samples from {\rev some} distribution. 
Under certain regularity conditions and conditional on the treatment assignments, 
the $t$-statistic is asymptotically standard Gaussian, in the sense that 
$\TST \mid \bm{Z} \converged \mathcal{N}(0,1).$
Second, under the finite population inference, all the potential event times are viewed as fixed constants or equivalently being conditioned on. 
Under certain regularity conditions and conditional on the potential event times, 
the $t$-statistic is asymptotically standard Gaussian \citep{Neyman:1923, hajek1960limiting, lidingclt2016}, in the sense that 
$\TST \mid \bm{T}(1), \bm{T}(0) \converged \mathcal{N}(0,1).$

Although leading to similar conclusions 
justifying the two-sample $t$-test, 
the superpopulation and finite population inferences are quite different. 
The main difference
comes from the source of randomness. 
The distribution of $\TST$ under the superpopulation setting is conditioning on the treatment assignment $\bm{Z}$, and the randomness comes solely from the random potential event times $\bm{T}(1)$ and $\bm{T}(0)$. 
However, the distribution of $\TST$ under finite population setting is conditioning on the potential event times $\bm{T}(1)$ and $\bm{T}(0)$, and the randomness comes solely from the random treatment assignment $\bm{Z}$. 
Moreover, 
for the asymptotics, 
these two types of inference impose different forms of regularity conditions. 
The superpopulation asymptotics 
imposes regularity conditions on the population distribution that generates the i.i.d.\ samples, 
while the finite population asymptotics imposes regularity conditions directly on the fixed potential event times. 
The regularity conditions for finite population inference does not, at least explicitly,  require the potential event times $T_i(0)$'s to be either identically distributed or independent, 
and can be weaker than that for superpopulation inference \citep{lidingclt2016}. 
Therefore, 
intuitively, 
finite population inference can give a broader justification for the two-sample $t$-test in randomized experiments. 
We also conduct a simulation study to confirm this intuition; 
see Appendix A3 of the Supplementary Material.

\section{Two-sample logrank test %
	under usual superpopulation inference}\label{sec:review_logrank}

As discussed 
in Section \ref{sec:review_t_test},  
randomization justifies the usual two-sample $t$-test, 
without %
any distributional assumption on the potential event times. 
However, 
the methods that are valid under the superpopulation inference cannot always be justified by  randomization under finite population inference. 
For example, 
\citet{freedman2008,freedman2008regression_a, freedman2008regression_b} showed that 
randomization may not justify linear or logistic regression.  
In this paper, we investigate whether randomization can justify the logrank test without requiring any distributional assumption on the potential event times. 
We first 
briefly review usual superpopulation inference for the logrank test.

For any %
time $t\ge 0$, 
let 
\begin{align*}
	&\overline{N}_1(t) = \sum_{i=1}^{n} Z_i \I (W_i\ge t),  \quad 
	\overline{N}_{0}(t) = \sum_{i=1}^{n} (1-Z_i) \I (W_i\ge t),\\
	& \overline{N}(t) = \overline{N}_1(t) + \overline{N}_{0}(t) = \sum_{i=1}^{n} \I (W_i\ge t)
\end{align*}
be the numbers of units at risk at time $t$ in treated, control and both groups, respectively. 
Let 
\begin{align*}
	& \overline{D}_1(t) = \sum_{i=1}^{n} Z_i \Delta_i \I (W_i= t), \quad 
	\overline{D}_0(t) = \sum_{i=1}^{n} (1-Z_i) \Delta_i \I (W_i= t),\\
	& \overline{D}(t) = \overline{D}_1(t) + \overline{D}_0(t) = \sum_{i=1}^{n} \Delta_i \I (W_i= t)
\end{align*}
be the numbers of units having events at time $t$ in treated, control and both groups, respectively. 
In the context of the SOLVD, $\overline{N}_1(t)$, $\overline{N}_0(t)$ and $\overline{N}(t)$ denote the numbers of units that are still alive and under follow-up at time $t$ in treated, control and both groups, 
and $\overline{D}_1(t)$, $\overline{D}_0(t)$ and $\overline{D}(t)$ denote the numbers of units that die at time $t$ in treated, control and both groups.
These quantities constitute the contingency table at time $t$, as shown in Table \ref{tab:contingency}. 

\begin{table}
    \centering
	\caption{The contingency table at time $t$} \label{tab:contingency}
	\resizebox{\columnwidth}{!}{
		\begin{tabular}{lccc}
			\toprule
			Group & Having event & Not having event & Total\\
			\midrule
			Treatment & $\overline{D}_1(t) = \sum_{i=1}^{n} Z_i \Delta_i \I(W_i = t)$ & $\overline{N}_1(t) - \overline{D}_1(t)$ & $\overline{N}_1(t) = \sum_{i=1}^{n} Z_i  \I(W_i \ge t)$\\
			Control & $\overline{D}_0(t) = \sum_{i=1}^{n} (1-Z_i) \Delta_i \I(W_i = t)$ & $\overline{N}_0(t) - \overline{D}_0(t)$ & $\overline{N}_0(t) = \sum_{i=1}^{n} (1-Z_i)  \I (W_i \ge t)$\\
			\midrule
			Total & $D(t) = \sum_{i=1}^{n} \Delta_i \I (W_i = t)$ & $N(t) - D(t)$ & $N(t) = \sum_{i=1}^{n} \I (W_i \ge t)$\\
			\bottomrule
		\end{tabular}
	}
\end{table}

Under Fisher's null $H_0$ in \eqref{eq:H0} and usual superpopulation formulation, 
the units' potential event times $T_i(1) = T_i(0)$'s are i.i.d.\ samples from a superpopulation. 
Suppose the treatment assignment $\bm{Z}$, the potential even times $(\bm{T}(1), \bm{T}(0))$, and the potential censoring times $(\bm{C}(1), \bm{C}(0))$ are mutually independent. 
Then conditional on the treatment assignment $\bm{Z}$, 
the realized event times $(T_i, C_i)$'s follow the usual random censorship model \citep{Harrington1991}. 
For any $t\ge 0$,  \citet{mantel1966} observed that, 
conditional on the margins of Table \ref{tab:contingency}, 
the number of observed events in treated group %
follows a hypergeometric distribution:
\begin{align}\label{eq:hypergeo_super}
\overline{D}_1(t) \mid \bm{Z}, \overline{N}_1(t), \overline{N}_0(t), \overline{D}(t)
\ \sim \ 
\Hypergeometric(\overline{N}(t), \overline{D}(t), \overline{N}_1(t) ),
\end{align}
with mean and variance  
\begin{align}\label{eq:mean_var_conti}
\overline{\ED}(t) = 
\frac{\overline{D}(t)\overline{N}_1(t)}{\overline{N}(t)}  ,
\quad
\overline{V}(t) = 
\frac{
	\overline{D}(t)  \{\overline{N}(t)-\overline{D}(t)\}  \overline{N}_1(t)  \overline{N}_0(t)
}{
	\{\overline{N}(t)\}^2 \{\overline{N}(t)-1\}
}, 
\end{align}
{\rev where we use $\Hypergeometric(m, k, b)$ to denote the hypergeometric distribution that describes the probability of $x$ successes in $b$ draws, without replacement, from a finite population of size $m$ with exactly $k$ successes,  
for any integers $m, k$ and $b$ with $0\le k,b\le m$.} 
The denominators in \eqref{eq:mean_var_conti} can be zero when $\overline{N}(t)$ equals 0 or 1, making $\overline{\ED}(t)$ or $\overline{V}(t)$ undefined. For descriptive convenience, we define $0/0$ as 0. This is reasonable because when the denominator of $\overline{V}(t)$ is zero, the numerator must be zero and $\overline{D}_1(t)$ is a constant equal to $\overline{M}(t)$ with zero variance. 
Below we give some intuition for the hypergeometric distribution in \eqref{eq:hypergeo_super}. 
For any unit at risk at time $t$, 
no matter whether it received treatment or control (i.e., in Row 1 or 2 of Table \ref{tab:contingency}), 
it has the same probability to have or not to have the event (i.e., in Column 1 or 2 of Table \ref{tab:contingency}). 
This is because the probability of immediately having an event (i.e., hazard) for any unit at risk at time $t$ under either treatment arm is the same, a fact implied by 
the i.i.d.\ assumptions on the potential event times and their independence of 
the treatment assignments and potential censoring times. 
Similar to Fisher's exact test, %
conditioning on all the margins of Table \ref{tab:contingency}, 
the number of treated units having events at time $t$ follows a hypergeometric distribution as in \eqref{eq:hypergeo_super}. 

The logrank statistic is then defined as the summation of the differences between the observed and expected event counts in treated group over all time \citep{mantel1966}, i.e., 
$\sum_t\{ \overline{D}_1(t) - \overline{\ED}(t) \}$. 
To conduct the logrank test, we will first estimate its standard deviation and then compare the logrank statistic standardized by its estimated standard deviation to standard Gaussian quantiles. 
Here we consider the commonly used variance estimator $\sum_{t} \overline{V}(t)$; see, e.g., \citet{Brown1984} for other choice of variance estimator. 
The standardized logrank statistic is then 
\begin{align}\label{eq:logrankstat}
\LR  =\left\{\sum_t \overline{V}(t) \right\}^{-1/2} \sum_t \left\{ \overline{D}_1(t) - \overline{\ED}(t) \right\}. 
\end{align}
For convenience, in the following we will simply call the standardized logrank statistic in \eqref{eq:logrankstat} as the logrank statistic.
Because  $\overline{D}_1(t) - \overline{\ED}(t)$ and $\overline{V}(t) $ are nonzero only if $\overline{D}(t)>0$, 
the summations 
in \eqref{eq:logrankstat} only need to be taken over the observed event times. 
When the potential event times $T_i(1)=T_i(0)$'s are i.i.d.\ from a superpopulation, 
from the classical theory for the logrank test, 
$\LR$ in  \eqref{eq:logrankstat} is asymptotically standard Gaussian under some regularity conditions:  
\begin{align}\label{eq:logrank_clt_super}
	\LR  \mid \bm{Z} \converged \mathcal{N}(0,1);
\end{align} 
see 
Appendix A2 of the Supplementary Material 
for more detailed discussion, and \citet{Aalen1978, Gill1980} and \citet{Harrington1991} for rigorous justification using  
the counting process formulation and martingale properties. %
The distribution \eqref{eq:logrank_clt_super}  of $\LR$  
is 
conditioning  
on the treatment assignments, and relies crucially on the random potential event times. 
In simple words, the treatment assignments are fixed, and the potential event times are random.

\section{A simulation study of the logrank test 
under both superpopulation and finite population inferences}\label{sec:simu_logrank}

We conduct a simulation study for the logrank test from both the superpopulation and finite population perspectives, with possibly non-identically distributed and dependent potential event times. 
Specifically, we will show that, with non-identically distributed or correlated potential event times $T_i(1) = T_i(0)$'s across units, the conditional distribution of the logrank statistic given treatment assignments from the superpopulation perspective can be quite different from the standard Gaussian distribution, while the conditional distribution given the potential event times from the finite population perspective can be close to the standard Gaussian distribution. 
This implies that the finite population inference can give a broader justification for the classical logrank test. 

\begin{table}[htb]
    \centering
	\caption{Parameter values for different super population settings}\label{tab:simulation}
	\begin{tabular}{ccccc}
		\toprule
		Case & $\rho$ & $\theta$  & Dependence & Heterogeneity\\
		\midrule
		1 & $0$ & $0$ & No & No\\
		2 & $0.5$ & $0$ & Yes & No\\
		3 & $0$ & $1$ & No & Yes\\
		4 & $0.5$ & $1$ & Yes & Yes\\
		\bottomrule
	\end{tabular}
\end{table}

For any $\rho\in (-1,1)$, let $\bm{\Sigma}_{\rho} \in \mathbb{R}^{n\times n}$ be a correlation matrix whose $(i,j)$th
element is $\rho^{|i-j|}$. 
A random vector $(\eta_1, \eta_2, \ldots, \eta_n)' \in \mathbb{R}^n$ is said to follow a Gaussian copula %
with correlation matrix $\bm{\Sigma}_{\rho}$ if $(\Phi^{-1}(\eta_1), \ldots,$ $\Phi^{-1}(\eta_n))' \sim \mathcal{N}(0, \bm{\Sigma}_{\rho})$, 
where $\Phi^{-1}(\cdot)$ denotes the quantile function of the standard Gaussian distribution. 
We can verify 
that the marginal distributions of the $\eta_i$'s are all uniform on $(0,1)$, 
and thus the $-\log(1-\eta_i)$'s all follow the exponential distribution with rate parameter 1, denoted by $\exponential(1)$. 
For $1\le i \le n$, we introduce a binary $X_i$ to denote some pretreatment covariate of unit $i$. 
We  consider the following model for the potential event times: 
\begin{align}\label{eq:model_potential_outcome}
& T_i(1) = T_i(0) =  - \log(1-\eta_i) \cdot \left(1 + \theta X_i \right), \quad (1\le i \le n)
\\
& (\eta_1, \eta_2, \ldots, \eta_n) \sim \text{Gaussian copula with correlation matrix } \bm{\Sigma}_{\rho},
\nonumber
\\
& \bm{X} \equiv (X_1, \ldots, X_n) = (\underbrace{0, 0, \ldots, 0}_{0.2\times n}, 
\underbrace{1, 1, \ldots, 1}_{0.3\times n},
\underbrace{0, 0, \ldots, 0}_{0.3\times n}, 
\underbrace{1, 1, \ldots, 1}_{0.2\times n}
).
\nonumber
\end{align}
In model \eqref{eq:model_potential_outcome},
the potential event times marginally follow scaled exponential distributions, with the scales depending on the parameter $\theta$ and the
covariates $X_i$'s,
and their dependence structure is determined by the parameter $\rho$. 
The 
covariates $X_i$'s 
in \eqref{eq:model_potential_outcome}
are fixed constants, or equivalently we can view \eqref{eq:model_potential_outcome} as a model conditional on the $X_i$'s. 
We consider four choices of the parameters $(\theta, \rho)$ as in Table \ref{tab:simulation}, which determine the existence of heterogeneity and dependence among the units' potential event times. 
For the potential censoring times, we generate them as i.i.d.\ samples from the following model: 
\begin{align}\label{eq:model_censor}
C_i(1) \sim 2 \cdot \exponential(1), \ 
C_i(0) \sim \exponential(1), \ 
C_i(1) \ind C_i(0), \quad 
(1\le i \le n)
\end{align}
where the censoring time tends to be larger under treatment than under control. 
For the treatment assignments, we consider the following generating model: 
\begin{align}\label{eq:model_treatment}
Z_1, Z_2, \ldots, Z_n \  \overset{\text{i.i.d.}}{\sim} \  \Bernoulli(0.5). 
\end{align}
Furthermore, 
the potential event times $\bm{T}(1) = \bm{T}(0)$ in \eqref{eq:model_potential_outcome}, 
the {\rev potential} censoring times $(\bm{C}(1), \bm{C}(0))$ in \eqref{eq:model_censor}, 
and the treatment assignment $\bm{Z}$ in  \eqref{eq:model_treatment} are mutually independent.  
Throughout %
the paper,
under each simulation scenario, 
the sample size $n$ equals $1000$, and the number of simulated datasets equals $10^4$. 

\begin{figure}[htbp]
	\centering
	\begin{subfigure}{.4\textwidth}
		\centering
		\includegraphics[width=1\linewidth]{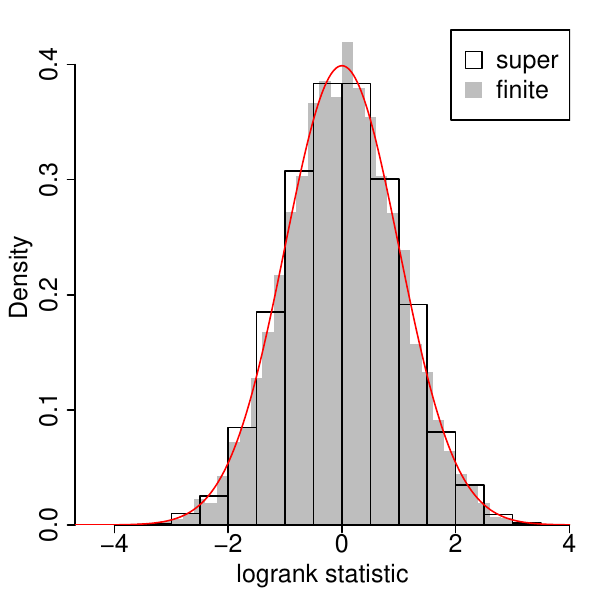}
		\caption*{case 1}
	\end{subfigure}%
	\begin{subfigure}{.4\textwidth}
		\centering
		\includegraphics[width=1\linewidth]{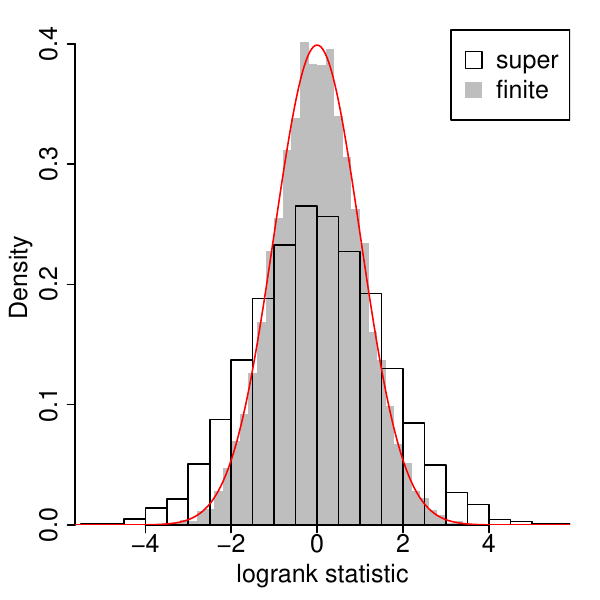}
		\caption*{case 2}
	\end{subfigure}
	\begin{subfigure}{.4\textwidth}
		\centering
		\includegraphics[width=1\linewidth]{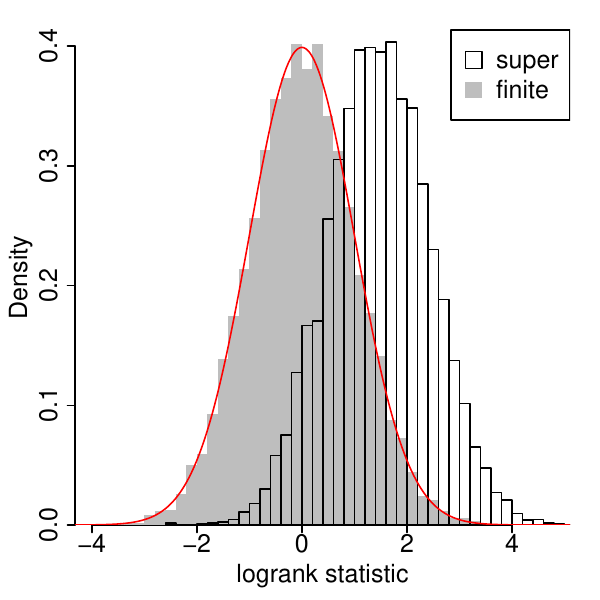}
		\caption*{case 3}
	\end{subfigure}%
	\begin{subfigure}{.4\textwidth}
		\centering
		\includegraphics[width=1\linewidth]{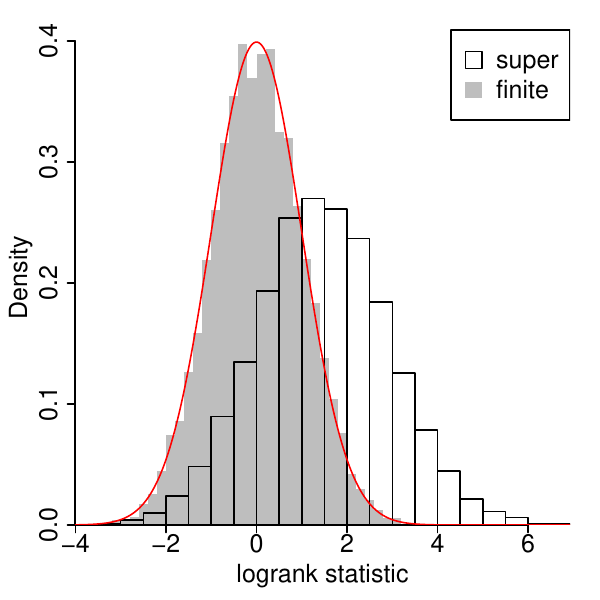}
		\caption*{case 4}
	\end{subfigure}
	\caption{
		Histograms of the logrank statistic in the superpopulation and finite population settings with parameter values in the four cases of 
		Table \ref{tab:simulation}. 
		The white histograms denote the distributions of the logrank statistic under the superpopulation setting with random potential event times and fixed treatment assignments, 
		and the gray histograms denote the distributions of the logrank statistic under finite population setting with fixed potential event times and random treatment assignments. 
		The red lines denote the density of a standard Gaussian distribution.
	}\label{fig:logrank}
\end{figure}

First, we mimic the superpopulation inference conditioning on the treatment assignment. 
We fix the treatment assignment $\bm{Z}$ as the following constant vector: 
\begin{align}\label{eq:fix_treatment}
\bm{Z} = (\underbrace{1, 1, \ldots, 1}_{n/2}, \underbrace{0, 0, \ldots, 0}_{n/2})',
\end{align}
and generate random potential event times $(\bm{T}(1),\bm{T}(0))$ and random potential censoring times $(\bm{C}(1),\bm{C}(0))$ from  models \eqref{eq:model_potential_outcome} and \eqref{eq:model_censor}, respectively. 
The resulting distribution of the logrank statistic is then $\LR$ conditional on the treatment assignment $\bm{Z}$ in \eqref{eq:fix_treatment}. In Figure \ref{fig:logrank}, the four 
white histograms
show the distributions of the logrank statistic under the four cases in Table \ref{tab:simulation}. 
When there exist dependence or heterogeneity among units' potential event times, the conditional distribution of $\LR$  given $\bm{Z}$ 
in \eqref{eq:fix_treatment} is no longer close to a standard Gaussian distribution.

Second, we mimic finite population inference conditioning on the potential event times. 
We fix the potential event times as one realization from model \eqref{eq:model_potential_outcome}:
\begin{align}\label{eq:fix_potential_event_time}
\bm{T}(1) = \bm{T}(0) = \text{a realization from model \eqref{eq:model_potential_outcome}},
\end{align}
and generate random potential censoring times $(\bm{C}(1),\bm{C}(0))$ and  random treatment assignment $\bm{Z}$ 
from \eqref{eq:model_censor} and \eqref{eq:model_treatment}, respectively. 
The resulting distribution of the logrank statistic is then $\LR$ conditional on $\bm{T}(1)=\bm{T}(0)$ in \eqref{eq:fix_potential_event_time}. 
In Figure \ref{fig:logrank}, the four gray histograms show the distributions of the logrank statistic under the four cases in Table \ref{tab:simulation} or equivalently four different realizations of the potential event times.
From Figure \ref{fig:logrank}, %
they 
are all close to a standard Gaussian distribution.

Third, we consider the scenario where all potential event times, potential censoring times and treatment assignments are randomly generated from \eqref{eq:model_potential_outcome}, \eqref{eq:model_censor} and \eqref{eq:model_treatment}, respectively. 
The resulting distribution of the logrank statistic can be viewed as the conditional distribution of $\LR$ given $\bm{T}(1)$ and $\bm{T}(0)$ marginalizing over all potential event times. 
From the finite population simulation, we expect the distribution of $\LR$ to be close to a standard Gaussian distribution. 
This is confirmed by Figure A3 in the Supplementary Material. 

From the simulation, randomization justifies the logrank test, 
without requiring distributional assumptions (such as i.i.d.) on the potential event times. 
In the remainder of the paper, we will study explicitly the conditional distribution of $\LR$ given all the potential event times, and prove its asymptotic Gaussianity under certain regularity conditions.

\section{Randomization distribution of the logrank statistic}\label{sec:rand_dist_logrank}

In this section, we study the randomization distribution of the logrank statistic 
under the finite population inference, where the potential event times are being conditioned on or equivalently viewed as fixed constants,  
i.e., 
$\LR \mid \bm{T}(1), \bm{T}(0)$. 
Below we first introduce some finite population quantities depending only on the potential event times, which are also summarized in Table \ref{tab:notation}. 
For the potential event times under control, 
we use $t_1 < \ldots < t_K$ to denote distinct values in $\{T_i(0): 1\leq i\le n\}$, where $K$ is the total number of distinct event times. 
Let 
$d_k = \sum_{i=1}^{n} \I (T_i(0) = t_k)$ be the number of units potentially having events at time $t_k$, 
$n_k = \sum_{i=1}^{n} \I (T_i(0) \ge t_k)$ be the number of units potentially having events no earlier than time $t_k$, 
and $h_k = d_k/n_k$ be the hazard at time $t_k$.
We introduce $F(\cdot)$ to denote the empirical survival distribution for the control potential event times, i.e., 
$F(t) = n^{-1} \sum_{i=1}^n \I\{T_i(0)\ge t\}$. 
Then, by definition, 
$F(t) = 1 = n_1/n$ when $t\le t_1$, 
$F(t) = n_k/n$ when $t_{k-1}<t\le t_k$, for $2\le k \le K$, and $F(t) = 0$ when $t>t_K$. 
Consequently, 
the hazard function for the distribution $F(\cdot)$ is 
$\lambda(t)$, which takes value $h_k$ when $t =t_k$ for $1\le k \le K$ and zero otherwise, 
and the corresponding integrated hazard function is 
$
\Lambda(t) = \sum_{k:t_k \le t} h_k. 
$

\subsection{Contingency tables under randomization}\label{sec:conti_table_finite}

We first investigate the contingency table at time $t$ (i.e., Table \ref{tab:contingency}) under finite population inference.  
The treatment assignment $\bm{Z}$, instead of being conditioned on as in Section \ref{sec:review_logrank}, plays an important role in inducing randomness of the quantities in Table \ref{tab:contingency}. 
As demonstrated in Appendix A4 of the Supplementary Material, for any unit at risk at time $t$, no matter whether it had an event or not (i.e., in Column 1 or 2 of Table \ref{tab:contingency}), it has the same probability of receiving treatment or control (i.e., in Row 1 or 2 of Table \ref{tab:contingency}). 
The mutual independence among $(Z_i, W_i, \Delta_i)$'s under Assumptions \ref{asmp:rbe} and \ref{asmp:noninformative} then implies that, conditional on the margins of Table \ref{tab:contingency}, $\overline{D}_1(t)$ follows a hypergeometric distribution. 

\begin{theorem}\label{thm:hyper_randomization}
	Under Assumptions \ref{asmp:rbe} and \ref{asmp:noninformative} and the null hypothesis $H_0$ in \eqref{eq:H0}, 
	and conditioning on all the potential event times $\bm{T}(1)$ and $\bm{T}(0)$, 
	\begin{align}\label{eq:hyper_randomization}
	\overline{D}_1(t) \mid \bm{T}(1), \bm{T}(0), \overline{N}(t), \overline{D}(t), \overline{N}_1(t)  
	\ \sim \   \Hypergeometric(\overline{N}(t), \overline{D}(t), \overline{N}_1(t)). 
	\end{align}
\end{theorem}

In the superpopulation setting in Section  \ref{sec:review_logrank}, the hypergeometric distribution \eqref{eq:hypergeo_super} of the same form is justified by the i.i.d.\ potential event times with fixed treatment assignments. 
However, \eqref{eq:hyper_randomization} in Theorem \ref{thm:hyper_randomization}  is justified by the random treatment assignments with fixed potential event times. 
In particular, Theorem \ref{thm:hyper_randomization} does not impose any distributional assumption on the potential event times, and thus  
allows 
them 
to have unit-specific distributions and arbitrary dependence structure. 
Moreover, comparing the justification for hypergeometric distributions in \eqref{eq:hypergeo_super} and \eqref{eq:hyper_randomization}, 
the former relies on the fact that units in either row of Table \ref{tab:contingency} have the same probability to fall in either column, 
while the latter relies on the fact that units in either column of Table \ref{tab:contingency} have the same probability to fall in either row. 
In simple words, one takes a row-wise perspective, while the other takes a column-wise perspective.

\subsection{A martingale difference sequence representation}\label{sec:mds}

Recall that $\{t_1, \ldots, t_K \}$ are the distinct values of all control potential event times. 
Under Fisher's null $H_0$ in \eqref{eq:H0}, 
the observed event times must be in the set $\{t_1, \ldots, t_K\}$. 
Thus, the summations in \eqref{eq:logrankstat} only need to be taken over $t_1, t_2,  \ldots, t_K$. 
To facilitate the discussion, 
at each time $t_k$, 
we introduce 
\begin{align*}
	N_{1k} = \overline{N}_1(t_k), 
	\quad 
	N_{0k} = \overline{N}_0(t_k), 
	\quad
	N_k = \overline{N}(t_k), 
	\qquad (1\le k \le K)
\end{align*}
to denote the numbers of units at risk in treated, control and both groups, 
and 
\begin{align*}
	D_{1k} = \overline{D}_1(t_k), 
	\quad
	D_{0k} = \overline{D}_0(t_k), 
	\quad
	D_k = \overline{D}(t_k), 
	\qquad (1\le k \le K)
\end{align*}
to denote the numbers of units having events in treated, control and both groups. 
Analogous to \eqref{eq:mean_var_conti}, 
for $1\le k \le K$, 
we further define 
\begin{align*}
	M_k = \overline{\ED}(t_k) = 
	\frac{D_k N_{1k}}{N_k},
	\quad 
	V_k = \overline{V}(t_k) = 
	\frac{
		D_k (N_k-D_k) N_{1k} N_{0k}
	}{
		N_k^2 (N_k-1)
	}. 
\end{align*}
The overall observed-expected difference and the corresponding overall variance are, respectively, 
\begin{align}\label{eq:L_U}
	L \equiv \sum_{k=1}^{K} (D_{1k} - \ED_{k})
	 \ \ \ \text{and} \ \ \  U \equiv \sum_{k=1}^{K} V_{k}. 
\end{align}
We can then rewrite the logrank statistic in \eqref{eq:logrankstat} as
\begin{align}\label{eq:logrank_randomization}
\LR = 
\left( \sum_{k=1}^K V_{k} \right) ^{-1/2}\sum_{k=1}^K (D_{1k} - \ED_{k})
=  
U^{-1/2} L. 
\end{align}

The difficulty in studying the distribution of \eqref{eq:logrank_randomization} comes from the dependence among all the contingency tables at times $t_1, \ldots, t_K$. 
Fortunately, 
as demonstrated shortly, 
the sequence of the observed-expected differences, $\{D_{1k} - M_{k}: k=1,\ldots, K\}$, is actually a martingale difference sequence. 
Define $t_0 = -\infty$ and $t_{K+1} = \infty$ for descriptive convenience. 
For $0\le k \le K$, define a $\sigma$-algebra $\mathcal{F}_k$ as 
\begin{align*}%
\mathcal{F}_{k} & = \sigma\left\{
\I(W_i \geq t_q), 
\Delta_i\I(W_i =t_q), 
Z_i \Delta_i \I(W_i = t_{q-1}), 
\vphantom{\sum_{j=1}^{n} Z_j \I(W_j \geq t_q)}
\right.\\
& \quad \ 
\left. 
\qquad
\qquad
\qquad
\qquad
\sum_{j=1}^{n} Z_j \I(W_j \geq t_q), 
\quad
1\leq q \leq k+1, \ \ 1\leq i\leq n
\right\}. 
\end{align*}
Intuitively, 
$\mathcal{F}_{k}$ contains the information of whether the units had events or were at risk up to time $t_{k+1}$, the treatment indicators for units having events up to time $t_k$, and the numbers of treated units at risk up to time $t_{k+1}$. 
We can verify that $\mathcal{F}_k$ contains the contingency tables up to time $t_k$, and margins of the contingency table at time $t_{k+1}$. 
Consequently, $D_{1k}$ is $\mathcal{F}_k$-measurable, and $(N_k, D_k, N_{1k})$ are $\mathcal{F}_{k-1}$-measurable. 
Moreover, the information in  $\mathcal{F}_{k-1}$ other than $(N_k, D_k, N_{1k})$ does not change the conditional distribution of $D_{1k}$. 

\begin{theorem}\label{thm:mds}
	Under Assumptions \ref{asmp:rbe} and \ref{asmp:noninformative} and the null hypothesis $H_0$ in \eqref{eq:H0}, 
	and conditioning on all the potential event times $\bm{T}(1)$ and $\bm{T}(0)$, 
	\begin{itemize}
		\item[(i)] 
		$\{D_{1k}-M_{k}\}_{k=1}^K$ is a martingale difference sequence with respect to the filtration $\mathcal{F}_k$'s;
		\item[(ii)] 
		conditional on $\mathcal{F}_{k-1}$, $D_{1k}$ follows a hypergeometric distribution with parameters $(N_{k},  D_k, N_{1k})$, i.e., 
		\begin{align*}
		D_{1k} \mid \bm{T}(1), \bm{T}(0), \mathcal{F}_{k-1}
		\ \sim \   \Hypergeometric(N_{k},  D_k, N_{1k}). 
		\end{align*}
	\end{itemize}
\end{theorem}

Theorem \ref{thm:mds} implies that $L$ in \eqref{eq:L_U} is a summation of a martingale difference sequence. 
This is similar to usual superpopulation inference, where $L$ is related to integrals with respect to martingales. 
Under both finite population and superpopulation inferences, the martingale property plays a critical role in the large-sample asymptotic analyses. 
However, they are driven by different sources of randomness: the martingale property under finite population inference relies on the random treatment assignments, while that under the superpopulation inference relies on the random potential event times.
Martingale properties driven by random treatment assignments have also been utilized in settings with sequentially assigned treatments over time; see, e.g., recent work of \citet{Bojinov2019} and \citet{papadogeorgou2020causal}.

Below we investigate the first two moments of the randomization distribution of the observed-expected difference $L$ in \eqref{eq:L_U}. 
Let $G(t) = p_1 G_1(t) + p_0 G_0(t)$ be the survival function of 
the realized censoring time 
$C_i$, $g_k = G(t_k)$ be the probability that 
the realized censoring time 
$C_i$ 
is no less than $t_k$,
and 
$\phi_k = p_1 G_1(t_k)/G(t_k)$ be the conditional probability of receiving active treatment given 
that the realized censoring time is no less than $t_k$.

\begin{theorem}\label{thm:mean_var_L}
	Under Assumptions \ref{asmp:rbe} and \ref{asmp:noninformative} and the null hypothesis $H_0$ in \eqref{eq:H0}, 
	and conditioning on all the potential event times $\bm{T}(1)$ and $\bm{T}(0)$,  
	the observed-expected difference 
	$L$ in \eqref{eq:L_U} has mean 0, i.e., 
	$\E\{ L \mid \bm{T}(1), \bm{T}(0)\} = 0$, 
	and it has variance 
	\begin{align}\label{eq:var_L_finite}
	& \quad \ \Var\left\{ L \mid \bm{T}(1), \bm{T}(0) \right\} \\
	& 
	= 
	\E\left\{ U \mid \bm{T}(1), \bm{T}(0)  \right\} 
	= \sum_{k=1}^{K} \E\left\{ V_{k} \mid \bm{T}(0), \bm{T}(0) \right\}
	\nonumber
	\\
	& 
	= \sum_{k=1}^{K}
	\frac{n_k^2}{n_k-1}  h_k (1-h_k) \phi_k (1-\phi_k) 
	\left\{
	g_k - \frac{1 - (1-g_k)^{n_k}}{n_k} 
	\right\}.
	\nonumber
	\end{align}
\end{theorem}

From Theorem \ref{thm:mean_var_L}, 
conditional on the potential event times, 
$U$ in \eqref{eq:L_U} is an unbiased variance estimator for 
$L$, 
and thus 
$\LR$ in \eqref{eq:logrank_randomization} is indeed a standardization of $L$. 
This is similar to the superpopulation setting that conditions on the treatment assignments. 
Below we further make a connection between the variance formula of $L$ in \eqref{eq:var_L_finite} and its asymptotic variance formula under the superpopulation inference. 
Recall that $F(\cdot)$ is the empirical survival distribution for the control potential event times $\{T_i(0): 1\le i \le n \}$, 
and $\Lambda(\cdot)$ is the corresponding integrated hazard function. 
Ignoring the terms of order $n_k^{-1}$, 
we can simplify the variance formula in \eqref{eq:var_L_finite} as: 
\begin{align*}
	& \quad \ \Var\left\{ L \mid \bm{T}(1), \bm{T}(0) \right\} \\
	& = 
	\sum_{k=1}^{K} n_k 
	\left( 1 + \frac{1}{n_k-1} \right)  h_k (1-h_k) \phi_k (1-\phi_k) 
	\left\{
	g_k - \frac{1 - (1-g_k)^{n_k}}{n_k} 
	\right\}\\
	& \approx 
	\sum_{k=1}^{K} n_k h_k (1-h_k) \phi_k (1-\phi_k) 
	g_k \\
	& = 
	np_1 p_0 \int_{0}^\infty F(t)    \frac{G_1(t)G_0(t)}{p_1 G_1(t) + p_0G_0(t)}  \left\{ 1- \Delta \Lambda(t) \right\} \text{d} \Lambda(t),
\end{align*} 
which has the same form as the asymptotic variance of $L$ under usual superpopulation inference, except that the distribution $F$ is replaced by the population distribution of the potential event times. %
Therefore, we can view the variance formula \eqref{eq:var_L_finite} as a finite population analogue of the 
superpopulation
variance formula. 
Importantly, 
the former is for the variance of $\LR$ given the potential event times, while the latter is for the variance of $\LR$ given the treatment assignments.

\section{Asymptotic randomization distribution of the logrank statistic}\label{sec:asym_dist}

The exact randomization distribution of the logrank statistic depends on all the potential event times and is thus generally intractable. 
In this section we study its large-sample approximation using only the first two moments. %
The finite population asymptotics embeds the units into a sequence of finite populations with increasing sizes, and studies the limiting distributions of  certain statistics along this sequence of finite populations \citep{lidingclt2016}. 
Here we consider two cases 
depending on the total number of distinct potential event times.  
In the first case, the number of distinct potential event times goes to infinity as the sample size of the finite population goes to infinity; 
in the second case, the number of distinct potential event times is bounded. 
The first case is more reasonable when the event times are continuous, whereas the second case is more reasonable when the event times are discrete with a finite support. 

\subsection{
A diverging number of 
distinct potential event times} 

We first consider the case in which the number of distinct potential event times $K$ goes to infinity as the sample size $n$ increases. 
Let 
\begin{align}\label{eq:d_tilde}
	\tilde{d} = \max_{1\le k\le K} d_k = \max_{1\le k\le K} 
	\left|
	\left\{i: T_i(0) = t_k \right\}
	\right|
\end{align}
be the maximum number of units potentially having events at the same time. 
We further introduce a random vector $(\tilde{C}_i(1), \tilde{C}_i(0))$ that shares the same marginal distributions as $(C_i(1), C_i(0))$ but has two independent components, 
i.e., 
$\tilde{C}_i(z) \sim C_i(z)$ for $z=0,1$ and $\tilde{C}_i(1) \ind \tilde{C}_i(0)$. 
The survival function of $\min\{\tilde{C}_i(1), \tilde{C}_i(0)\}$ 
is then $\tilde{G}(t) = G_1(t) G_0(t)$.
Under Fisher's null $H_0$ in \eqref{eq:H0}, with the pseudo potential censoring times $\tilde{C}_i(1)$ and  $\tilde{C}_i(0)$, we can observe unit $i$'s event no matter {\rev whether} it is assigned to treatment or control if and only if $\min\{\tilde{C}_i(1), \tilde{C}_i(0)\}\ge T_i(1) = T_i(0)$. 
Thus, the expected proportion of units whose events can be observed under both treatment and control with the pseudo potential censoring times %
is 
\begin{align}\label{eq:g_tilde}
\tilde{g} & \equiv 
\E
\left[
\frac{1}{n}
\sum_{i=1}^{n} \I\left\{
\min\{\tilde{C}_i(1), \tilde{C}_i(0)\} \ge T_i(0)
\right\}
\right]
= \frac{1}{n} \sum_{i=1}^{n} \tilde{G}(T_i(0)) 
\\
& 
= \frac{1}{n} \sum_{k=1}^{K} d_k \tilde{G}(t_k). 
\nonumber
\end{align}
Here $\tilde{g}$ is defined in terms of the pseudo potential censoring times and 
may be 
different from that defined in terms of the original potential censoring times. 
Nevertheless, 
since we do not specify the distribution of $(C_i(1), C_i(0))$ and allow arbitrary dependence between them, 
$\tilde{g}$ is a quantity that can be realized when 
$C_i(1)$ and $C_i(0)$ are indeed independent. 
Note that both $\tilde{d}$ and $\tilde{g}$ in \eqref{eq:d_tilde} and \eqref{eq:g_tilde}, as well as $p_z$ and $G_z(\cdot)$ for $z=0,1$, depend on the finite population of size $n$.
For descriptive convenience, we make such dependence implicit. 
As shown in the following condition, 
the limiting behavior of $\tilde{d}$ and $\tilde{g}$ plays an important role in 
the asymptotic Gaussian approximation 
of the logrank statistic. 

\begin{condition}\label{cond:fp_conti}
	Along the sequence of finite populations satisfying $H_0$ in \eqref{eq:H0}, as $n\rightarrow \infty$, for $z=0,1$, 
	\begin{itemize}
		\item[(i)] the probability of receiving treatment arm $z$, $p_z$, has a positive limit;
		\item[(ii)] 
		the maximum number of potentially simultaneous events $\tilde{d}$ in \eqref{eq:d_tilde} and the expected proportion of observed events 
		$\tilde{g}$ in \eqref{eq:g_tilde} satisfy
		\begin{align}\label{eq:cond_conti}
		\frac{\tilde{g} }{\tilde{d}} 
		\min\left\{
		\left(
		\frac{n}{\log n}
		\right)^{1/2}, 
		\frac{n}{\tilde{d}^2}
		\right\}
		= 
		\min\left\{
		\frac{\tilde{g} }{\tilde{d}} 
		\left(
		\frac{n}{\log n}
		\right)^{1/2},  
		\frac{\tilde{g}n}{\tilde{d}^3}
		\right\}
		\converge \infty. 
		\end{align}
	\end{itemize}
\end{condition}

Under finite population inference, the potential event times are being conditioned on or equivalently viewed as fixed constants. 
Consequently, Condition \ref{cond:fp_conti} imposes conditions on all the potential event times directly. 
This is different from regularity conditions under superpopulation inference, which are usually on the population distribution that generates the i.i.d.\ units. 
Specifically, 
Condition \ref{cond:fp_conti} involves $p_z$ for the treatment assignment mechanism, 
$\tilde{d}$ in \eqref{eq:d_tilde} that is a deterministic function of the potential event times, 
and 
$\tilde{g}$ in \eqref{eq:g_tilde} that is determined by both the potential event times and the censoring mechanism. 
These quantities are fixed and can vary with sample size $n$, 
and Condition \ref{cond:fp_conti} imposes some constraints on their limiting behavior as $n$ goes to infinity. 
Note that
$\tilde{g}\le 1$ and $\tilde{d}\ge 1$. 
Condition \ref{cond:fp_conti} implies that $K \ge n/\tilde{d} \rightarrow \infty$ as $n\rightarrow \infty$, i.e., 
the number of distinct potential event times goes to infinity as $n\rightarrow \infty$. 

Below we give some intuition for Condition \ref{cond:fp_conti}, showing that the regularity conditions there are rather weak. 
In Condition \ref{cond:fp_conti}, (i) is a natural requirement, and we investigate (ii) when units are i.i.d.\ samples from a superpopulation. 
Assume that the potential event times $T_i(1)=T_i(0)$'s are i.i.d.\ samples from a superpopulation with continuous distribution, and the distribution of potential censoring times $(C_i(1), C_i(0))$ is the same along the sequence of finite populations. 
If the minimum of the pseudo potential censoring times has a positive probability to be no less than the potential event time, i.e., 
$
\E\{
\tilde{G}(T_i(0))
\} 
=
\E\{
G_1(T_i(0)) G_0(T_i(0))
\} > 0, 
$
then as $n\rightarrow\infty$, with probability one, $\tilde{d}=1$, $\tilde{g}$ has a positive limit,  and (ii) in Condition \ref{cond:fp_conti} holds.

The theorem below shows that Condition \ref{cond:fp_conti} is sufficient for the asymptotic standard Gaussianity of the logrank statistic $\LR$ in \eqref{eq:logrank_randomization}.

\begin{theorem}\label{thm:CLT_conti}
	Under Assumptions \ref{asmp:rbe} and \ref{asmp:noninformative} and the null hypothesis $H_0$ in \eqref{eq:H0}, 
	and conditioning on all the potential event times, 
	if Condition \ref{cond:fp_conti} holds, then $U$ is a consistent variance estimator for $L$ in \eqref{eq:L_U},  i.e.,  
	$$ \frac{U}{\Var\left\{L \mid \bm{T}(1), \bm{T}(0) \right\}} \mid \bm{T}(1), \bm{T}(0) 
	\ \convergep \  1, 
	$$ 
	and 
	the logrank statistic $\LR = U^{-1/2} L$ in \eqref{eq:logrank_randomization} is asymptotically standard Gaussian, i.e., 
	\begin{align}\label{eq:logrank_clt_finite_conti}
		\LR \mid \bm{T}(1), \bm{T}(0) \ \converged \  \mathcal{N}(0,1). 
	\end{align}
\end{theorem}

The asymptotic Gaussianity in Theorem \ref{thm:CLT_conti} is established using the martingale central limit theorem \citep{hall1980}. 
In particular, the first term in \eqref{eq:cond_conti} guarantees the consistency of the variance estimator $U$, 
and the second term in \eqref{eq:cond_conti} %
guarantees the conditional Lindeberg condition. 
Moreover, 
although \eqref{eq:logrank_clt_super} and \eqref{eq:logrank_clt_finite_conti} lead to the same asymptotic distribution for the logrank statistic, they rely on different sources of randomness.
Specifically, the asymptotics in \eqref{eq:logrank_clt_super} relies on the random potential event times that are assumed to be i.i.d.\ from {\rev some} population while fixing the treatment assignments. 
On the contrary, 
the asymptotics in \eqref{eq:logrank_clt_finite_conti} relies on the random treatment assignments that can be physically implemented while fixing the potential event times, which are then allowed to have arbitrary dependence and heterogeneity across units. 
Therefore, 
Theorem \ref{thm:CLT_conti}, or more generally finite population inference, can sometimes provide a broader justification for the logrank test compared to usual superpopulation inference. 
This is also confirmed by the simulation results in Section \ref{sec:simu_logrank}.

\subsection{
A bounded number of distinct potential event times
}\label{sec:finite_distinct_event_times}

We now consider the case in which the number of distinct potential event times $K$ is bounded as the sample size $n$ increases. 
Without loss of generality, we assume that $K$ is fixed along the sequence of finite populations, because we can always add several distinct event times with no units potentially having events at those times. 
When $K$ is fixed, 
for a specific time $t_k$,  
if (i) there are non-negligible proportions of units receiving both treatment and control, 
(ii) the potential censoring times under both treatment and control have positive probabilities to be no less than $t_k$, 
and (iii) there are non-negligible proportions of units potentially having events both at time $t_k$ and later than $t_k$, 
then we expect that all margins of the contingency table at time $t_k$, $(N_{1k}, N_{0k}, D_k, N_k - D_k)$, as well as the conditional variance $V_{k}$, go to infinity as the sample size $n$ goes to infinity. 
The property of hypergeometric distributions then enables us to establish the 
Gaussian approximation 
for $V_{k}^{-1/2}(D_{1k}-\ED_{k})$ \citep[see, e.g.,][]{lehmann2006nonparametrics, Vatutin1982, lidingclt2016}. 

However, due to the dependence among the contingency tables, the marginal asymptotic Gaussianity of $V_{k}^{-1/2}(D_{1k}-\ED_{k})$, {\rev for $k=1,\ldots, K$}, does not guarantee their joint asymptotic Gaussianity, as well as the 
Gaussian approximation 
for the logrank statistic. 
Therefore, 
we need a finer analysis for the Gaussian approximation of hypergeometric distributions. %
We invoke
the Berry--Esseen type result for bounding the difference between a hypergeometric distribution and its Gaussian approximation \citep[][Theorem 2.3]{kou1996asymptotics}. 
This %
further helps us analyze
the difference between the joint distribution of the standardized conditional hypergeometric random variables $V_{k}^{-1/2}(D_{1k}-\ED_{k})$'s and the multivariate standard Gaussian distribution. 

Recall that the finite population asymptotics embeds the units into a sequence of finite populations with increasing sizes. 
We impose the following regularity conditions on the sequence of finite populations, 
which involve 
$p_z$'s for the treatment assignment mechanism, 
$G_{z}(\cdot)$'s for the censoring mechanism, 
and %
$(d_k, h_k)$'s
that are deterministic functions of the potential event times.

\begin{condition}\label{cond:fp_discrete}
	Along the sequence of finite populations satisfying $H_0$ in \eqref{eq:H0}, as $n\rightarrow \infty$, 
	the number of distinct potential event times $K$ is fixed, and  
	for $z=0,1$ and $1\le k \le K$, 
	\begin{itemize}
		\item[(i)] the probability of receiving treatment arm $z$, $p_z$, has a positive limit; 
		\item[(ii)] the proportion of units potentially having events at time $t_k$, $d_k/n$, has a limiting value;
		\item[(iii)] the survival function of the potential censoring time evaluated at time $t_k$, 
		$G_z(t_k)$, has a limiting value;
		\item[(iv)] there exists at least one 
		$k$ such that $d_k/n, G_{1}(t_k)$ and $G_{0}(t_k)$ have positive limits, 
		and $\lim_{n\rightarrow \infty} h_k < 1$.
	\end{itemize}
\end{condition}

Below we give some intuition for Condition \ref{cond:fp_discrete}, showing that the regularity conditions are rather weak. 
In Condition \ref{cond:fp_discrete}, 
(i) is a natural requirement,
and 
the limit $h_k$ in (iv) exists due to 
(ii) and the positive limit of $d_k/n$.
{\rev Analogous to the discussion after Condition \ref{cond:fp_conti}, we investigate (ii)--(iv) when units are i.i.d.\ samples from a superpopulation. 
Specifically, we assume that the potential event times $T_i(1)=T_i(0)$'s are i.i.d. from some distribution, and the distribution of the potential censoring times $(C_i(1), C_i(0))$ is the same along the sequence of finite populations. 
Moreover, we assume that the distribution of the potential event time is discrete and has a finite support $\mathcal{T}$. 
}
If $\mathcal{T}$ has at least two elements, and 
the values of both $G_1(\cdot)$ and $G_2(\cdot)$ evaluated at the  smallest element of $\mathcal{T}$ are positive, 
then as $n \rightarrow \infty$, with probability one, (ii)--(iv) in Condition \ref{cond:fp_discrete} hold.

The following theorem shows that Condition \ref{cond:fp_discrete} is sufficient for the asymptotic standard Gaussianity of the logrank statistic $\LR$ in \eqref{eq:logrank_randomization}

\begin{theorem}\label{thm:CLT_discrete}
Theorem \ref{thm:CLT_conti} holds if we replace Condition \ref{cond:fp_conti} by  Condition
\ref{cond:fp_discrete}. 
\end{theorem}

Although leading to the same conclusion, Theorem \ref{thm:CLT_discrete} differs significantly from Theorem \ref{thm:CLT_conti}. 
The asymptotic Gaussianity in Theorem \ref{thm:CLT_conti} requires infinitely many distinct potential event times and 
is justified by the martingale central limit theorem, 
while the one in Theorem \ref{thm:CLT_discrete} requires a bounded number of distinct potential event times and is justified by the Gaussian approximation of hypergeometric distributions. 
Moreover, 
both %
Theorems \ref{thm:CLT_conti} and \ref{thm:CLT_discrete}
differ from 
the asymptotic Gaussianity in 
\eqref{eq:logrank_clt_super} under usual superpopulation inference;
the former two fix the potential event times and rely crucially on the randomness of the treatment assignments 
while the latter one is the opposite. 

\section{A simulation study under violation of Assumptions 1 or 2}\label{sec:violation}

We conduct a simulation study to show that Assumptions \ref{asmp:rbe} and \ref{asmp:noninformative} are in some way necessary for the validity of the logrank test. 
In particular, we consider the violation of Assumptions \ref{asmp:rbe} and \ref{asmp:noninformative} by allowing the distribution of the treatment assignment and potential censoring times to vary across units. 
Similar to Section \ref{sec:simu_logrank}, we require that the potential event times $(\bm{T}(1), \bm{T}(0))$, the treatment assignment $\bm{Z}$, and the potential censoring times $(\bm{C}(1), \bm{C}(0))$ are mutually independent. 
The potential event times are generated from \eqref{eq:model_potential_outcome}, with $\rho=0.5$ and $\theta = 1$ (i.e., case 4 in Table \ref{tab:simulation}). 
Thus, the potential event times for all units are dependent and heterogeneous, where the latter is introduced by the fixed covariate vector $\bm{X}$ in \eqref{eq:model_potential_outcome}. 
\begin{table}[htb]
    \centering
	\caption{Models for generating the treatment assignments and potential censoring times}\label{tab:violation}
	\begin{tabular}{ccccc}
		\toprule
		Case & $I_z$ & $I_c$  & Heterogeneous assignment & Heterogeneous censoring\\
		\midrule
		i & $0$ & $0$ & No & No\\
		ii & $0$ & $1$ & No & Yes\\
		iii & $1$ & $0$ & Yes & No\\
		iv & $1$ & $1$ & Yes & Yes\\
		\bottomrule
	\end{tabular}
\end{table}
For $1\le i \le n$, 
the treatment assignments $Z_i$'s and the potential censoring times $(C_i(1), C_i(0))$'s are independent samples from the following models:
\begin{align}\label{eq:assign_heter}
	Z_i \sim \Bernoulli\left( 0.5 + I_z \cdot 0.2  \cdot (1-2X) \right), %
\end{align}
and
\begin{align}\label{eq:censor_heter}
	C_i(1) \sim 10^{1+I_c (X_i-1)} \cdot \exponential(1), 
	\ \ 
	C_i(0) \sim 10^{I_c (X_i-1)} \cdot \exponential(1), 
\end{align}
where $C_i(1) \ind C_i(0)$, and $I_z, I_c \in \{0,1\}$ are binary indicators for whether the treatment assignment and censoring mechanisms are heterogeneous across units.  
In particular, we consider four cases for generating the treatment assignments and potential censoring times from \eqref{eq:assign_heter} and \eqref{eq:censor_heter}, with values of $I_{z}$ and  $I_{c}$ listed in Table \ref{tab:violation}. 
These correspond to cases with homogeneous or heterogeneous treatment assignment or censoring mechanisms.

\begin{figure}[htbp]
	\centering
	\begin{subfigure}{.4\textwidth}
		\centering
		\includegraphics[width=1\linewidth]{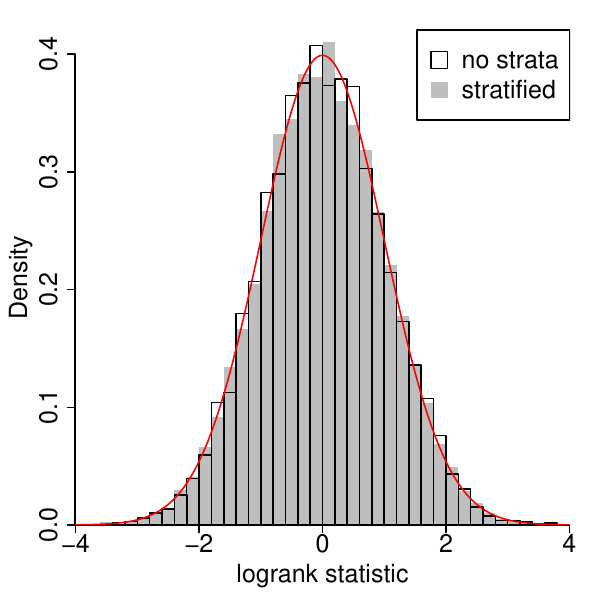}
		\caption*{case i}
	\end{subfigure}%
	\begin{subfigure}{.4\textwidth}
		\centering
		\includegraphics[width=1\linewidth]{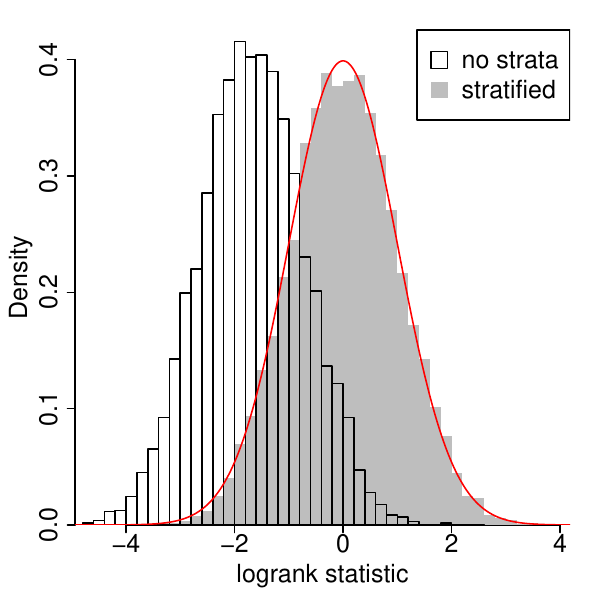}
		\caption*{case ii}
	\end{subfigure}
	\begin{subfigure}{.4\textwidth}
		\centering
		\includegraphics[width=1\linewidth]{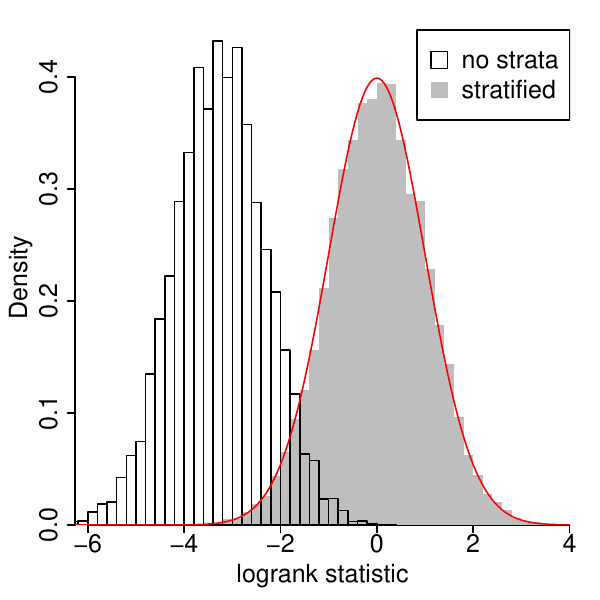}
		\caption*{case iii}
	\end{subfigure}%
	\begin{subfigure}{.4\textwidth}
		\centering
		\includegraphics[width=1\linewidth]{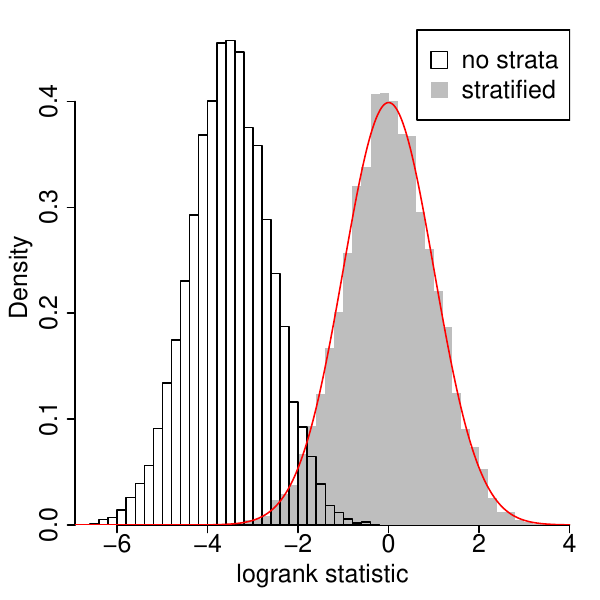}
		\caption*{case iv}
	\end{subfigure}
	\caption{
		Histograms of the logrank and stratified logrank statistics under homogeneous/heterogeneous treatment assignment and censoring mechanisms as in Table \ref{tab:violation}. 
		The white histograms denote the distributions of the logrank statistic, 
		and the gray histograms denote the distributions of the stratified logrank statistic. 
		The red lines denote the density of a standard Gaussian distribution.
	}\label{fig:logrank_violation}
\end{figure}
First, we study the distribution of the logrank statistic under the four cases in Table \ref{tab:violation} for generating the treatment assignments and potential censoring times. 
To mimic finite population inference, 
we simulate the fixed potential event times as in \eqref{eq:fix_potential_event_time}, 
and the resulting distribution for the logrank statistic will be its conditional distribution given $(\bm{T}(1), \bm{T}(0))$ fixed at \eqref{eq:fix_potential_event_time}. 
The white histograms in Figure \ref{fig:logrank_violation} 
show the distributions of  
the logrank statistic under the four cases with or without heterogeneity in the treatment assignment or censoring mechanisms. 
From Figure \ref{fig:logrank_violation}, 
when there exists heterogeneity in treatment assignment or censoring mechanisms, 
i.e., Assumptions \ref{asmp:rbe} or \ref{asmp:noninformative} fail, 
the logrank statistic is no longer close to be standard Gaussian distributed. 
We also conduct simulations with random potential event times from \eqref{eq:model_potential_outcome}, and the corresponding histograms under the four cases are similar to that with fixed potential event times, as shown in Figure A4 in the Supplementary Material. 
In simple words, even with randomly generated potential event times,  heterogeneity in the treatment assignment or censoring mechanisms still invalidates the logrank test.

Second, we study the distribution of the stratified logrank test under the previous simulation setting, where the stratified logrank statistic will be introduced and discussed shortly in the next section. 
In particular, we stratify on the covariates $X_i$'s in \eqref{eq:model_potential_outcome}. 
The gray histograms in Figure \ref{fig:logrank_violation} show the distributions of the stratified logrank statistic under the four cases in Table \ref{tab:violation}. 
From Figure \ref{fig:logrank_violation}, 
they 
are all close to a standard Gaussian distribution, indicating that the stratified logrank test is approximately valid. 
This phenomena is also true when the potential event times are randomly generated, as shown in Figure A4 in the Supplementary Material. 
Comparing the white and gray histograms in Figure \ref{fig:logrank_violation}, 
stratifying on the covariates $X_i$'s is crucial to ensure the validity of the logrank test. 
Intuitively, 
this is because, 
within each stratum defined by the covariates $X_i$'s, 
both the treatment assignment and censoring mechanisms become homogeneous across units.  %
In the next section, we will give a more detailed discussion for the stratified logrank test,  and provide rigorous conditions to ensure its asymptotic validity.

\section{Extension to Stratified Logrank Test}\label{sec:strata}
We 
provide a solution to possible violation of Assumptions \ref{asmp:rbe} or \ref{asmp:noninformative}, namely the stratified logrank test \citep{mantelhaenszel1959}, under the scenario where a categorical pretreatment covariate can fully adjust the inhomogeneity of treatment assignment and censoring mechanisms across units.

In usual superpopulation inference, 
the
stratified logrank test is invoked when we want to 
adjust for some covariates that 
may
be related to 
or can help explain heterogeneity in 
the event times. 
In finite population inference with i.i.d.\ treatment assignment and censoring (i.e., Assumptions \ref{asmp:rbe} and \ref{asmp:noninformative} hold), 
as discussed in Sections \ref{sec:rand_dist_logrank} and \ref{sec:asym_dist},
the potential event times are allowed to be any constants and can thus have arbitrary heterogeneity across units. 
Therefore, we 
are able to conduct valid inference using the original logrank test even if we ignore some covariates related to the event times. 
However, 
as demonstrated by simulation in Section \ref{sec:violation}, 
the stratified logrank test can still be useful 
in finite population inference 
if the distributions of treatment assignment and potential censoring times depend on some pretreatment covariates. 
For example, the treatment assignment may come from a randomized block design with varying probabilities of receiving active treatment across blocks as in \eqref{eq:assign_heter}, and the censoring pattern may differ across units with different covariates as in \eqref{eq:censor_heter}.

Below we first introduce the stratified logrank statistic.  
Suppose that the units are divided into $S$ strata. %
For units in each stratum $s \in \{1, 2, \ldots,S\}$, 
similar to Section \ref{sec:mds}, 
let $L_{[s]}$ be the overall observed-expected difference, and $U_{[s]}$ be the corresponding overall variance. 
We can then represent the (standardized) stratified logrank statistic as 
\begin{align}\label{eq:SLR}
\SLR \equiv \left( \sum_{s=1}^{S} U_{[s]} \right)^{-1/2} \cdot \sum_{s=1}^{S} L_{[s]}.
\end{align}
When 
the treatment assignments and potential censoring times are independent across strata, 
and 
Assumptions \ref{asmp:rbe} and \ref{asmp:noninformative} hold within each stratum, 
the observed-expected differences $L_{[s]}$'s, as well as the variances $U_{[s]}$'s,  are mutually independent for all $1\le s \le S$, 
and the  means and variances for the randomization distributions of the $L_{[s]}$'s can be similarly derived as in Theorem \ref{thm:mean_var_L}. 
These immediately imply the first two moments of the randomization distribution of the total observed-expected difference $\sum_{s=1}^{S} L_{[s]}$. 
Specifically, 
conditioning on all the potential event times, 
$\sum_{s=1}^{S} L_{[s]}$ has mean zero, and 
$ \sum_{s=1}^{S} U_{[s]}$ is an unbiased variance estimator for it. 
Therefore, 
$\SLR$ is essentially a standardization of the total observed-expected difference.  
To establish the 
asymptotic Gaussianity of 
the stratified logrank statistic,  
we introduce the following regularity condition. %

\begin{condition}\label{cond:strata}
	Along the sequence of finite populations satisfying $H_0$ in \eqref{eq:H0}, 
	as $n\rightarrow \infty$, the number of strata $S$ is fixed, and for each stratum $1\le s \le S$, 
	\begin{itemize}
		\item[(i)] the number of units in stratum $s$ goes to infinity,
		\item[(ii)] Condition \ref{cond:fp_conti} or \ref{cond:fp_discrete} holds for units within stratum $s$. 
	\end{itemize}
\end{condition}

In Condition \ref{cond:strata}, 
the number of units within each stratum increases with the total sample size $n$, but 
the proportion of units within each stratum does not need to have a limit. 
From the discussion for Conditions \ref{cond:fp_conti} and \ref{cond:fp_discrete} 
in Section \ref{sec:asym_dist}, 
when units within each stratum are i.i.d.\ samples from some superpopulation, 
some weak conditions on the superpopulation can ensure that Condition \ref{cond:strata} holds with probability one. 
The following theorem shows that Condition \ref{cond:strata} is sufficient for the 
asymptotic standard Gaussianity
of the stratified logrank statistic $\SLR$.  

\begin{theorem}\label{thm:CLT_SLR}
	Under the null hypothesis 
	$H_0$ in \eqref{eq:H0} and conditioning on all the potential event times, 
	if (a) the treatment assignments and potential censoring times, $(Z_i, C_i(1), C_i(0))$'s, are mutually independent across all units,
	(b) Assumptions \ref{asmp:rbe} and \ref{asmp:noninformative} hold within each stratum, and 
	(c) Condition \ref{cond:strata} holds, 
	then the stratified logrank statistic \eqref{eq:SLR} is asymptotically standard Gaussian, i.e., 
	$$
	\SLR \mid \bm{T}(1), \bm{T}(0) \converged \mathcal{N}(0,1). 
	$$
\end{theorem}

In Theorem \ref{thm:CLT_SLR}, 
we only require Assumptions \ref{asmp:rbe} and \ref{asmp:noninformative} to hold within each stratum, and thus allow the probability of receiving the active treatment and the joint distribution of the potential censoring times under treatment and control to vary across strata. 
Thus, 
Theorem \ref{thm:CLT_SLR} is particularly useful for randomized block designs and %
when the censoring mechanism varies across strata, as demonstrated by simulation in Section \ref{sec:violation}.  
Moreover, 
similar to
the discussion at the end of Section \ref{sec:two_mechanism}, 
we can relax condition (b) in Theorem \ref{thm:CLT_SLR} by 
instead requiring 
$(Z_i, C_i(1),C_i(0))$'s to be independent of the potential event times and i.i.d.\
within each stratum, 
allowing dependence between the treatment assignment and the potential censoring times. 

The choice of strata is an important issue in practice. To ensure the validity of the stratified logrank test, we need to consider the covariates that affect either the treatment assignment or potential censoring times. 
However, 
the stratified logrank test can only take into account discrete covariates. 
It will be interesting to extend the randomization-based logrank test to adjust for more general (continuous) covariates, 
possibly
utilizing ideas from inverse 
probability 
weighting \citep{Horvitz1952, RUBIN1983, robins2020}, Cox regression models \citep{Cox1972} and regression adjustment for randomized experiments without censoring \citep{lin2013, Basse2021}.

\section{Conclusion and Discussion}\label{sec:conclusion}

We studied the randomization distribution of the logrank statistic when the treatment assignment and potential censoring times are i.i.d.\ across all units, 
and proved its asymptotic standard Gaussianity under certain regularity conditions. 
The theory developed here does not require any distributional assumption on the potential event times. 
It supplements the classical theory under usual superpopulation inference, and helps 
provide a broader justification for 
the logrank test. 
We invoked the stratified logrank test to overcome the inhomogeneity in treatment assignment and censoring mechanisms among units, 
and 
studied its randomization distribution as well as the corresponding asymptotic approximation. 

In this paper we focus mainly on the logrank test. 
It will be interesting to extend the discussion to other variants or generalizations of the logrank test.  
For example, researchers have considered weighted logrank tests \citep[see, e.g.,][]{Harrington1991}, 
and recently \citet{kernel_logrank2021} considered a kernel logrank test, where the test statistic corresponds to the supremum of a potentially infinite collection of weighted logrank statistics with weighting functions belonging to a reproducing kernel Hilbert space.

\section*{Acknowledgements}
We thank the Associate Editor and two reviewers for constructive comments.

\bibliographystyle{plainnat}
\bibliography{causal}

\newpage
\setcounter{equation}{0}
\setcounter{section}{0}
\setcounter{figure}{0}
\setcounter{example}{0}
\setcounter{proposition}{0}
\setcounter{corollary}{0}
\setcounter{theorem}{0}
\setcounter{table}{0}
\setcounter{condition}{0}
\setcounter{assumption}{0}

\renewcommand {\theproposition} {A\arabic{proposition}}
\renewcommand {\theexample} {A\arabic{example}}
\renewcommand {\thefigure} {A\arabic{figure}}
\renewcommand {\thetable} {A\arabic{table}}
\renewcommand {\theequation} {A\arabic{equation}}
\renewcommand {\thelemma} {A\arabic{lemma}}
\renewcommand {\thesection} {A\arabic{section}}
\renewcommand {\thetheorem} {A\arabic{theorem}}
\renewcommand {\thecorollary} {A\arabic{corollary}}
\renewcommand {\thecondition} {A\arabic{condition}}
\renewcommand {\theassumption} {A\arabic{assumption}}

\renewcommand {\thepage} {A\arabic{page}}
\setcounter{page}{1}

\begin{center}
	\bf \LARGE 
	Supplementary Material 
\end{center}

Appendix \ref{sec:slight_relax} discusses the relaxation of Assumptions \ref{asmp:rbe} and \ref{asmp:noninformative}, by allowing dependence between the treatment assignment and potential censoring times. 

Appendix \ref{sec:summary} summarizes and compares justifications for the logrank test from both superpopulation and finite population inferences.

Appendix \ref{app:simulation} provides additional simulation results for the main text.

Appendix \ref{app:exact} provides the proofs for the exact randomization distribution of the logrank statistic. 

Appendix \ref{app:asym_infinite} provides the proof for the asymptotic standard Gaussianity of the logrank statistic when the number of potentially distinct event times goes to infinity as the sample size goes to infinity. 

Appendix \ref{app:asymp_finite} provides the proof for the asymptotic standard Gaussianity of the logrank statistic when the number of potentially distinct event times is bounded as the sample size increases. 

Appendix \ref{app:logrank_strata} provides the proof for the asymptotic standard Gaussianity of the stratified logrank statistic.

Throughout the proofs of all lemmas and theorems from Appendix \ref{app:exact}--\ref{app:logrank_strata}, 
we are all conducting finite population inference by conditioning on all the potential event times $\bm{T}(1)$ and $\bm{T}(0)$. 
For clarification, in the statement of the lemmas, we write such conditioning on $(\bm{T}(1), \bm{T}(0))$ explicitly. 
For descriptive convenience, in the proofs of the lemmas and the theorems, 
we suppress the conditioning on $(\bm{T}(1), \bm{T}(0))$, 
assume the potential event times are always conditioned on implicitly, 
and thus equivalently view the $T_i(1)$'s and $T_i(0)$'s as fixed constants.

\section{Relaxation of Assumptions 1 and 2}\label{sec:slight_relax}

If we relax the %
independence between 
the treatment assignment and potential censoring times %
as in Assumptions \ref{asmp:rbe} and \ref{asmp:noninformative}, 
the results on finite population inference in Sections \ref{sec:rand_dist_logrank} and \ref{sec:asym_dist}  
still hold
except for a slight modification on the definition of $G_z(\cdot)$ for $z=0,1$.  
Specifically, we can replace Assumptions \ref{asmp:rbe} and \ref{asmp:noninformative} by the following assumption on both the treatment assignment and censoring mechanisms. 

\begin{assumption}\label{asmp:noninfor_both}
	The treatment assignments and potential censoring times are independent of the potential event times, in the sense that  
	$(\bm{Z}, \bm{C}(1), \bm{C}(0))$ $\ind 	(\bm{T}(1), \bm{T}(0))$, 
	and the $(Z_i, C_i(1), C_i(0))$'s are i.i.d.\ across all units. 
\end{assumption}

Assumption \ref{asmp:noninfor_both} is more general than Assumptions \ref{asmp:rbe} and  \ref{asmp:noninformative}. It allows the treatment assignment to depend on the potential censoring times. 
For the randomization distribution of the logrank statistic, 
all the conclusions established under Assumptions \ref{asmp:rbe} and \ref{asmp:noninformative} also hold 
under Assumption \ref{asmp:noninfor_both}, 
with the survival functions $G_z(\cdot)$'s 
changed to 
$\overline{G}_z(\cdot)$'s, 
where $\overline{G}_z(c) \equiv \pr(C_i(z)\ge c \mid Z_i=z)$ for any $c\in \mathbb{R}$ and $z=0,1$. 

Below 
we explain why the theory developed in Sections  \ref{sec:rand_dist_logrank} and \ref{sec:asym_dist} is still true when Assumption \ref{asmp:noninfor_both} holds 
instead of Assumptions \ref{asmp:rbe} and \ref{asmp:noninformative}. 
For each unit $i$, we introduce two independent random variables $\overline{C}_i(1)$ and $\overline{C}_i(0)$ for a pair of pseudo potential censoring times that follow the distributions of $C_i(1)$ given $Z_i=1$ and $C_i(0)$ given $Z_i=0$, respectively, 
i.e., 
$$
\overline{C}_i(1) \sim C_i(1) \mid Z_i=1, 
\quad 
\overline{C}_i(0) \sim C_i(0) \mid Z_i=0, 
\quad
\overline{C}_i(1) \ind \overline{C}_i(0). 
$$
We further assume the pseudo potential censoring times 
$(\overline{C}_i(1), \overline{C}_i(0))$'s are i.i.d.\  across all units, and they are independent of the treatment assignments $Z_i$'s. 
By construction, 
if the units' potential censoring times were replaced by the pseudo potential censoring times $(\overline{C}_i(1), \overline{C}_i(0))$'s, Assumptions \ref{asmp:rbe} and \ref{asmp:noninformative} would hold and all the results established in Sections \ref{sec:rand_dist_logrank} and  \ref{sec:asym_dist}, including the central limit theorems for the logrank statistics, would be true. 
Note that, by construction,
for each unit $i$,  
conditional on $Z_i$, 
the realized pseudo censoring time $\overline{C}_i \equiv Z_i \overline{C}_i(1) + (1-Z_i) \overline{C}_i(0)$ follows the same distribution as the realized censoring time $C_i$, and consequently, 
$$
(\overline{C}_1, \overline{C}_2, \ldots, \overline{C}_n) \mid \bm{Z} \ \sim \ 
(C_1, C_2, \ldots, C_n) \mid \bm{Z}.
$$
This implies that, under Assumption \ref{asmp:noninfor_both} and conditioning on all the potential event times, 
the logrank statistic with the original potential censoring times $(C_i(1), C_i(0))$'s follows the same distribution as that with the pseudo potential censoring times $(\overline{C}_i(1), \overline{C}_i(0))$'s. 
Therefore, 
for the randomization distribution of the logrank statistic, 
all the conclusions established under Assumptions \ref{asmp:rbe} and \ref{asmp:noninformative} also hold 
under Assumption \ref{asmp:noninfor_both}, 
with the survival functions $G_z(\cdot)$'s for the potential censoring times changed to survival functions
$\overline{G}_z(\cdot)$'s for the pseudo potential censoring times. 
Specifically, $\overline{G}_z(c) \equiv \pr(\overline{C}_i(z)\ge c) = \pr(C_i(z)\ge c \mid Z_i=z)$, for any $c\in \mathbb{R}$ and $z=0,1$.

\section{Summary: broader justifications for the logrank test}\label{sec:summary}

From the discussion in Sections \ref{sec:review_logrank}--\ref{sec:asym_dist} and Appendix \ref{sec:slight_relax}, 
finite population inference can provide additional justification for the validity of the logrank test, relaxing the distributional assumptions on the potential event times and supplementing the classical theory from usual superpopulation inference. %
Below we summarize justifications for the logrank test from both the superpopulation inference \citep[see][]{Gill1980, Harrington1991} 
and finite population inference (see Sections \ref{sec:rand_dist_logrank} and \ref{sec:asym_dist} and Appendix \ref{sec:slight_relax}). 
We assume that the treatment assignments and potential censoring times are independent of the potential event times, i.e., 
$
(\bm{Z}, \bm{C}(1), \bm{C}(0)) \ind (\bm{T}(1), \bm{T}(0)). 
$
This assumption requires unconfounded treatment assignment and non-informative censoring, and is required for both the superpopulation and finite population inferences for the logrank test. 
Below we discuss additional assumptions needed for the superpopulation and finite population justifications, respectively. 

First, 
conditional on the treatment assignment vector $\bm{Z}$ and all the potential censoring times $\bm{C}(1)$ and $\bm{C}(0)$, 
the realized event times $T_i$'s and censoring times $C_i$'s follow the general random censorship model with fixed censorship \citep[][Page 55]{Gill1980}. 
From \citet{Gill1980} and \citet{Harrington1991}, 
if the potential event times $(T_i(1), T_i(0))$'s are i.i.d.\ across all units and the potential event times under treatment and control are identically distributed, i.e., $T_i(1) \sim T_i(0)$, 
then, with certain regularity conditions, the conditional distribution of the logrank statistic given $(\bm{Z}, \bm{C}(1), \bm{C}(0))$ is asymptotically standard Gaussian. 
Therefore, 
the superpopulation inference can justify the logrank test without any distributional assumption on the treatment assignments and potential censoring times, 
and it relies crucially on the randomness in the i.i.d.\ sampling of units from a %
superpopulation. 
Second, from the discussion in Sections \ref{sec:rand_dist_logrank} and \ref{sec:asym_dist} and Appendix \ref{sec:slight_relax}, 
if the treatment assignment and potential censoring times $(Z_i, C_i(1), C_i(0))$'s are i.i.d.\ across all units, 
then, with certain regularity conditions (e.g., Condition \ref{cond:fp_conti} or \ref{cond:fp_discrete}), 
the conditional distribution of the logrank statistic given $(\bm{T}(1), \bm{T}(0))$ is asymptotically standard Gaussian 
under Fisher's null $H_0$ in \eqref{eq:H0}. 
Therefore, finite population inference can justify the logrank test without any distributional assumption on the potential event times, 
and it relies crucially on the randomness in the treatment assignments and censoring times.
The top half of Table \ref{tab:justification} summarizes the assumptions needed for the validity of the logrank test, with justifications from either superpopulation or finite population inference. 
From Table \ref{tab:justification}, finite population inference 
provide additional justification for the logrank test, 
supplementing the results from the classical superpopulation inference. 
Combining both types of inference, they provide a broader justification for the logrank test.

\begin{table}
    \centering
	\caption{Justifications for the asymptotic standard Gaussianity of logrank statistic. 
		The potential event times are assumed to be independent of the treatment assignment and potential censoring times, i.e., 
		$
		(\bm{T}(1), \bm{T}(0)) \independent (\bm{Z}, \bm{C}(1), \bm{C}(0)). 
		$
		\textsf{NDA} stands for no distributional assumption. 
		The top two rows considers the justification for the logrank test in general case. The bottom two rows considers the justification for the logrank test in a special case for testing Fisher's null of no treatment effect on any unit under a Bernoulli randomized experiment. 
	}\label{tab:justification}
	\resizebox{\columnwidth}{!}{
		\begin{tabular}{cccc}
			\toprule
			& Potential event times & Treatment assignment \& & Justification for\\
			& & censoring mechanisms & logrank test \\[2pt]
			& $(\bm{T}(1), \bm{T}(0))$ &   $(\bm{Z}, \bm{C}(1), \bm{C}(0))$ & \\
			\midrule
		General case	& $(T_i(1), T_i(0))$'s are i.i.d., $T_i(1)\sim T_i(0)$  
			& 
			\textsf{NDA}
			& 
			superpopulation\\
			& $T_i(1) = T_i(0)$ for all $i$, \textsf{NDA}
			& 
			$(Z_i, C_i(1), C_i(0))$'s are i.i.d.
			& 
			finite population\\
			\midrule
			Special case & $T_i(1)=T_i(0)$'s are i.i.d.
			& 
			\textsf{NDA}
			& 
			superpopulation\\
			& \textsf{NDA}
			& 
			$(C_i(1), C_i(0))$'s are i.i.d.
			& 
			finite population\\
			\bottomrule
		\end{tabular}
	}
\end{table}

Below we consider a special case in which we want to test Fisher's null $H_0$ in \eqref{eq:H0} under a Bernoulli randomized experiment, i.e., Assumption \ref{asmp:rbe} holds. 
Note that here Assumption \ref{asmp:rbe} is justified by the physical randomization of the treatment assignment and it holds by the design of the experiment. 
In this special case, 
with certain regularity conditions, 
the logrank test is asymptotically valid when either the potential event times or potential censoring times are i.i.d.\ from a certain distribution, with justification from the superpopulation and finite population inferences, respectively.
The bottom half of Table \ref{tab:justification} summarizes the assumptions needed for the validity of the logrank test for testing Fisher's null under a Bernoulli randomized experiment. 
From Table \ref{tab:justification}, 
we can see that the logrank test has a double-robust property, in the sense that it is asymptotically valid when either the potential event times or the potential censoring times are i.i.d. 
However, 
these two i.i.d. assumptions on potential event and censoring times, although they look symmetric, are quite different in nature. 
Specifically, 
the potential event times are of interest, while the potential censoring times are generally not of interest and introduce difficulties in inference for potential event times. 

The appropriateness of the two i.i.d. assumptions can vary on a case-by-case basis. For potential event times it depends on
whether the enrolled experimental units are homogeneous in their reactions to the treatment and control; 
it can be more reasonable when the enrolled units are randomly selected from some population of interest. 
For potential censoring times it depends on
the censoring mechanism; 
in some cases, 
the experimenters are able to control the censoring mechanism and can thus make the assumption hold by homogeneous censoring for units under each treatment arm. 
For example, the experimenters can follow the units within the same treatment arm for random amounts of time drawn i.i.d.\ from some distribution, or even the same fixed amount of time. 
Also, in the case of administrative censoring in which units within the same treatment group are censored at the same time, the i.i.d.\ assumption on potential censoring times can hold if the time for entering the study is i.i.d.\ across all units.

As demonstrated by simulation in Section \ref{sec:violation}, 
under finite population inference, 
the logrank test can be invalid when the potential censoring times are not i.i.d. 
This indicates that, under finite population inference without any distributional assumption on the potential event times, 
the assumption of i.i.d.\ potential censoring times is in some way necessary for the validity of the logrank test; see also the discussion in Section \ref{sec:two_mechanism}.

\section{More simulation results}\label{app:simulation}

In this section, we provide additional simulation results for the two sample $t$-test, the logrank test and the stratified logrank test in Sections \ref{sec:review_t_test}, \ref{sec:simu_logrank} and \ref{sec:violation}, respectively. 

\subsection{A simulation study for the $t$-test under both superpopulation and finite population inferences}

In this subsection, we conduct a simulation study for the $t$-test from both the superpopulation and finite population perspective, in parallel with Section \ref{sec:simu_logrank}. 
We assume the potential event times $\bm{T}(1)$ and $\bm{T}(0)$ are generated from \eqref{eq:model_potential_outcome}, 
the treatment assignments are 
from a completely randomized experiment with half of the units receiving treatment and control, 
and all the potential censoring times are infinity, i.e., there is no censoring. 
Furthermore, 
the potential event times and the treatment assignments are mutually independent. 
 
\begin{figure}%
	\centering
	\begin{subfigure}{.4\textwidth}
		\centering
		\includegraphics[width=1\linewidth]{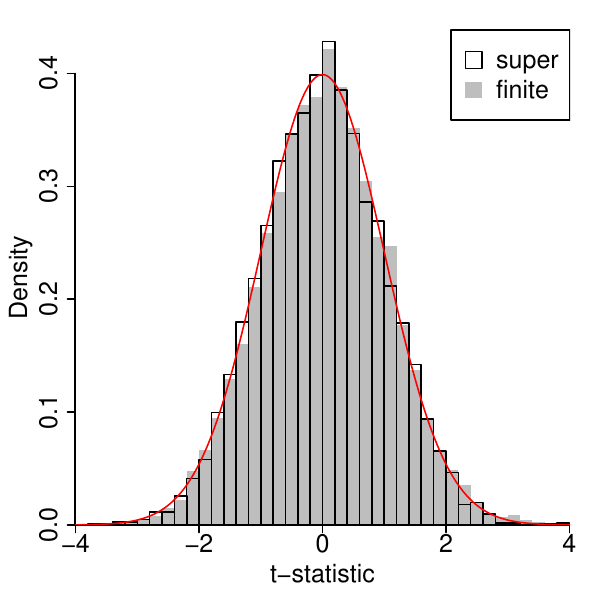}
		\caption*{case 1}
	\end{subfigure}%
	\begin{subfigure}{.4\textwidth}
		\centering
		\includegraphics[width=1\linewidth]{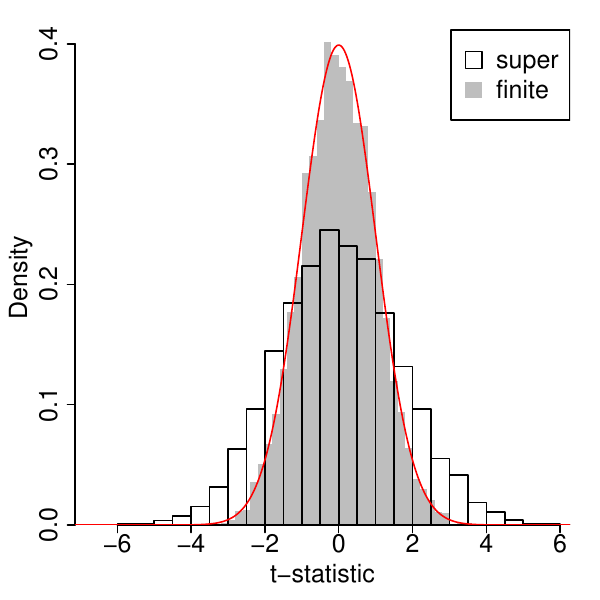}
		\caption*{case 2}
	\end{subfigure}
	\begin{subfigure}{.4\textwidth}
		\centering
		\includegraphics[width=1\linewidth]{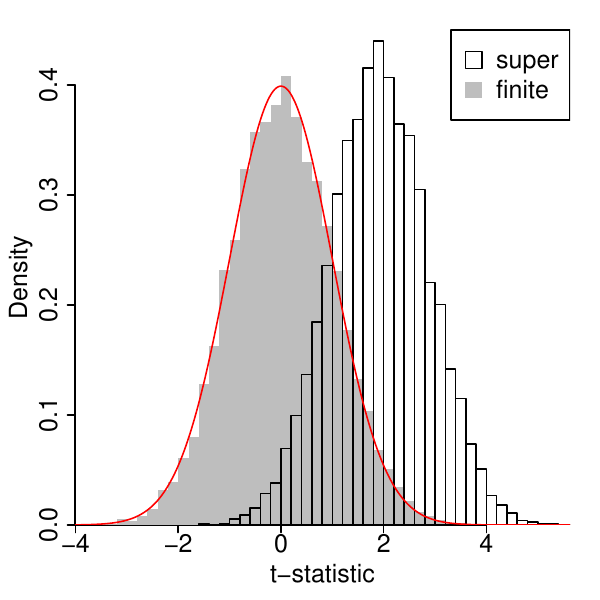}
		\caption*{case 3}
	\end{subfigure}%
	\begin{subfigure}{.4\textwidth}
		\centering
		\includegraphics[width=1\linewidth]{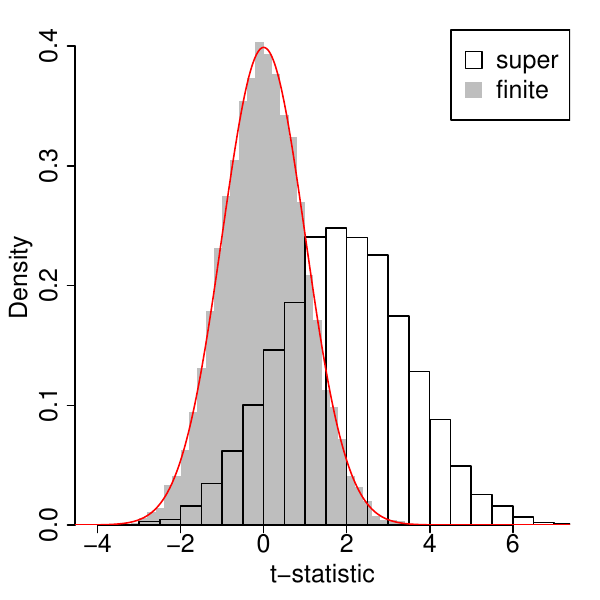}
		\caption*{case 4}
	\end{subfigure}
	\caption{
		Histograms of the $t$-statistic in the superpopulation and finite population settings with parameters values in the four cases of 
		Table \ref{tab:simulation}. 
		The white histograms denote the distributions of the $t$-statistic under superpopulation setting with random potential event times and fixed treatment assignments, 
		and the gray histograms denote the distributions of the $t$-statistic under finite population setting with fixed potential event times and random treatment assignments. 
		The %
		red lines denote the density of a standard Gaussian distribution.
	}\label{fig:t_test}
\end{figure}

\begin{figure}%
	\centering
	\begin{subfigure}{.4\textwidth}
		\centering
		\includegraphics[width=1\linewidth]{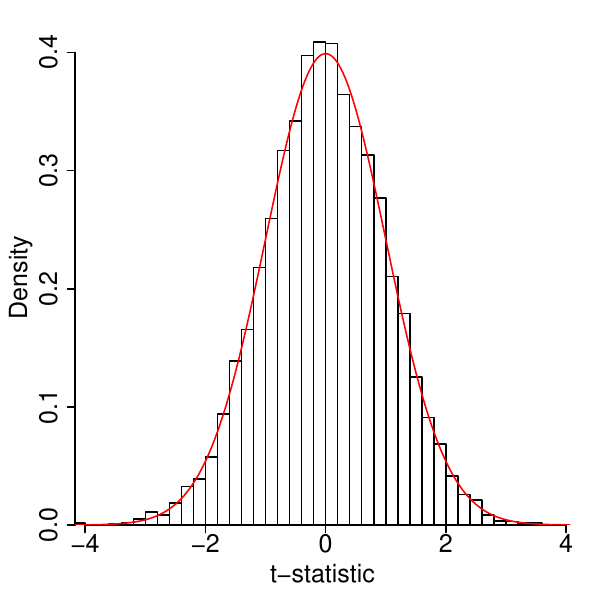}
		\caption*{case 1}
	\end{subfigure}%
	\begin{subfigure}{.4\textwidth}
		\centering
		\includegraphics[width=1\linewidth]{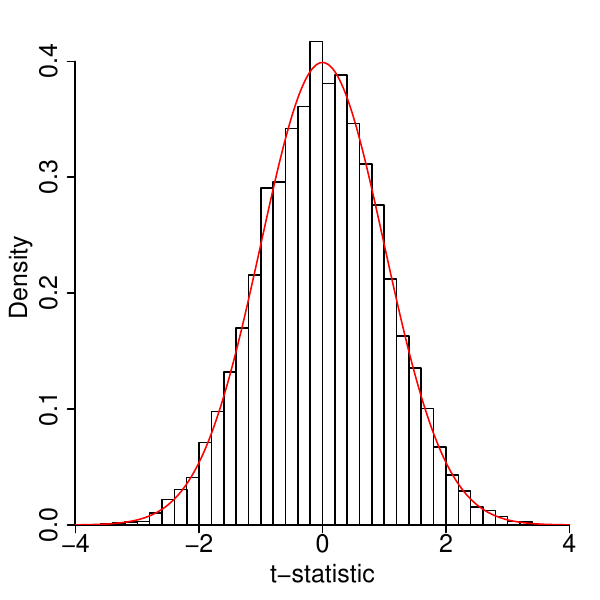}
		\caption*{case 2}
	\end{subfigure}
	\begin{subfigure}{.4\textwidth}
		\centering
		\includegraphics[width=1\linewidth]{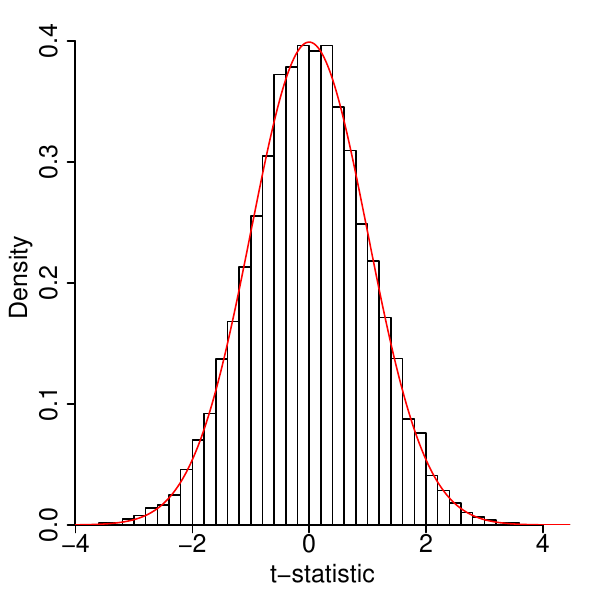}
		\caption*{case 3}
	\end{subfigure}%
	\begin{subfigure}{.4\textwidth}
		\centering
		\includegraphics[width=1\linewidth]{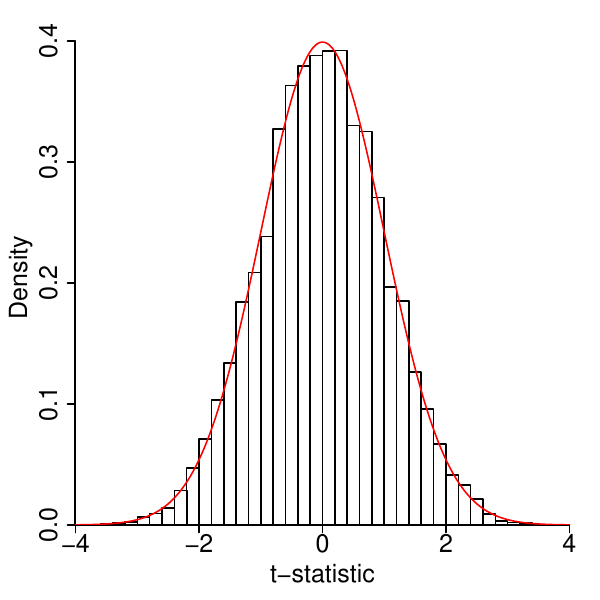}
		\caption*{case 4}
	\end{subfigure}
	\caption{
		Histograms of the $t$-statistic when both the potential outcomes and treatment assignment are randomly and independently generated from \eqref{eq:model_potential_outcome} and \eqref{eq:model_treatment}, with parameters values in the four cases of 
		Table \ref{tab:simulation}. 
		The red lines denote the density of a standard Gaussian distribution.
	}\label{fig:t_test_supp}
\end{figure}

First, we mimic superpopulation inference where the treatment assignment $\bm{Z}$ is being conditioned on.  
We fix the treatment assignment $\bm{Z}$ as in \eqref{eq:fix_treatment}, 
and generate the random potential event times $\bm{T}(1)$ and $\bm{T}(0)$ from  model \eqref{eq:model_potential_outcome}. 
The resulting distribution of the $t$-statistic 
is then 
the conditional distribution of $\TST$ given the treatment assignment $\bm{Z}$ in \eqref{eq:fix_treatment}. 
In Figure \ref{fig:t_test}, the four white histograms show the distributions of the $t$-statistic under the four cases in Table \ref{tab:simulation}. 
From Figure \ref{fig:t_test}, when there exist dependence or heterogeneity among the units' potential event times, 
the conditional distribution of the $t$-statistic $\TST$ given $\bm{Z}$ in \eqref{eq:fix_treatment} is no larger close to a standard Gaussian distribution.

Second, 
we mimic finite population inference where the potential event times are being conditioned on. 
We fix the potential event times $\bm{T}(1)$ and $\bm{T}(0)$ as in \eqref{eq:fix_potential_event_time}, 
and generate the random treatment assignment $\bm{Z}$ from the completely randomized experiment. 
The resulting distribution of the $t$-statistic 
is then 
the conditional distribution of $\TST$ given the potential event times $\bm{T}(1)=\bm{T}(0)$ in \eqref{eq:fix_potential_event_time}. 
In Figure \ref{fig:t_test}, the four gray histograms show the distributions of the $t$-statistic under the four cases in Table \ref{tab:simulation}. 
From Figure \ref{fig:t_test}, 
under all cases in Table \ref{tab:simulation} for generating the fixed potential event times in \eqref{eq:fix_potential_event_time}, the distributions of the $t$-statistic are all approximately standard Gaussian.

Third, we consider the setting where both the potential event times and the treatment assignment are random. 
Specifically, 
we generate the potential event times from model \eqref{eq:model_potential_outcome} and the treatment assignment from the completely randomized experiment, 
and they are mutually independent. 
Note that 
the distribution of the $t$-statistic here can be viewed as the conditional distribution 
of $\TST$ given $\bm{T}(1)$ and $\bm{T}(0)$ marginal over all potential event times. 
From the finite population simulation results, 
it is not surprising that the distributions of the $t$-statistic under all cases in Table \ref{tab:simulation} are close to a standard Gaussian distribution, as shown in Figure \ref{fig:t_test_supp}. 
Therefore, randomization justifies the usual two-sample $t$-test, and it does not require any distributional assumption on the potential event times.

\subsection{Additional simulation results for the logrank test}

\begin{figure}%
	\centering
	\begin{subfigure}{.4\textwidth}
		\centering
		\includegraphics[width=1\linewidth]{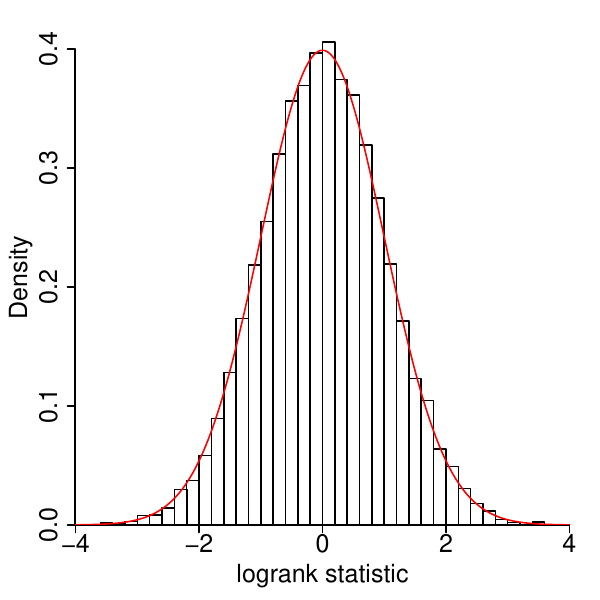}
		\caption*{case 1}
	\end{subfigure}%
	\begin{subfigure}{.4\textwidth}
		\centering
		\includegraphics[width=1\linewidth]{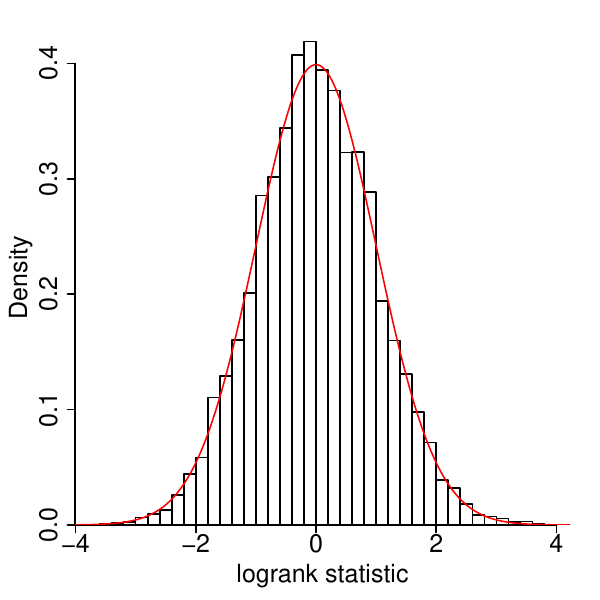}
		\caption*{case 2}
	\end{subfigure}
	\begin{subfigure}{.4\textwidth}
		\centering
		\includegraphics[width=1\linewidth]{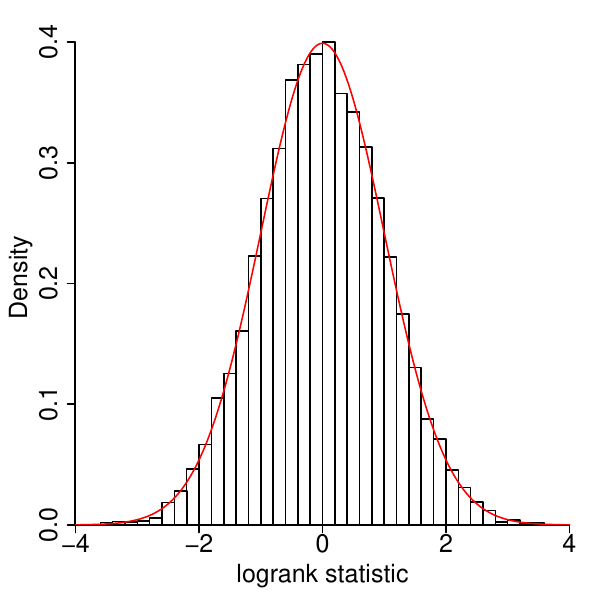}
		\caption*{case 3}
	\end{subfigure}%
	\begin{subfigure}{.4\textwidth}
		\centering
		\includegraphics[width=1\linewidth]{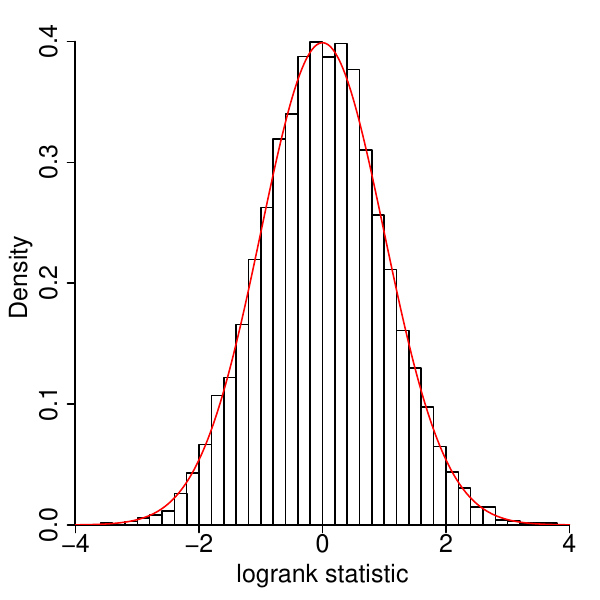}
		\caption*{case 4}
	\end{subfigure}
	\caption{
		Histograms of the logrank statistic 
		when the potential event times, potential censoring times and treatment assignment are mutually independent and randomly generated from 
		\eqref{eq:model_potential_outcome}, \eqref{eq:model_censor} and \eqref{eq:model_treatment}, respectively, with parameters values in the four cases of 
		Table \ref{tab:simulation}. 
		The red lines denote the density of a standard Gaussian distribution.
	}\label{fig:logrank_supp}
\end{figure}

\begin{figure}%
	\centering
	\begin{subfigure}{.4\textwidth}
		\centering
		\includegraphics[width=1\linewidth]{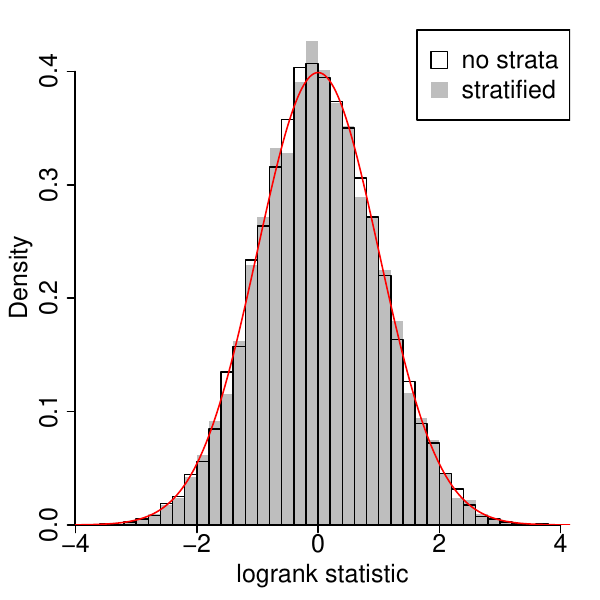}
		\caption*{case i}
	\end{subfigure}%
	\begin{subfigure}{.4\textwidth}
		\centering
		\includegraphics[width=1\linewidth]{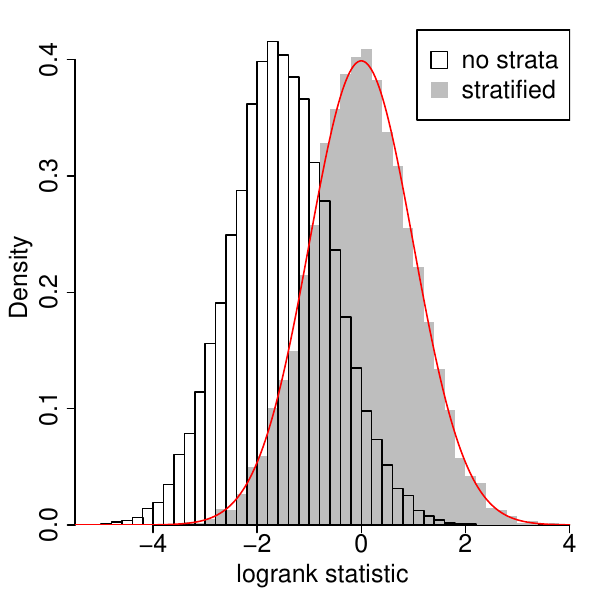}
		\caption*{case ii}
	\end{subfigure}
	\begin{subfigure}{.4\textwidth}
		\centering
		\includegraphics[width=1\linewidth]{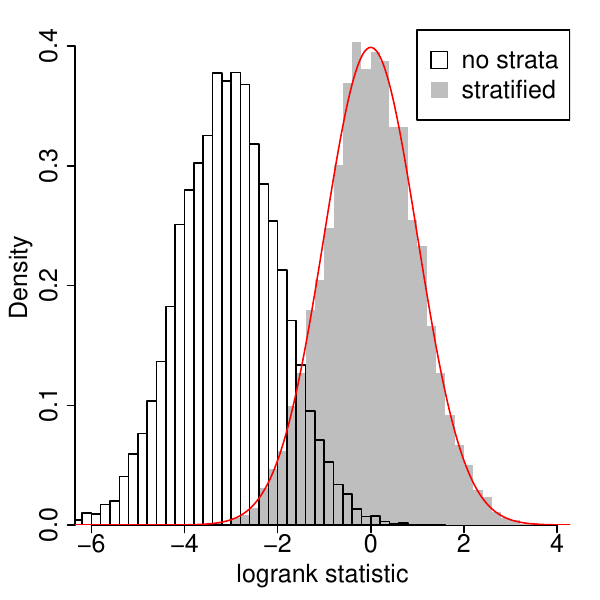}
		\caption*{case iii}
	\end{subfigure}%
	\begin{subfigure}{.4\textwidth}
		\centering
		\includegraphics[width=1\linewidth]{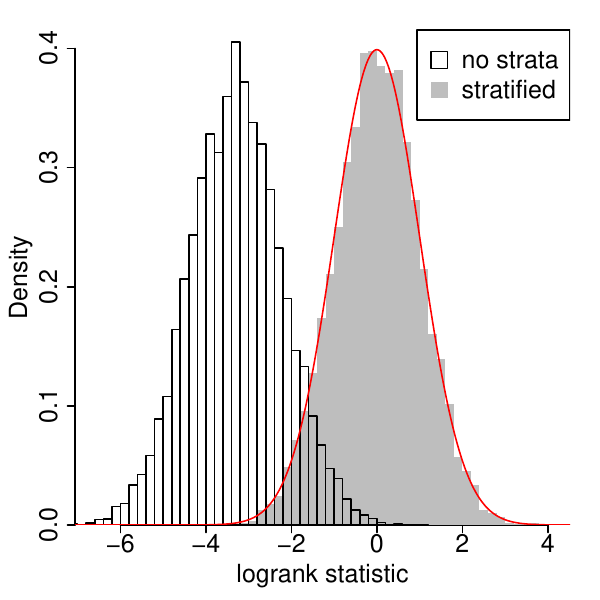}
		\caption*{case iv}
	\end{subfigure}
	\caption{
		Histograms of the logrank and stratified logrank statistics 
		under homogeneous/heterogeneous treatment assignment and censoring mechanisms as in Table \ref{tab:violation}, 
		with mutually independent potential event times, potential censoring times and treatment assignment randomly generated from \eqref{eq:model_potential_outcome}, \eqref{eq:censor_heter} and \eqref{eq:assign_heter}, respectively. 
		The white histograms denote the distributions of the logrank statistic, 
		and the gray histograms denote the distributions of the stratified logrank statistic. 
		The red lines denote the density of the standard Gaussian distribution.
	}\label{fig:logrank_violation_supp}
\end{figure}

Figure \ref{fig:logrank_supp} shows the histograms of the logrank statistic, when the potential event times $(\bm{T}(1), \bm{T}(0))$, potential censoring times $(\bm{C}(1), \bm{C}(0))$ and treatment assignment $\bm{Z}$ are mutually independent and randomly generated from 
\eqref{eq:model_potential_outcome}, \eqref{eq:model_censor} and \eqref{eq:model_treatment}, respectively. 
The parameters values for $\rho$ and $\theta$ are chosen as in the four cases in Table \ref{tab:simulation}. 
From Figure \ref{fig:logrank_supp}, 
the distributions of the logrank statistic under all four cases are close to the standard Gaussian distribution. 
This is not surprising because we can view the distribution of the logrank statistic with random potential event times and random treatment assignment as the marginalization of the randomization distribution in \eqref{eq:logrank_clt_finite_conti}, 
whose asymptotic standard Gaussianity is established in Section \ref{sec:asym_dist}  and is also demonstrated in Figure \ref{fig:logrank}.

Figure \ref{fig:logrank_violation_supp} shows the histograms of the original logrank statistic and the stratified logrank statistic based on the covariates $\bm{X}$ in \eqref{eq:model_potential_outcome}, 
when the potential event times $(\bm{T}(1), \bm{T}(0))$, potential censoring times $(\bm{C}(1), \bm{C}(0))$ and treatment assignment $\bm{Z}$ are mutually independent and randomly generated from 
\eqref{eq:model_potential_outcome} with $\rho = 0.5$ and $\theta = 1$, \eqref{eq:censor_heter} and \eqref{eq:assign_heter}, respectively. 
The values for $I_z$ and $I_c$ are chosen as in the four cases in Table \ref{tab:violation}. 
From Figure \ref{fig:logrank_violation_supp}, 
the distribution of the logrank statistic is far from the standard Gaussian distribution, when there exists heterogeneity among units' treatment assignment or censoring mechanisms. 
This indicates that Assumptions  \ref{asmp:rbe} and \ref{asmp:noninformative} are in some way necessary for the asymptotic standard Gaussianity of the logrank statistic. 
However, 
the distributions of the stratified logrank statistic under all four cases are close to the standard Gaussian distribution. 
This is not surprising because we can view the distribution of the stratified logrank statistic with random potential event times and random treatment assignment as the marginalization of the randomization distribution in Theorem \ref{thm:CLT_SLR}, 
whose asymptotic Gaussianity is established %
there and is also demonstrated in Figure \ref{fig:logrank_violation}.

\section{Exact randomization distribution of the logrank statistic}\label{app:exact}

\subsection{Lemmas}

For any integer $m\ge 0$ and constant $p \in [0,1]$, a random variable with probability mass function $m!/\{j!(m-j)!\} \cdot p^j (1-p)^{m-j}$ for $0\le j \le m$ is said to follow a binomial distribution with parameters $(m,p)$, denoted as $\Binomial(m,p)$. 
The following two lemmas study the randomization distribution of the treatment indicators for units at risk at time $t$ given observed information up to time $t$, 
and plays an important role in proving Theorems  \ref{thm:hyper_randomization} and \ref{thm:mds}. 

\begin{lemma}\label{lemma:Z_risk_event}
	Under Assumptions \ref{asmp:rbe} and \ref{asmp:noninformative} and the null hypothesis $H_0$ in \eqref{eq:H0},
	and conditioning on all the potential event times $\bm{T}(1)$ and $\bm{T}(0)$, 
	\begin{align*}
	& \quad \ \E\left\{ Z_i\I( W_i \ge t) \mid \bm{T}(1), \bm{T}(0), \I(W_i \ge t), \Delta_i  \I(W_i = t) \right\}\\
	& = \E\left\{ Z_i\I (W_i \ge t) \mid \bm{T}(1), \bm{T}(0), \I (W_i \ge t) \right\}\\
	& =  \I(W_i \ge t) \phi(t). 
	\end{align*}
\end{lemma}

\begin{lemma}\label{lemma:joint_D1_N1}
	Under Assumptions \ref{asmp:rbe} and \ref{asmp:noninformative} and the null hypothesis $H_0$ in \eqref{eq:H0}, 
	and conditioning on all the potential event times $\bm{T}(1)$ and $\bm{T}(0)$, 
	for any time $t\ge 0$, 
	\begin{align*}
	& \quad \ \pr\left\{
	\overline{D}_1(t) = A, 
	\overline{N}_1(t) = B
	\mid 
	\bm{T}(1), \bm{T}(0), 
	\Delta_i \I(W_i = t), \I(W_i \ge t), 
	1\le i \le n
	\right\}\\
	& = 
	\binom{\overline{D}(t)}{A} 
	\binom{\overline{N}(t) - \overline{D}(t)
	}{B-A} 
	\{\phi(t)\}^B \{1-\phi(t)\}^{\overline{N}(t)-B}.
	\end{align*}
\end{lemma}

The following three lemmas study the randomization distributions of  the $(N_k, D_k, N_{1k})$'s, 
and play an important role in characterizing the variance of the overall observed-expected difference $L$ in \eqref{eq:L_U}.

\begin{lemma}\label{lemma:joint_dist_D_N1_given_N}
	Under Assumptions \ref{asmp:rbe} and \ref{asmp:noninformative} and the null hypothesis $H_0$ in \eqref{eq:H0}, 
	and conditioning on all the potential event times $\bm{T}(1)$ and $\bm{T}(0)$, 
	for $1\le k \le K$, 
	conditional on $N_k$, $D_k$ and $N_{1k}$  are independent, 
	i.e., 
	$$
	N_{1k} \ind D_k \mid \bm{T}(1), \bm{T}(0),  N_k, 
	$$
	and they follow hypergeometric and  binomial distributions, respectively: 
	\begin{align*}
	D_k \mid \bm{T}(1), \bm{T}(0), N_k  & \sim \Hypergeometric(n_k, d_k, N_k), \\ 
	N_{1k} \mid \bm{T}(1), \bm{T}(0),  N_k & \sim \Binomial(N_k, \phi_k).
	\end{align*}
\end{lemma}

\begin{lemma}\label{lemma:marg_dist_D_N1_N}
	Under Assumptions \ref{asmp:rbe} and \ref{asmp:noninformative} and the null hypothesis $H_0$ in \eqref{eq:H0}, 
	and conditioning on all the potential event times $\bm{T}(1)$ and $\bm{T}(0)$, 
	for $1\le k \le K$, %
	$N_k$, $D_k$ and $N_{1k}$ all follow binomial distributions:
	\begin{align*}
	D_k \mid  \bm{T}(1), \bm{T}(0) & \sim  \Binomial(d_k, g_k), \ \ 
	N_{1k}  \mid  \bm{T}(1), \bm{T}(0) \sim \Binomial(n_k, g_k \phi_k), \\
	N_k \mid  \bm{T}(1), \bm{T}(0) & \sim \Binomial(n_k, g_k). 
	\end{align*}
\end{lemma}

\begin{lemma}\label{lemma:mean_D_N_minus_D}
	Under Assumptions \ref{asmp:rbe} and \ref{asmp:noninformative} and the null hypothesis $H_0$, 
	and conditioning on all the potential event times $\bm{T}(1)$ and $\bm{T}(0)$, 
	\begin{align}\label{eq:mean_D_N_minus_D_given_N}
	\E\left\{
	D_k (N_k - D_k) \mid \bm{T}(1), \bm{T}(0), N_k
	\right\}
	& = N_k(N_k-1) h_k(1-h_k)\frac{n_k}{n_k-1},\\
	\label{eq:mean_N1_N_minus_N1_given_N}
	\E\left\{
	N_{1k} (N_k - N_{1k}) \mid \bm{T}(1), \bm{T}(0), N_k
	\right\} & = N_k(N_k-1) \phi_k (1 - \phi_k). 
	\end{align}
\end{lemma}

\subsection{Proofs of the lemmas}

\begin{proof}[Proof of Lemma \ref{lemma:Z_risk_event}]
	Because $\Delta_i \I (W_i = t) = \I(W_i \ge t) \I(T_i(0)=t)$, $\Delta_i \I (W_i = t)$ is a constant given $\I (W_i \ge t)$. This implies 
	\begin{align}\label{eq:Z_risk_event_proof1}
		& \quad \ \E\left\{ Z_i\I(W_i \ge t) \mid  \I(W_i \ge t), \Delta_i  \I(W_i = t ) \right\} \\
		& = 
		\E \left\{ Z_i \I(W_i \ge t) \mid \I (W_i \ge t) \right\} 
		= \I(W_i \ge t) \E\left( Z_i \mid W_i \ge t \right)
		\nonumber\\
		& =  \I (W_i \ge t) 
		\frac{\pr(Z_i=1, W_i \ge t )}{\sum_{z=0}^{1} \pr(Z_i = z, W_i \ge t )}. 
		\nonumber
	\end{align}
	By definition, we can simplify \eqref{eq:Z_risk_event_proof1} as 
	\begin{align}\label{eq:Z_risk_event_proof}
	& \quad \ \E\left\{ Z_i\I(W_i \ge t) \mid \I(W_i \ge t), \Delta_i  \I(W_i = t ) \right\} 
	\\
	& = \I (W_i \ge t) 
	\frac{\pr(Z_i=1, C_i(1)\ge t, T_i(1) \ge t )}{\sum_{z=0}^{1} \pr(Z_i = z, C_i(z) \ge t, T_i(z)\ge t )} 
	\nonumber
	\\
	& = 
	 \I (W_i \ge t) 
	\frac{ 
	\I(T_i(1)\ge t) p_1 G_1(t)
	}{
	\I(T_i(1)\ge t) p_1 G_1(t) + \I(T_i(0)\ge t) p_0 G_0(t)
	}. 
	\nonumber
	\end{align}
	Because under the null hypothesis $H_0$, 
	$W_i \ge t$ implies that $T_i(1) = T_i(0) \ge t$, we can further simplify \eqref{eq:Z_risk_event_proof} as 
	\begin{align*}
		& \quad \ \E\left\{ Z_i\I(W_i \ge t) \mid \I(W_i \ge t), \Delta_i  \I(W_i = t ) \right\} \\
		& = 
		 \I (W_i \ge t) 
		\frac{ 
			p_1 G_1(t)
		}{
			p_1 G_1(t) + p_0 G_0(t)
		}\\
	& = \I (W_i\ge t) \phi(t). 
	\end{align*}
	Therefore, Lemma \ref{lemma:Z_risk_event} holds. 
\end{proof}

\begin{proof}[Proof of Lemma \ref{lemma:joint_D1_N1}]
	For any $t\ge 0$, 
	let 
	$$
	\overline{\mathcal{I}}(t) = \{i: \Delta_i\I(W_i =t) = 1, 1\le i \le n \}
	$$ 
	be the set of units having events at time $t$,  
	$$
	\overline{\mathcal{J}}(t) = \{i: 
	\I(W_i \ge t) = 1, 
	\Delta_i\I(W_i =t) = 0, 1\le i \le n \}
	$$ 
	be the set of units at risk but not having events at time $t$, 
	and $|\overline{\mathcal{I}}(t)|$ and 
	$|\overline{\mathcal{J}}(t)|$ be the cardinalities of these two sets. 
	For any $t\ge 0$, we define a $\sigma$-algebra $\mathcal{E}(t)$ as follows: 
	$$
	\mathcal{E}(t) 
	= \sigma \left\{ 
	\Delta_i \I(W_i = t), \I(W_i \ge t), 
	1\le i \le n \right\}.
	$$
	By definition, the conditional distribution of $\overline{D}_1(t)$ and $\overline{N}_1(t)$ given $\mathcal{E}(t)$  has the following equivalent forms:
	\begin{align}\label{eq:joint_D1_N1_step1}
	& \quad \ \pr\left\{
	\overline{D}_1(t) = A, 
	\overline{N}_1(t) = B
	\mid 
	\mathcal{E}(t)
	\right\}
	\\
	& = 
	\pr\left\{
	\sum_{i=1}^n Z_i \Delta_i \I(W_i = t) = A, 
	\ 
	\sum_{i=1}^n Z_i \I(W_i \ge t)  = B
	\mid 
	\mathcal{E}(t) \right\}
	\nonumber
	\\
	& = 
	\pr\left\{
	\sum_{i \in \overline{\mathcal{I}}(t)} Z_i = A, 
	\sum_{i \in \overline{\mathcal{I}}(t) \cup \overline{\mathcal{J}}(t) } Z_i  = B
	\mid 
	\mathcal{E}(t)
	\right\}
	\nonumber
	\\
	\nonumber
	&  = 
	\pr\left\{
	\sum_{i \in \overline{\mathcal{I}}(t)} Z_i = A, 
	\sum_{i \in \overline{\mathcal{J}}(t) } Z_i  = B - A
	\mid 
	\mathcal{E}(t)
	\right\}.
	\end{align}
	Under Assumptions \ref{asmp:rbe} and \ref{asmp:noninformative}, 
	$(Z_i, W_i, \Delta_i)$'s are mutually independent for $1\leq i\leq n$. 
	Thus, 
	from Lemma \ref{lemma:Z_risk_event}, 
	conditional on %
	$\mathcal{E}(t)$, 
	$\sum_{i \in \overline{\mathcal{I}}(t)} Z_i$ follows a binomial distribution with parameters $(|\overline{\mathcal{I}}(t)|, \phi(t))$, 
	$\sum_{i \in \overline{\mathcal{J}}(t) } Z_i$ follows a binomial distribution with parameters $(|\overline{\mathcal{J}}(t)|, \phi(t))$, and they are independent. 
	Therefore, the conditional distribution \eqref{eq:joint_D1_N1_step1} has the following equivalent forms: 
	\begin{align*}
	& \quad \ \pr\left\{
	\overline{D}_1(t) = A, 
	\overline{N}_1(t) = B
	\mid 
	\mathcal{E}(t)
	\right\}
	\nonumber
	\\
	& = 
	\binom{|\overline{\mathcal{I}}(t)|}{A} \{\phi(t)\}^{A} \{1 - \phi(t)\}^{|\overline{\mathcal{I}}(t)|- A}
	\binom{|\overline{\mathcal{J}}(t)|}{B-A} \{\phi(t)\}^{B-A} \{1 - \phi(t)\}^{|\overline{\mathcal{J}}(t)|- (B-A)}
	\nonumber
	\\
	& = 
	\binom{|\overline{\mathcal{I}}(t)|}{A} 
	\binom{|\overline{\mathcal{J}}(t)|}{B-A}
	\{\phi(t)\}^{B} 
	\{1 - \phi(t)\}^{|\overline{\mathcal{I}}(t)|+ |\overline{\mathcal{J}}(t)|- B}. 
	\end{align*}
	By definition, 
	$|\overline{\mathcal{I}}(t)| = \overline{D}(t)$, 
	$|\overline{\mathcal{J}}(t)| = \overline{N}(t) - \overline{D}(t),$ 
	and thus Lemma \ref{lemma:joint_D1_N1} holds. 
\end{proof}

\begin{proof}[Proof of Lemma \ref{lemma:joint_dist_D_N1_given_N}]
	First, we derive the conditional distribution of $N_{1k}$ given $N_k$, and show that $N_{1k}$ and $D_k$ are conditionally independent given $N_k$. 
	For any $1\le k \le K$, from Lemma \ref{lemma:Z_risk_event} and the mutual independence of the $(Z_i, W_i, \Delta_i)$'s under Assumptions \ref{asmp:rbe} and \ref{asmp:noninformative}, we can know that 
	\begin{align}\label{eq:N1_given_indicators}
	& \quad \ \sum_{i=1}^{n} Z_i \I (W_i \ge t_k) \mid \I(W_j \ge t_k), \Delta_j \I(W_j = t_k), 1\le j \le n \\
	& \sim \ 
	\Binomial\left(
	N_k, \phi_k
	\right). 
	\nonumber
	\end{align}
	By definition, 
	both $N_k$ and $D_k$ are deterministic functions of the 
	$\I(W_j \ge t_k)$'s and $\Delta_j \I(W_j = t_k)$'s, respectively. 
	Thus, from \eqref{eq:N1_given_indicators}, conditional on $N_k$,  $N_{1k}$ follows a binomial distribution with parameters $(N_k, \phi_k)$, and it is independent of $D_k$, i.e., 
	$$
	N_{1k} \mid N_k \sim \Binomial(N_k, \phi_k), 
	\ \ \text{ and } \ \ 
	N_{1k} \ind D_k \mid N_k. 
	$$
	
	Second, we study the distribution of $(D_k,N_k)$, and derive the conditional distribution of $D_k$ given $N_k$. 
	Under $H_0$, for any unit $i$ with $\Delta_i \I (W_i =  t_k) =1$, we must have $T_i(1) = T_i(0)  = t_k$. 
	Similarly, for any unit $i$ with $\I(W_i \ge t_k)=1$, we must have $T_i(1) = T_i(0) \ge t_k$. Thus, by definition, $D_k \le d_k$ and $D_k \le N_k \le n_k$. 
	For any constant nonnegative integers $A, B$ satisfying $A\le d_k$ and $A\le B \le n_k$, the  probability that $D_k = A$ and $N_k = B$ has the following equivalent forms:
	\begin{align}\label{eq:joint_D_N_sum}
	& \quad \ \pr\left(
	D_k = A, N_k = B
	\right) \\
	& = \pr\left(
	\sum_{i=1}^{n} \Delta_i \I (W_i =  t_k) = A, \
	\sum_{i=1}^{n}\I (W_i \ge t_k) = B
	\right)
	\nonumber
	\\
	& = 
	\pr\left(
	\sum_{i: T_i(0) = t_k} \Delta_i \I (W_i =  t_k) = A, 
	\sum_{i: T_i(0)\ge t_k}\I(W_i \ge t_k) = B
	\right). 
	\nonumber
	\end{align}
	Under $H_0$, for any unit $i$ with $T_i(0) = t_k$, we must have $\Delta_i \I\{W_i = t_k\} =  \I \{W_i \ge  t_k\}$. Thus, \eqref{eq:joint_D_N_sum}  reduces to 
	\begin{align}\label{eq:joint_D_N_two_indep_sum}
	& \quad\ \pr\left(
	D_k = A, N_k = B
	\right) \\
	& = 
	\pr\left(
	\sum_{i: T_i(0) = t_k} \I(W_i \ge t_k) = A, 
	\sum_{i: T_i(0) \ge t_k}\I(W_i \ge t_k) = B
	\right)
	\nonumber
	\\
	& = \pr\left(
	\sum_{i: T_i(0) = t_k} \I(W_i \ge t_k) = A, 
	\sum_{i: T_i(0)> t_k}\I(W_i \ge t_k) = B-A
	\right).
	\nonumber
	\end{align}
	Because (i) for any unit $i$ with $T_i(1) = T_i(0) \ge t_k$, 
	$
	\pr (W_i \ge t_k) 
	= 
	\pr\left(
	C_i \ge t_k
	\right) = G(t_k) = g_k,
	$
	and 
	(ii) $W_i$'s are mutually independent for $1\le i\le n$, 
	we can know that 
	\begin{align}\label{eq:dist_Dk}
	D_k & = \sum_{i: T_i(0) = t_k} \I(W_i  \ge  t_k) \sim \Binomial(d_k, g_k), 
	\\
	N_k - D_k & =  \sum_{i: T_i(0) > t_k}\I(W_i \ge t_k) \sim \Binomial(n_k-d_k, g_k), 
	\nonumber
	\end{align}
	and they are mutually independent. 
	Thus, we can further simplify \eqref{eq:joint_D_N_two_indep_sum} as
	\begin{align*}
	& \quad \ \pr\left(
	D_k = A, N_k = B
	\right) \\
	& = 
	\binom{
		d_k
	}{A}
	g_k^A (1-g_k)^{d_k - A}
	\binom{n_k - d_k}{B-A} 
	g_k^{B-A} (1-g_k)^{n_k - d_k - (B-A)}\\
	& = \binom{d_k}{A}
	\binom{n_k - d_k}{B-A} 
	g_k^B (1-g_k)^{n_k- B}. 
	\end{align*}
	This immediately implies the marginal distribution of $N_k$:   
	\begin{align}\label{eq:dist_Nk}
	& \quad \ \pr\left(
	N_k = B
	\right) \\
	& = 
	\sum_{A}
	\pr\left(
	D_k = A, N_k = B
	\right) 
	= 
	\sum_{A}
	\binom{d_k}{A}
	\binom{n_k - d_k}{B-A} g_k^B (1-g_k)^{n_k- B}
	\nonumber
	\\
	& = 
	\binom{n_k}{B} g_k^B (1-g_k)^{n_k- B}, 
	\nonumber
	\end{align}
	i.e., $N_k$ follows a binomial distribution with parameters $(n_k, g_k)$. 
	Therefore, the conditional distribution of $D_k$ given $N_k$ has the following equivalent forms:
	\begin{align*}
	\pr\left(
	D_k = A \mid N_k = B
	\right) 
	& = 
	\frac{
		\pr(
		D_k = A, N_k = B
		) 
	}{
		\pr(
		N_k = B
		) 
	}
	= \frac{
		\binom{d_k}{A}
		\binom{n_k - d_k}{B-A} 
		g_k^B (1-g_k)^{n_k- B}
	}{
		\binom{n_k}{B} g_k^B (1-g_k)^{n_k- B}
	}
	\\
	& = 
	\frac{
		\binom{d_k}{A}
		\binom{n_k - d_k}{B-A} 
	}{
		\binom{n_k}{B}
	},
	\end{align*}
	i.e., conditional on $N_k$, $D_k$ follows a hypergeometric distribution with parameters $(n_k, d_k, N_k)$. 
	
	From the above, Lemma \ref{lemma:joint_dist_D_N1_given_N} holds.
\end{proof}

\begin{proof}[Proof of Lemma \ref{lemma:marg_dist_D_N1_N}]
	First, from \eqref{eq:dist_Dk}, $D_k$ follows a binomial distribution with parameters $(d_k, g_k)$. 
	Second, from \eqref{eq:dist_Nk}, $N_k$ follows a binomial distribution with parameters $(n_k, g_k)$. 
	Third, because (i) $N_k\sim \Binomial(n_k, g_k)$ and (ii) $N_{1k} \mid N_k \sim \Binomial(N_k, \phi_k)$
	from Lemma \ref{lemma:joint_dist_D_N1_given_N}, 
	the property of binomial distributions implies that $N_{1k} \sim \Binomial(n_k, g_k \phi_k)$. 
	From the above, Lemma \ref{lemma:marg_dist_D_N1_N} holds. 
\end{proof}

\begin{proof}[Proof of Lemma \ref{lemma:mean_D_N_minus_D}]
	First, we prove \eqref{eq:mean_D_N_minus_D_given_N}. 
	The left hand side of \eqref{eq:mean_D_N_minus_D_given_N} has the following equivalent forms:
	\begin{align}\label{eq:mean_D_N_minus_D_mean_var}
	& \quad \ \E\left\{
	D_k (N_k - D_k) \mid N_k
	\right\}\\
	& = 
	N_k \E\left(
	D_k \mid N_k
	\right) - 
	\E\left(
	D_k^2 \mid N_k
	\right)
	\nonumber
	\\
	& = 
	N_k \E\left(
	D_k \mid N_k
	\right) - 
	\left\{\E\left(
	D_k \mid N_k
	\right)\right\}^2 - \Var\left(
	D_k \mid N_k
	\right)
	\nonumber
	\\
	& = 
	\left\{
	N_k - \E\left(
	D_k \mid N_k
	\right)
	\right\} \E\left(
	D_k \mid N_k
	\right) - 
	\Var\left(
	D_k \mid N_k
	\right).
	\nonumber
	\end{align}
	From Lemma \ref{lemma:joint_dist_D_N1_given_N} and the property of hypergeometric distribution, we have 
	\begin{align*}
	\E\left\{
	D_k (N_k - D_k) \mid N_k
	\right\}
	& = 
	\left(
	N_k - N_k h_k
	\right) N_k h_k - 
	\frac{
		d_k (n_k - d_k) N_k (n_k - N_k)
	}{n_k^2(n_k-1)}\\
	& = 
	N_k^2 
	(
	1 - h_k
	) h_k - 
	N_k h_k (1-h_k)
	\frac{
		n_k - N_k
	}{n_k-1}\\
	& = 
	N_k (N_k-1) h_k (1 - h_k) \frac{n_k}{n_k-1}.
	\end{align*}
	
	Second, we prove \eqref{eq:mean_N1_N_minus_N1_given_N}. 
	By the same logic as \eqref{eq:mean_D_N_minus_D_mean_var}, 
	\begin{align}\label{eq:mean_N1_N_minus_N1_mean_var}
	& \quad \ \E\left\{
	N_{1k} (N_k - N_{1k}) \mid N_k
	\right\}\\
	& = 
	\left\{
	N_k - \E\left(
	N_{1k} \mid N_k
	\right)
	\right\} \E\left(
	N_{1k} \mid N_k
	\right) - 
	\Var\left(
	N_{1k} \mid N_k
	\right).
	\nonumber
	\end{align}
	From Lemma \ref{lemma:joint_dist_D_N1_given_N} and 
	the property of binomial distribution, we have
	\begin{align*}
	\E\left\{
	N_{1k} (N_k - N_{1k}) \mid N_k
	\right\}
	& = 
	(N_k - N_k \phi_k) N_k \phi_k - N_k \phi_k (1-\phi_k)
	\\
	& = 
	N_k^2 (1 - \phi_k) \phi_k - N_k \phi_k (1-\phi_k)\\
	& = 
	N_k (N_k - 1) \phi_k (1 - \phi_k).
	\end{align*}
	
	From the above, Lemma \ref{lemma:mean_D_N_minus_D} holds. 
\end{proof}

\subsection{Proofs of the theorems}

\begin{proof}[Proof of Theorem \ref{thm:hyper_randomization}]
	From Lemma \ref{lemma:joint_D1_N1}, 
	\begin{align*}
	& \ \quad \pr\left\{
	\overline{N}_1(t) = B
	\mid 
	\Delta_i \I(W_i = t), \I(W_i \ge t), 
	1\le i \le n
	\right\}\\
	& = 
	\sum_{A}
	\pr\left\{
	\overline{D}_1(t) = A, 
	\overline{N}_1(t) = B
	\mid 
	\Delta_i \I(W_i = t), \I(W_i \ge t), 
	1\le i \le n
	\right\}\\
	& = \sum_A \binom{\overline{D}(t)}{A} 
	\binom{\overline{N}(t) - \overline{D}(t)
	}{B-A} 
	\{\phi(t)\}^B \{1-\phi(t)\}^{\overline{N}(t)-B}\\
	& = \binom{\overline{N}(t)}{B} \{\phi(t)\}^B \{1-\phi(t)\}^{\overline{N}(t)-B},
	\end{align*}
	and thus
	the conditional distribution of $\overline{D}_1(t)$ given $\overline{N}_1(t)$, $\Delta_i \I(W_i = t)$'s, and $\I(W_i \ge t)$'s is
	\begin{align}\label{eq:cond_dist_D1_given_N1_indicators}
	& \quad \ \pr\left\{
	\overline{D}_1(t) = A
	\mid 
	\overline{N}_1(t) = B, 
	\Delta_i \I(W_i = t), \I(W_i \ge t), 
	1\le i \le n
	\right\}
	\\
	& = 
	\frac{\pr\left\{
		\overline{D}_1(t) = A, \overline{N}_1(t) = B
		\mid 
		\Delta_i \I(W_i = t), \I(W_i \ge t), 
		1\le i \le n
		\right\}}{
		\pr\left\{
		\overline{N}_1(t) = B
		\mid 
		\Delta_i \I(W_i = t), \I(W_i \ge t), 
		1\le i \le n
		\right\}
	}
	\nonumber
	\\
	& = 
	\frac{\binom{\overline{D}(t)}{A} 
		\binom{\overline{N}(t) - \overline{D}(t)
		}{B-A} 
		\{\phi(t)\}^B \{1-\phi(t)\}^{\overline{N}(t)-B}}{
		\binom{\overline{N}(t)}{B} \{\phi(t)\}^B \{1-\phi(t)\}^{\overline{N}(t)-B}
	}
	\nonumber
	\\
	& = \binom{\overline{D}(t)}{A} 
	\binom{\overline{N}(t) - \overline{D}(t)
	}{B-A} /
	\binom{\overline{N}(t)}{B}.
	\nonumber
	\end{align}
	From \eqref{eq:cond_dist_D1_given_N1_indicators}, by the law of iterated expectation, 
	we have
	\begin{align*}
	& \quad \ \pr\left\{
	\overline{D}_1(t) = A
	\mid 
	\overline{N}_1(t), 
	\overline{N}(t), \overline{D}(t)
	\right\}\\
	& = 
	\E\left[
	\pr\left\{
	\overline{D}_1(t) = A
	\mid 
	\overline{N}_1(t), 
	\Delta_i \I(W_i = t), \I(W_i \ge t), 
	\forall i
	\right\}
	\mid \overline{N}_1(t), 
	\overline{N}(t), \overline{D}(t)
	\right]\\
	& = 
	\binom{\overline{D}(t)}{A} 
	\binom{\overline{N}(t) - \overline{D}(t)
	}{\overline{N}_1(t)-A} /
	\binom{\overline{N}(t)}{\overline{N}_1(t)}. 
	\end{align*}
	Therefore, Theorem \ref{thm:hyper_randomization} holds. 
\end{proof}

\begin{proof}[Proof of Theorem \ref{thm:mds}]
	First, we show that conditional on $\mathcal{F}_{k-1}$, $D_{1k}$ follows a hypergeometric distribution with parameters $(N_{k},  D_k, N_{1k})$. 
	For $1\leq k\leq K$, define 
	a $\sigma$-algebra $\mathcal{H}_{k-1}$ as 
	\begin{align*}
		\mathcal{H}_{k-1} & = 
		\sigma \left\{
		\I(W_i \geq t_q), 
		\Delta_i\I(W_i =t_q), 
		Z_i \Delta_i \I(W_i = t_{q-1}), 
		\vphantom{\sum_{j=1}^{n} Z_j \I(t_{q} \le W_j < t_k)}
		\right.\\
		& \left. 
		\qquad  \qquad \qquad \qquad \qquad
		\sum_{j=1}^{n} Z_j \I(t_{q} \le W_j < t_k), \  1\leq q \leq k, 1\leq i\leq n
		\right\}, 
	\end{align*}
	where $\I(t_{k} \le W_j < t_k)$ is defined to be constant zero. 
	The set of units having events at time $t_k$ is $\mathcal{I}_k = \{i: \Delta_i\I(W_i =t_k) = 1, 1\le i \le n \}$, 
	and the set of units at risk but not having events at time $t_k$ is 
	$
	\mathcal{J}_k = \{i: \I(W_i \ge t_k) = 1, \Delta_i\I(W_i =t_k) = 0, 1\le i \le n \}. 
	$
	By definition, the cardinalities of sets $\mathcal{I}_k$ and $\mathcal{J}_k$ are $D_k$ and $N_k - D_k.$ 
	We can then divide the information in $\mathcal{H}_{k-1}$ into three parts for disjoint sets of units: 
	(i) $\mathcal{I}_k$, (ii) $\mathcal{J}_k$, and (iii)
	\begin{align*}
		\I(W_i \geq t_q), 
		\Delta_i\I(W_i =t_q), 
		Z_i \Delta_i \I(W_i = t_{q-1}), 
		\sum_{j \in \mathcal{I}_k^c\cap \mathcal{J}_k^c} Z_j \I(t_{q} \le W_j < t_k), \\
		 1\leq q \leq k, i \in \mathcal{I}_k^c\cap \mathcal{J}_k^c.
	\end{align*}
	Under Assumptions \ref{asmp:rbe} and \ref{asmp:noninformative}, $(Z_i, W_i, \Delta_i)$'s are mutually independent for $1\leq i\leq n$. 
	Thus, 
	the conditional distribution of $(D_{1k}, N_{1k})$ given $\mathcal{H}_{K-1}$ has the following equivalent forms: 
	\begin{align}\label{eq:D1_N1_given_H}
	& \quad \ \pr\left(
	D_{1k} = A, N_{1k} = B \mid \mathcal{H}_{k-1}
	\right) 
	\\
	& = \pr\left(
	D_{1k} = A, N_{1k} = B \mid \mathcal{I}_k, \mathcal{J}_k
	\right)
	\nonumber
	\\
	& = 
	\pr\left\{
	D_{1k} = A, N_{1k} = B \mid 
	\Delta_i \I(W_i = t_k), \I(W_i \ge t_k), 1\le i \le n
	\right\}.
	\nonumber
	\end{align}
	From Lemma \ref{lemma:joint_D1_N1}, 
	we can further simplify \eqref{eq:D1_N1_given_H} as
	\begin{align}\label{eq:joint_D1_N1_given_H}
	& \quad \ \pr\left(
	D_{1k} = A, N_{1k} = B \mid \mathcal{H}_{k-1}
	\right)
	\\
	& = 
	\pr\left\{
	\overline{D}_{1}(t_k) = A, \overline{N}_{1}(t_k) = B \mid 
	\Delta_i \I(W_i = t_k), \I(W_i \ge t_k), 
	\forall i
	\right\}
	\nonumber
	\\
	& = 
	\binom{D_k}{A} 
	\binom{N_k - D_k
	}{B-A} 
	\phi_k^B (1-\phi_k)^{N_k-B}.
	\nonumber
	\end{align}
	This implies that the marginal distribution of $N_{1k}$ given $\mathcal{H}_{k-1}$ is 
	\begin{align}\label{eq:marg_N1_given_H}
	& \quad \ \pr\left(
	N_{1k} = B \mid \mathcal{H}_{k-1}
	\right)
	\\
	& = 
	\sum_{A} \pr\left(
	D_{1k} = A, N_{1k} = B \mid \mathcal{H}_{k-1}
	\right)
	\nonumber
	\\
	& = \binom{N_k}{B} \phi_k^B (1-\phi_k)^{N_k-B}. 
	\nonumber
	\end{align}
	From \eqref{eq:joint_D1_N1_given_H} and \eqref{eq:marg_N1_given_H}, the conditional distribution of $D_{1k}$ given $N_{1k}$ and $\mathcal{H}_{k-1}$ is  
	\begin{align*}
	\pr\left(
	D_{1k} = A \mid N_{1k}, \mathcal{H}_{k-1}
	\right)
	= 
	\frac{\binom{D_k}{A} \binom{N_k - D_k }{N_{1k}-A} }{
		\binom{N_k}{N_{1k}}}.
	\end{align*}
	This implies that, conditional on $\mathcal{H}_{k-1}$ and $N_{1k}$, $D_{1k}$ follows a hypergeometric distribution with parameters $(N_k, D_k, N_{1k})$. 
	Because for $1\le q\le k$, 
	\begin{align*}
	\sum_{j=1}^{n} Z_j \I(t_{q} \le W_j < t_k) + N_{1k} & = 
	\sum_{j=1}^{n} Z_j \I(t_{q} \le W_j < t_k) + \sum_{j=1}^{n} Z_j  \I(W_j \ge t_k) \\
	& = 
	\sum_{j=1}^{n} Z_j \I(W_j \ge t_{q}), 
	\end{align*}
	by definition, $(\mathcal{H}_{k-1}, N_{1k})$ contains the same information as $\mathcal{F}_{k-1}$. 
	Thus, 
	the conditional distribution of $D_{1k}$ given $\mathcal{H}_{k-1}$ and $N_{1k}$ is the same as that given $\mathcal{F}_{k-1}$, i.e., conditional on $\mathcal{F}_{k-1}$, $D_{1k}$ follows a hypergeometric distribution with parameters $(N_k, D_k, N_{1k})$. 
	
	Second, we show that $\{D_{1k}-M_{k}\}_{k=1}^K$ is a martingale difference sequence with respect to filtration $\mathcal{F}_k$'s. 
	By construction, $\mathcal{F}_k$'s constitute a filtration, and $D_{1k}$ is $\mathcal{F}_k$-measurable. 
	From the first part, the property of hypergeometric distribution, and the definition of $M_k$,  
	$\E(D_{1k} \mid \mathcal{F}_{k-1}) = \ED_k$. 
	Therefore, $\{D_{1k}-M_{k}\}_{k=1}^K$ is a martingale difference sequence with respect to filtration $\mathcal{F}_k$'s.
	
	From the above, Theorem \ref{thm:mds} holds. 
\end{proof}

\begin{proof}[Proof of Theorem \ref{thm:mean_var_L}]
	Theorem \ref{thm:mds} immediately implies that $\E(L) = 0$. 
	From Lemmas \ref{lemma:joint_dist_D_N1_given_N} and  \ref{lemma:mean_D_N_minus_D}, 
	\begin{align*}
	& \quad \ \E\left\{
	D_k(N_k-D_k) N_{1k} (N_k - N_{1k}) \mid N_k 
	\right\}\\
	& = 
	\E\left\{
	D_k(N_k-D_k) \mid N_k 
	\right\}
	\cdot
	\E\left\{
	N_{1k} (N_k - N_{1k}) \mid N_k 
	\right\}\\
	& = N_k^2 (N_k-1)^2 h_k(1-h_k) \phi_k (1 - \phi_k)\frac{n_k}{n_k-1}.
	\end{align*}
	This then implies
	\begin{align*}
	\E(V_k \mid N_k) & = 
	\E\left\{
	\frac{D_k(N_k-D_k) N_{1k} (N_k - N_{1k})}{N_k^2(N_k-1)} \mid N_k
	\right\}
	\\
	& 
	= 
	\frac{N_k^2 (N_k-1)^2 h_k(1-h_k) \phi_k (1 - \phi_k)}{N_k^2(N_k-1)} \frac{n_k}{n_k-1} \\
	& = 
	(N_k-1 + \I\{N_k =0 \}) h_k(1-h_k) \phi_k (1 - \phi_k)\frac{n_k}{n_k-1}. 
	\end{align*}
	From Lemma \ref{lemma:marg_dist_D_N1_N} and the law of iterated expectation, we have
	\begin{align}\label{eq:mean_Vk}
	\E(V_k) 
	& = \E\left\{ \E(V_k \mid N_k)  \right\} \\
	& = 
	\left\{
	\E(N_k)-1 + \pr( N_k =0 )
	\right\} h_k(1-h_k) \phi_k (1 - \phi_k)\frac{n_k}{n_k-1}
	\nonumber
	\\
	& = 
	\left\{
	n_k g_k -1 + (1-g_k)^{n_k}
	\right\} h_k(1-h_k) \phi_k (1 - \phi_k)\frac{n_k}{n_k-1}
	\nonumber
	\\
	& = 
	\frac{n_k^2}{n_k-1}  h_k (1-h_k) \phi_k (1-\phi_k) 
	\left\{
	g_k - \frac{1 - (1-g_k)^{n_k}}{n_k} 
	\right\}.
	\nonumber
	\end{align}
	From Theorem \ref{thm:mds} and  the property of the martingale difference sequence, 
	\begin{align*}
	\Var(L) & = \E(U) = 
	\E\left(
	\sum_{k=1}^{K} V_k
	\right)
	= 
	\sum_{k=1}^{K} \E(V_k)\\
	& = \sum_{k=1}^{K} \frac{n_k^2}{n_k-1}  h_k (1-h_k) \phi_k (1-\phi_k) 
	\left\{
	g_k - \frac{1 - (1-g_k)^{n_k}}{n_k} 
	\right\}. 
	\end{align*}
	Therefore, Theorem \ref{thm:mean_var_L} holds. 
\end{proof}

\section{Asymptotic distribution of the logrank statistic with infinitely many distinct potential event times}\label{app:asym_infinite}

\subsection{Lemmas}

To prove Theorem \ref{thm:CLT_conti}, we need the following six lemmas. 
Lemma \ref{lemma:martingale_CLT} gives a sufficient condition for the asymptotic Gaussianity of the logrank statistic using the martingale central limit theorem. 
It shows that the consistency of $U$ for the variance of $L$ is crucial for establishing the central limit theorem. 
To prove that 
conditioning on all the potential event times, 
$U/\E(U \mid \bm{T}(1), \bm{T}(0)) \mid \bm{T}(1), \bm{T}(0) \convergep 1$, or equivalently, 
$$
\frac{U-\E(U\mid \bm{T}(1), \bm{T}(0))}{
	\E(U\mid \bm{T}(1), \bm{T}(0))
} \mid \bm{T}(1), \bm{T}(0)  \convergep 0,
$$
we need to bound the denominator $\E(U \mid \bm{T}(1), \bm{T}(0) )$ from below, and bound the absolute value of the numerator, $|U-\E(U\mid \bm{T}(1), \bm{T}(0))|,$ from above. 
Lemma \ref{lemma:bound_EU} gives a lower bound of $\E(U\mid \bm{T}(1), \bm{T}(0))$. 
To bound $U-\E(U\mid \bm{T}(1), \bm{T}(0))$, 
for $1\le k \le K$, 
we decompose $V_k  - \E(V_k\mid \bm{T}(1), \bm{T}(0))$ into 
four parts: 
$$
V_k  - \E(V_k \mid \bm{T}(1), \bm{T}(0) ) = \xi_{1k} + \xi_{2k} + \xi_{3k} + \xi_{4k}, 
$$
where
\begin{align*}
\xi_{1k} & \equiv \frac{D_k(N_k-D_k) }{N_k^2 (N_k-1)}
\left[ N_{1k} (N_k - N_{1k})
- \E\{N_{1k} (N_k - N_{1k})\mid \bm{T}(1), \bm{T}(0), N_k\}
\right], 
\\[12pt]
\xi_{2k} & \equiv \frac{
	D_k  
	\E\{N_{1k} (N_k - N_{1k})\mid \bm{T}(1), \bm{T}(0), N_k\}
}{
	N_k^2 (N_k-1)
} 
\\
& %
\quad \ 
\times
\left\{(N_k-D_k)
-
 \E(N_k-D_k \mid \bm{T}(1), \bm{T}(0), N_k )
\right\}, \\[12pt]
\xi_{3k} & \equiv \frac{
	D_k  \E( N_k-D_k \mid  \bm{T}(1), \bm{T}(0), N_k )	\E\{N_{1k} (N_k - N_{1k})\mid  \bm{T}(1), \bm{T}(0), N_k\}
}{
	N_k^2 (N_k-1)
} \\
&  \quad \ 
- 
\E(D_k \mid \bm{T}(1), \bm{T}(0))  (1-h_k) \phi_k(1-\phi_k),\\[12pt]
\xi_{4k} & \equiv \E(D_k \mid \bm{T}(1), \bm{T}(0))  (1-h_k) \phi_k(1-\phi_k) - \E(V_k \mid \bm{T}(1), \bm{T}(0)).
\end{align*}
Because $D_K = N_K$ by definition, we have $V_K = \E(V_K) =  0$, and 
\begin{align}\label{eq:decompose_dev_U}
& \quad \ U - \E(U \mid \bm{T}(1), \bm{T}(0)) \\
& = \sum_{k=1}^{K-1} \left\{V_k  - \E(V_k\mid \bm{T}(1), \bm{T}(0)) \right\} 
\nonumber
\\
& = 
\sum_{k=1}^{K-1} \xi_{1k} + \sum_{k=1}^{K-1} \xi_{2k} + \sum_{k=1}^{K-1} \xi_{3k} + \sum_{k=1}^{K-1} \xi_{4k}. 
\nonumber
\end{align}
Lemmas \ref{lemma:dev_U_first}--\ref{lemma:dev_U_fourth} give upper bounds for the four terms in  \eqref{eq:decompose_dev_U}, separately. These further provide a probability upper bound for $U - \E(U \mid \bm{T}(1), \bm{T}(0))$. 

\begin{lemma}\label{lemma:martingale_CLT}
	Under Assumptions \ref{asmp:rbe} and \ref{asmp:noninformative}, 
	consider a sequence of finite populations satisfying the null hypothesis $H_0$, 
	If, conditioning on all the potential event times,  (i) the conditional Lindeberg condition holds, i.e., for all $\varepsilon >0,$
	\begin{small}
		\begin{align*}
		& \sum_{k=1}^{K} \E\left[
		\frac{\left(D_{1k} - \ED_{k}\right)^2}{\Var(L\mid \bm{T}(1), \bm{T}(0))} 
		\I\left\{
		\frac{|D_{1k}-\ED_{k}|}{\sqrt{\Var(L\mid \bm{T}(1), \bm{T}(0))}} > \varepsilon
		\right\}
		\mid 
		\bm{T}(1), \bm{T}(0), 
		\mathcal{F}_{k-1}
		\right] \convergep 0,
		\end{align*}
	\end{small}%
	and (ii) the summation of the conditional variances is consistent for its mean, i.e., 
	\begin{align*}
	\frac{\sum_{k=1}^{K}V_k}{\Var(L\mid \bm{T}(1), \bm{T}(0) )} = 
	\frac{\sum_{k=1}^{K}V_k}{\sum_{k=1}^{K}\E(V_k \mid \bm{T}(1), \bm{T}(0)  )} = \frac{U}{\E(U\mid \bm{T}(1), \bm{T}(0) )} \convergep 1,
	\end{align*}
	then the logrank statistic is asymptotically standard Gaussian, i.e., 
	$$
	U^{-1/2} L \mid \bm{T}(1), \bm{T}(0)  \converged \mathcal{N}(0,1). 
	$$
\end{lemma}

\begin{lemma}\label{lemma:bound_EU}
	Under Assumptions \ref{asmp:rbe} and \ref{asmp:noninformative}, and conditioning on all the potential event times, 
	the expectation of $U \equiv \sum_{k=1}^K V_k$ has the following lower bound: 
	\begin{align*}
	\E(U \mid \bm{T}(1), \bm{T}(0) ) & \ge 
	p_1 p_0 \left( \frac{n}{\tilde{d}+1} \frac{1}{n}\sum_{i=1}^{n}
	\tilde{G}(T_i(0)) - 1 \right)
	= p_1 p_0 \left( \frac{n\tilde{g}}{\tilde{d}+1}  - 1 \right).
	\end{align*}
\end{lemma}

\begin{lemma}\label{lemma:dev_U_first}
	Under Assumptions \ref{asmp:rbe} and \ref{asmp:noninformative} and the null hypothesis $H_0$, 
	and conditioning on all the potential event times, 
	\begin{align*}
		\sum_{k=1}^{K-1} \xi_{1k} & \equiv
		\sum_{k=1}^{K-1} \frac{D_k(N_k-D_k) }{N_k^2 (N_k-1)} 
		\left[ N_{1k} (N_k - N_{1k})
		- \E\{N_{1k} (N_k - N_{1k})\mid \bm{T}(1), \bm{T}(0), N_k\}
		\right]\\
		& = \left( n\log n \right)^{1/2} O_\pr (1).
	\end{align*}
\end{lemma}

\begin{lemma}\label{lemma:dev_U_second}
	Under Assumptions \ref{asmp:rbe} and \ref{asmp:noninformative} and the null hypothesis $H_0$, 
	and conditioning on all the potential event times, 
	\begin{align*}
		\sum_{k=1}^{K-1} \xi_{2k} & \equiv 
		\sum_{k=1}^{K-1} \left[ \frac{
			D_k  
			\E\{N_{1k} (N_k - N_{1k})\mid \bm{T}(1), \bm{T}(0), N_k\}
		}{
			N_k^2 (N_k-1)
		} \right.
		\\
		& %
		\quad \ 
		\left. 
		\vphantom{\frac{
				D_k  
				\E\{N_{1k} (N_k - N_{1k})\mid \bm{T}(1), \bm{T}(0), N_k\}
			}{
				N_k^2 (N_k-1)
		}}
		\qquad \ 
		\times
		\left\{(N_k-D_k)
		-
		\E(N_k-D_k \mid \bm{T}(1), \bm{T}(0), N_k )
		\right\} \right]\\
		& = \tilde{d}^{1/2} \log n \cdot O_\pr( 1 ). 
	\end{align*}
\end{lemma}

\begin{lemma}\label{lemma:dev_U_third}
	Under Assumptions \ref{asmp:rbe} and \ref{asmp:noninformative} and the null hypothesis $H_0$, 
	and conditioning on all the potential event times,  
	\begin{align*}
		\sum_{k=1}^{K-1} \xi_{3k} & \equiv 
		\sum_{k=1}^{K-1}
		\left[ 
		\frac{
			D_k  \E( N_k-D_k \mid  \bm{T}(1), \bm{T}(0), N_k )	\E\{N_{1k} (N_k - N_{1k})\mid  \bm{T}(1), \bm{T}(0), N_k\}
		}{
			N_k^2 (N_k-1)
		} 
		\right. 
		\\
		&  \quad \ 
		\left. 
		\vphantom{
		\frac{
			D_k  \E( N_k-D_k \mid  \bm{T}(1), \bm{T}(0), N_k )	\E\{N_{1k} (N_k - N_{1k})\mid  \bm{T}(1), \bm{T}(0), N_k\}
		}{
			N_k^2 (N_k-1)
		} 
		}
		\qquad \ \ 
		- 
		\E(D_k \mid \bm{T}(1), \bm{T}(0))  (1-h_k) \phi_k(1-\phi_k) \right]\\
		& = n^{1/2} \cdot O_\pr (1).
	\end{align*}
\end{lemma}

\begin{lemma}\label{lemma:dev_U_fourth}
	Under Assumptions \ref{asmp:rbe} and \ref{asmp:noninformative} and the null hypothesis $H_0$, 
	and conditioning on all the potential event times,  
	\begin{align*}
	\sum_{k=1}^{K-1} \xi_{4k} & = 
	\sum_{k=1}^{K-1}\left\{
	\E(D_k \mid \bm{T}(1), \bm{T}(0))  (1-h_k) \phi_k(1-\phi_k) - \E(V_k \mid \bm{T}(1), \bm{T}(0))
	\right\} \\
	& = \log n \cdot O\left( 1 \right). 
	\end{align*}
\end{lemma}

\subsection{Proofs of the lemmas}

\begin{proof}[Proof of Lemma \ref{lemma:martingale_CLT}]
	From Theorem \ref{thm:mds}, $\{D_{1k}-M_{k}\}_{k=1}^K$ is a martingale difference sequence with respect to filtration $\mathcal{F}_k$'s. 
	Lemma \ref{lemma:martingale_CLT} then follows immediately from the martingale central central limit theorem; see, e.g., \citet[][Theorem 2]{Brown1971} and \citet[][Corollary 3.1]{hall1980}.
\end{proof}

\begin{proof}[Proof of Lemma \ref{lemma:bound_EU}]
	From Theorem \ref{thm:mean_var_L} and the fact that $h_K=1$, 
	the mean of $U$ is
	\begin{align}\label{eq:EU_proof}
	\E(U) & = 
	\sum_{k=1}^{K-1}
	\frac{n_k}{n_k-1}  h_k (1-h_k) \phi_k (1-\phi_k) 
	\left\{
	n_k g_k - 1 + (1-g_k)^{n_k}
	\right\}.
	\end{align}
	For $1\le k \leq K$, let $X_k \sim \Binomial(n_k-1, g_k)$. We have
	\begin{align*}
	n_k g_k - 1 + (1-g_k)^{n_k} - (n_k-1) g_k^2
	& = (1-g_k)\left\{
	(n_k-1)g_k - 1 + (1-g_k)^{n_k-1}
	\right\} \\
	& = (1-g_k) \cdot \E\left(
	X - 1 + \I\{X=0\}
	\right) \\
	& \ge 0.
	\end{align*}
	Thus, $n_k g_k - 1 + (1-g_k)^{n_k} \ge (n_k-1) g_k^2$. 
	From \eqref{eq:EU_proof}, 
	we can then bound $\E(U)$ by 
	\begin{align}\label{eq:EU_proof2}
	\E(U) &  \ge \sum_{k=1}^{K-1}
	\frac{n_k}{n_k-1}  h_k (1-h_k) \phi_k (1-\phi_k) (n_k-1) g_k^2\\
	& = \sum_{k=1}^{K-1}
	n_k h_k (1-h_k) \phi_k (1-\phi_k)  g_k^2.
	\nonumber
	\end{align}
	Recall that $\tilde{d} = \max_{1\le k\le K} d_k$, and  $\tilde{G}(t) = G_1(t) G_0(t)$. 
	For $1\le k \le K-1$, we have 
	$
	h_k = d_k/n_k \le d_k/(d_k + 1) \le \tilde{d}/(\tilde{d}+1),
	$
	and $\phi_k (1-\phi_k)g_k^2 = p_1 G_1(t_k) p_0 G_0(t_k) = p_1 p_0 \tilde{G}(t_k)$. 
	From \eqref{eq:EU_proof2}, we can further bound $\E(U)$ by 
	\begin{align*}
	\E(U) & \ge 
	\sum_{k=1}^{K-1}
	n_k h_k \left(1-\frac{\tilde{d}}{\tilde{d}+1}\right) p_1 p_0 \tilde{G}(t_k)
	= 
	\frac{p_1 p_0 }{\tilde{d}+1} 
	\sum_{k=1}^{K-1}
	d_k \tilde{G}(t_k) 
	\\
	& = 
	\frac{p_1 p_0 }{\tilde{d}+1} 
	\left(
	\sum_{k=1}^{K}
	d_k \tilde{G}(t_k) - d_K \tilde{G}(t_K)
	\right)
	= 
	\frac{p_1 p_0 }{\tilde{d}+1} 
	\left(
	\sum_{i=1}^{n} \tilde{G}(T_i(0)) - d_K \tilde{G}(t_K)
	\right)
	\\
	& 
	=
	\frac{p_1 p_0 }{\tilde{d}+1} 
	\left(
	n \tilde{g} - d_K \tilde{G}(t_K)
	\right)
	\ge 
	\frac{p_1 p_0 }{\tilde{d}+1} 
	\left(
	n \tilde{g} - \tilde{d}
	\right)
	\ge 
	\frac{p_1 p_0 }{\tilde{d}+1} 
	\left(
	n \tilde{g} - \tilde{d}-1
	\right)\\
	& = 
	p_1 p_0 
	\left(
	\frac{n \tilde{g}}{\tilde{d}+1} - 1
	\right). 
	\end{align*}
	Therefore, Lemma \ref{lemma:bound_EU} holds. 
\end{proof}

\begin{proof}[Proof of Lemma \ref{lemma:dev_U_first}]
	First, we consider bounding the conditional variance of $N_{1k} (N_k - N_{1k})$ given $N_k$. 
	Let 
	$$
	\delta_{1k} \equiv  N_{1k} (N_k - N_{1k}) - \E\{N_{1k} (N_k - N_{1k})\mid N_k\}.
	$$
	It has the following equivalent forms:
	\begin{align*}%
	\delta_{1k} & = N_{1k} (N_k - N_{1k}) - \E\{N_{1k} (N_k - N_{1k})\mid N_k\} \\
	& = N_k \left\{ N_{1k} - E(N_{1k}\mid N_k) \right\} - 
	N_{1k}^2 + \E(N_{1k}^2 \mid N_k)
	\nonumber
	\\
	& = N_k \left\{ N_{1k} - E(N_{1k}\mid N_k) \right\} - 
	N_{1k}^2 + \left\{ \E(N_{1k} \mid N_k) \right\}^2 + \Var(N_{1k} \mid N_k)
	\nonumber
	\\
	& = N_k \left\{ N_{1k} - E(N_{1k}\mid N_k) \right\}
	- 
	\left\{ N_{1k} - \E(N_{1k} \mid N_k) \right\}
	\left\{ N_{1k} + \E(N_{1k} \mid N_k) \right\} 
	\nonumber
	\\
	& \quad \ + \Var(N_{1k} \mid N_k)
	\nonumber
	\\
	& = 
	\left\{ N_{1k} - E(N_{1k}\mid N_k) \right\}
	\left\{
	N_k - N_{1k} - \E(N_{1k} \mid N_k)
	\right\} + \Var(N_{1k} \mid N_k).
	\end{align*}
	Because $0 \le N_{1k} \le N_k$, we have
	$
	- N_k \le N_k - N_{1k} - \E(N_{1k} \mid N_k) \le N_k,
	$
	and 
	\begin{align*}%
	|\delta_{1k}| & \le 
	\left| \left\{ N_{1k} - E(N_{1k}\mid N_k) \right\}
	\left\{
	N_k - N_{1k} - \E(N_{1k} \mid N_k)
	\right\} \right| + 
	\Var(N_{1k} \mid N_k)
	\nonumber
	\\
	& \le
	N_k
	\left| N_{1k} - E(N_{1k}\mid N_k)  \right|
	+ 
	\Var(N_{1k} \mid N_k).
	\end{align*}
	By the Cauchy--Schwarz inequality, 
	\begin{align*}
	\delta_{1k}^2 & \le 
	2 N_k^2 
	\left\{ N_{1k} - E(N_{1k}\mid N_k)  \right\}^2
	+
	2 \left\{ \Var(N_{1k} \mid N_k ) \right\}^2.
	\end{align*}
	Taking conditional expectation given $N_k$ on both sides, 
	from Lemma \ref{lemma:joint_dist_D_N1_given_N}, 
	we have 
	\begin{align}\label{eq:bound_delta_1k}
	& \quad \ \E\left( \delta_{1k}^2 \mid N_k \right) \\
	& 
	\le 
	2 N_k^2 
	\E\left[
	\left\{ N_{1k} - E(N_{1k}\mid N_k)  \right\}^2 \mid N_k
	\right]
	+
	2 \left\{ \Var(N_{1k} \mid N_k) \right\}^2
	\nonumber
	\\
	\nonumber
	& = 2 N_k^2 
	\Var(N_{1k} \mid N_k)
	+
	2 \left\{ \Var(N_{1k} \mid N_k) \right\}^2
	\\
	& = 2 N_k^3 \phi_k (1 - \phi_k)
	+ 2 N_k^2 \phi_k^2 (1 - \phi_k)^2
	 \le 4 N_k^3 \phi_k (1 - \phi_k)
	 \le 
	 4 N_k^3 \cdot \frac{1}{4} 
	\nonumber
	\\
	& \le N_k^3.
	\nonumber
	\end{align}
	
	Second, we consider bounding the second moment of $\sum_{k=1}^{K-1}\xi_{1k}$. 
	By the definition of $\xi_{1k}$, 
	\begin{align*}
	\left| \xi_{1k} \right| & = \frac{
		D_k(N_k-D_k) |\delta_{1k}|
	}{
		N_k^2 (N_k-1)
	} \leq 
	\frac{
		D_k |\delta_{1k}|
	}{
		N_k^2 
	}.
	\end{align*}
	The Cauchy--Schwarz inequality then implies 
	\begin{align*}
	\left( \sum_{k=1}^{K-1} \xi_{1k} \right)^2 \leq 
	\left( \sum_{k=1}^{K-1} \left| \xi_{1k} \right| \right)^2 \leq 
	\left( 
	\sum_{k=1}^{K-1}
	\frac{
		D_k |\delta_{1k}|
	}{
		N_k^2 
	} \right)^2
	\leq 
	\left(
	\sum_{k=1}^{K-1} D_k
	\right)
	\left(
	\sum_{k=1}^{K-1}
	\frac{
		D_k \delta_{1k}^2
	}{
		N_k^4
	} \right).
	\end{align*}
	By definition, $D_k\le d_k$, $\sum_{k=1}^{K-1} D_k \le \sum_{k=1}^{K-1} d_k = n$, and thus
	\begin{align}\label{eq:bound_xi_1k_cauchy}
	\left( \sum_{k=1}^{K-1} \xi_{1k} \right)^2 \leq 
	n
	\sum_{k=1}^{K-1}
	\frac{
		D_k \delta_{1k}^2
	}{
		N_k^4
	}.
	\end{align}
	From Lemma \ref{lemma:joint_dist_D_N1_given_N} and \eqref{eq:bound_delta_1k}, we can know that 
	\begin{align*}
	\E\left( \frac{D_k \delta_{1k}^2}{N_k^4} \mid N_k \right) & = 
	\frac{\E( D_k \mid N_k ) \E(\delta_{1k}^2 \mid N_k )}{N_k^4}
	= \frac{ N_k h_k \E(\delta_{1k}^2 \mid N_k )}{N_k^4} \leq 
	\frac{ N_k h_k N_k^3}{N_k^4}  \\
	& = h_k. 
	\end{align*}
	By the law of iterated expectation, 
	\begin{align*}
	\E\left( \frac{D_k \delta_{1k}^2}{N_k^4}  \right) & = 
	\E\left\{
	\E\left( \frac{D_k \delta_{1k}^2}{N_k^4} \mid N_k \right)
	\right\} \leq h_k.
	\end{align*}
	\eqref{eq:bound_xi_1k_cauchy} then implies that 
	\begin{align*}
	\E\left\{ \left( \sum_{k=1}^{K-1} \xi_{1k} \right)^2 \right\} & \leq 
	n
	\sum_{k=1}^{K-1}
	\E\left(
	\frac{
		D_k \delta_{1k}^2
	}{
		N_k^4
	}
	\right) = 
	n
	\sum_{k=1}^{K-1} h_k.
	\end{align*}
	By definition and the mean value theorem, 
	$\log n_k - \log n_{k+1} \ge n_k^{-1} (n_k - n_{k+1})$, and thus
	\begin{align}\label{eq:sum_h}
	\sum_{k=1}^{K-1} h_k & = \sum_{k=1}^{K-1} \frac{d_k}{n_k} = \sum_{k=1}^{K-1} \frac{n_k - n_{k+1}}{n_k} \le 
	\sum_{k=1}^{K-1} \left\{ \log n_k - \log n_{k+1} \right\}  \\
	\nonumber
	& = \log n - \log n_{K} \\
	&  \leq \log n.
	\nonumber
	\end{align}
	Thus, 
	\begin{align*}
	\E\left\{ \left( \sum_{k=1}^{K-1} \xi_{1k} \right)^2 \right\} & \leq 
	n
	\sum_{k=1}^{K-1} h_k
	\le n \log n.
	\end{align*}
	By the Markov inequality,
	\begin{align*}
	\sum_{k=1}^{K-1} \xi_{1k} = O_\pr \left(
	\left[ \E\left\{ \left( \sum_{k=1}^{K-1} \xi_{1k} \right)^2 \right\} \right]^{1/2}
	\right) = \left( n\log n \right)^{1/2} O_\pr (1).
	\end{align*}
	Therefore, Lemma \ref{lemma:dev_U_first} holds.
\end{proof}

\begin{proof}[Proof of Lemma \ref{lemma:dev_U_second}]
	First, we consider bounding the conditional variance of $N_k - D_k$ given $N_k.$ 
	Let 
	$$
	\delta_{2k} \equiv N_k - D_k - \E(N_k - D_k \mid N_k) 
	= - D_k + \E(D_k \mid N_k).
	$$  
	From Lemma \ref{lemma:joint_dist_D_N1_given_N}, we have
	\begin{align}\label{eq:bound_delta_2k}
	\E\left(
	\delta_{2k}^2 \mid N_k
	\right) & = \Var\left(D_k \mid N_k\right) 
	= 
	\frac{d_k (n_k - d_k) N_k (n_k - N_k) }{n_k^2(n_k-1)}
	\\
	& = 
	N_k  \frac{d_k (n_k - d_k)  }{n_k(n_k-1)} \frac{(n_k - N_k)}{n_k}
	\leq N_k \frac{d_k}{n_k} 
	\nonumber
	\\
	& = N_k h_k.
	\nonumber
	\end{align}
	
	Second, we consider bounding the second moment of $\sum_{k=1}^{K-1} \xi_{2k}$. 
	By definition and from Lemma \ref{lemma:mean_D_N_minus_D}, 
	\begin{align*}
	\left| \xi_{2k} \right| & = 
	\frac{
		D_k \E\{N_{1k} (N_k - N_{1k})\mid N_k\}
	}{
		N_k^2 (N_k-1)
	} 
	\left| \delta_{2k} \right|
	\le 
	\frac{D_k}{N_k}  
	\phi_k(1-\phi_k) 
	\left| \delta_{2k} \right| \le \frac{D_k}{N_k}    
	\left| \delta_{2k} \right|.
	\end{align*}
	The Cauchy--Schwarz inequality implies that  
	\begin{align}\label{eq:xi_2k_cauchy}
	\left( \sum_{k=1}^{K-1} \xi_{2k} \right)^2 & \leq 
	\left( \sum_{k=1}^{K-1} \left| \xi_{2k} \right| \right)^2 \leq 
	\left(
	\sum_{k=1}^{K-1} 
	\frac{D_k|\delta_{2k}|}{N_k}   
	\right)^2\\
	& \leq 
	\left(
	\sum_{k=1}^{K-1} 
	\frac{D_k^2}{N_k}    
	\right)
	\left(
	\sum_{k=1}^{K-1} 
	\frac{\delta_{2k}^2}{N_k}    
	\right).
	\nonumber
	\end{align}
	By definition, $D_k \leq d_k \le \tilde{d}$, and $D_k \le N_k - N_{k+1}$. 
	By the mean value theorem,  
	$\log N_k - \log N_{k+1} \ge N_k^{-1} (N_k - N_{k+1})$ if $N_{k+1}\ge 1$, and thus
	\begin{align*}
	\sum_{k=1}^{K-1} 
	\frac{D_k^2}{N_k}    
	& \le 
	\tilde{d}
	\sum_{k=1}^{K-1} 
	\frac{D_k}{N_k}    
	\leq 
	\tilde{d}
	\sum_{k=1}^{K-1} 
	\frac{N_k - N_{k+1}}{N_k}    \\
	& \leq 
	\tilde{d}
	\sum_{k=1}^{K-1} 
	\left\{
	\I ( N_{k+1}\ge 1 )
	\left(
	\log N_k - \log N_{k+1}  
	\right)
	+ 
	\I( N_{k+1}=0, N_k \geq 1 )
	\right\} \\
	& = \tilde{d} \sum_{k=1}^{K-1} \I\{N_{k+1}\ge 1\}
	\left\{
	\log(N_k) - \log(N_{k+1}) 
	\right\}   \\
	& \quad \ + \tilde{d} 
	\sum_{k=1}^{K-1}
	\I\{N_{k+1}=0, N_k \geq 1 \} \\
	& \leq \tilde{d} \log n + \tilde{d} = \tilde{d} 
	\left(\log n + 1\right). 
	\end{align*}
	\eqref{eq:xi_2k_cauchy} then implies 
	\begin{align}\label{eq:xi_2k_bound}
	\left( \sum_{k=1}^{K-1} \xi_{2k} \right)^2  \leq 
	\left(
	\sum_{k=1}^{K-1} 
	\frac{D_k^2}{N_k}    
	\right)
	\left(
	\sum_{k=1}^{K-1} 
	\frac{\delta_{2k}^2}{N_k}    
	\right)
	\le 
	\tilde{d} 
	\left( \log n + 1 \right)
	\sum_{k=1}^{K-1} 
	\frac{\delta_{2k}^2}{N_k}.
	\end{align}
	From \eqref{eq:bound_delta_2k}, we have
	$
	\E(
	N_k^{-1}\delta_{2k}^2   \mid N_k
	) = N_k^{-1}
	\E(
	\delta_{2k}^2 \mid N_k
	)  \le h_k,
	$
	which immediately implies that 
	\begin{align*}
	\E\left(
	\frac{\delta_{2k}^2}{N_k}    
	\right) & = 
	\E\left\{
	\E\left(
	\frac{\delta_{2k}^2}{N_k} \mid N_k
	\right)\right\} \le h_k. 
	\end{align*}
	Thus, using \eqref{eq:sum_h} and \eqref{eq:xi_2k_bound}, we can bound the second moment of $\sum_{k=1}^{K-1} \xi_{2k}$ by 
	\begin{align*}
	\E\left[
	\left( \sum_{k=1}^{K-1} \xi_{2k} \right)^2
	\right] & \le 
	\tilde{d} 
	\left( \log n + 1\right)
	\sum_{k=1}^{K-1} 
	\E\left(
	\frac{\delta_{2k}^2}{N_k}    
	\right) \le 
	\tilde{d} 
	\left( \log n + 1\right)
	\sum_{k=1}^{K-1} 
	h_k
	\\
	& \leq 
	\tilde{d} 
	\left( \log n + 1 \right) \log n. 
	\end{align*}
	The Markov inequality then implies 
	\begin{align*}
	\sum_{k=1}^{K-1} \xi_{2k} = O_\pr \left(
	\left[ \E\left\{ \left( \sum_{k=1}^{K-1} \xi_{2k} \right)^2 \right\} \right]^{1/2}
	\right) = 
	\tilde{d}^{1/2} \log n \cdot
	O_\pr (1).
	\end{align*}
	Therefore, Lemma \ref{lemma:dev_U_second} holds. 
\end{proof}

\begin{proof}[Proof of Lemma \ref{lemma:dev_U_third}]
	Let 
	$$
	\delta_{3k} = D_k - \E(D_k)
	$$ 
	be the difference between $D_k$ and its mean. 
	From Lemmas \ref{lemma:joint_dist_D_N1_given_N} and  \ref{lemma:mean_D_N_minus_D}, $\xi_{3k}$ has the following equivalent forms:
	\begin{align*}
	\xi_{3k} & = 
	\frac{
		D_k  \E( N_k-D_k \mid N_k )	\E\{N_{1k} (N_k - N_{1k})\mid N_k\}
	}{
		N_k^2 (N_k-1)
	} 
	- 
	\E(D_k)  (1-h_k) \phi_k(1-\phi_k)\\
	& = D_k (1-h_k) \phi_k (1-\phi_k) \I\{N_k > 1\} - \E(D_k)  (1-h_k) \phi_k(1-\phi_k)\\
	& = D_k (1-h_k) \phi_k (1-\phi_k) - D_k (1-h_k) \phi_k (1-\phi_k)\I\{N_k \le 1\} \\
	& \quad \  - \E(D_k)  (1-h_k) \phi_k(1-\phi_k)\\
	& = (1-h_k) \phi_k (1-\phi_k) \delta_{3k} -  (1-h_k) \phi_k (1-\phi_k) D_k\I\{N_k \le 1\}. 
	\end{align*}
	Because $N_k = 0$  implies that $D_k = 0$, we have 
	$
	D_k\I\{N_k \le 1\} = D_k\I\{N_k = 1\},
	$
	and thus
	\begin{align}\label{eq:decomp_xi_3k}
	\sum_{k=1}^{K-1}\xi_{3k} & = \sum_{k=1}^{K-1} (1-h_k) \phi_k (1-\phi_k) \delta_{3k} \\
	& \quad \ -  \sum_{k=1}^{K-1}(1-h_k) \phi_k (1-\phi_k) D_k\I\{N_k = 1\}. 
	\nonumber
	\end{align}
	Below we consider the two terms in \eqref{eq:decomp_xi_3k} separately. 
	
	First, we consider the summation of $(1-h_k) \phi_k (1-\phi_k) \delta_{3k}$ in \eqref{eq:decomp_xi_3k}. Because any unit $i$ having event at time $t_k$ must satisfy $T_i(1) = T_i(0) = t_k$, 
	$\delta_{3k}$ has the following equivalent forms:
	\begin{align*}
	\delta_{3k} =  D_k - \E(D_k) = \sum_{i=1}^{n} \Delta_i\I (W_i = t_k) - \E(D_k) 
	= \sum_{i: T_i(0)=t_k} \I (C_i \ge t_k) - \E(D_k).
	\end{align*}
	Thus, 
	$\delta_{3k}$'s are mutually independent for $1\le k\le K-1$. 
	From Lemma \ref{lemma:marg_dist_D_N1_N}, 
	$
	D_k\sim \Binomial(d_k, g_k). 
	$
	Therefore, 
	$\Var(\delta_{3k}) = \Var(D_k) = d_k g_k (1-g_k)$, and 
	\begin{align*}
	& \quad \ \Var\left\{
	\sum_{k=1}^{K-1} (1-h_k) \phi_k (1-\phi_k) \delta_{3k}
	\right\}\\
	& = \sum_{k=1}^{K-1} (1-h_k)^2 \phi_k^2 (1-\phi_k) ^2 \Var(\delta_{3k}) = \sum_{k=1}^{K-1} (1-h_k)^2 \phi_k^2 (1-\phi_k)^2  d_k g_k (1-g_k) \\
	& \leq \sum_{k=1}^{K-1} d_k \le n.
	\end{align*} 
	By the Markov inequality, 
	\begin{align}\label{eq:delta_3k_first}
	& \quad \ \sum_{k=1}^{K-1} (1-h_k) \phi_k (1-\phi_k) \delta_{3k} 
	\\
	& = 
	O_\pr \left(
	\left[
	\Var\left\{
	\sum_{k=1}^{K-1} (1-h_k) \phi_k (1-\phi_k) \delta_{3k}
	\right\}
	\right]^{1/2}
	\right)
	\nonumber
	\\
	& = O_\pr \left(
	n^{1/2}
	\right). 
	\nonumber
	\end{align}
	
	Second, we consider the summation of $ (1-h_k) \phi_k (1-\phi_k) D_k\I\{N_k = 1\}$ in \eqref{eq:decomp_xi_3k}. 
	From Lemma \ref{lemma:joint_dist_D_N1_given_N}, 
	\begin{align*}
	\E\left\{ D_k\I (N_k = 1) \mid N_k \right\} = \I(N_k = 1) \cdot \E\left(
	D_k \mid N_k
	\right) =  \I (N_k = 1) \cdot N_k h_k \leq h_k.
	\end{align*}
	By the law of iterated expectation, 
	$$
	\E\left\{ D_k\I (N_k = 1) \right\}
	= 
	\E\left[ \E\left\{ D_k\I (N_k = 1) \mid N_k \right\} \right] 
	\le h_k. 
	$$
	From the linearity of expectation and \eqref{eq:sum_h},  
	\begin{align*}
	& \quad \ E\left\{
	\sum_{k=1}^{K-1}(1-h_k) \phi_k (1-\phi_k) D_k\I (N_k = 1)
	\right\}\\
	& = \sum_{k=1}^{K-1}(1-h_k) \phi_k (1-\phi_k)\E\left\{ D_k\I (N_k = 1) \right\}
	\leq \sum_{k=1}^{K-1}(1-h_k) \phi_k (1-\phi_k)h_k \\
	& \leq \sum_{k=1}^{K-1}h_k \le \log n. 
	\end{align*}
	From the Markov inequality, we have
	\begin{align}\label{eq:delta_3k_second}
	& \quad \ \sum_{k=1}^{K-1}(1-h_k) \phi_k (1-\phi_k) D_k\I (N_k = 1) 
	\\
	& =
	O_\pr \left(
	E\left\{
	\sum_{k=1}^{K-1}(1-h_k) \phi_k (1-\phi_k) D_k\I(N_k = 1)
	\right\}
	\right)
	\nonumber
	\\
	& = 
	O_\pr ( \log n ).
	\nonumber
	\end{align}
	
	From \eqref{eq:decomp_xi_3k}, \eqref{eq:delta_3k_first} and \eqref{eq:delta_3k_second}, we have 
	$
	\sum_{k=1}^{K-1}\xi_{3k}  = O_\pr \left(
	n^{1/2}
	\right),
	$
	i.e., Lemma \ref{lemma:dev_U_third} holds. 
\end{proof}

\begin{proof}[Proof of Lemma \ref{lemma:dev_U_fourth}]
	From Lemma \ref{lemma:marg_dist_D_N1_N} and  \eqref{eq:mean_Vk}, 
	for $1\le k \le K-1$, 
	\begin{align*}
	\xi_{4k} & = d_k g_k  (1-h_k) \phi_k(1-\phi_k) \\
	& \quad \ - \left\{
	n_k g_k -1 + (1-g_k)^{n_k}
	\right\} h_k(1-h_k) \phi_k (1 - \phi_k)\frac{n_k}{n_k-1}\\
	& = 
	h_k (1-h_k) \phi_k(1-\phi_k)
	\frac{n_k}{n_k-1}
	\left[
	(n_k-1) g_k  - 
	\left\{
	n_k g_k -1 + (1-g_k)^{n_k}
	\right\} 
	\right]\\
	& = h_k (1-h_k) \phi_k(1-\phi_k)
	\frac{n_k}{n_k-1}
	(1-g_k)
	\left\{
	1 - (1-g_k)^{n_k-1}
	\right\}\\
	& \le h_k (1-h_k) \frac{n_k}{n_k-1} = \frac{d_k(n_k-d_k)}{n_k(n_k-1)} \leq  \frac{d_k}{n_k}   \\
	& = h_k.
	\end{align*}
	\eqref{eq:sum_h} then implies that 
	\begin{align*}
	0 \le \sum_{k=1}^{K-1} \xi_{4k}  \leq \sum_{k=1}^{K-1} h_k \le \log n. 
	\end{align*}
	Therefore, $\sum_{k=1}^{K-1} \xi_{4k} = O\left( \log n \right)$, i.e., Lemma \ref{lemma:dev_U_fourth} holds. 
\end{proof}

\subsection{Proof of Theorem \ref{thm:CLT_conti}}

\begin{proof}[Proof of Theorem \ref{thm:CLT_conti}]
	We check the two conditions in Lemma \ref{lemma:martingale_CLT} separately. 
	First, we consider the conditional Lindeberg condition. 
	By definition, we know that $0 \le D_{1k}\le d_k \le \tilde{d}$. This immediately implies that $0 \le \ED_{k} = \E(D_{1k} \mid \mathcal{F}_{k-1}) \le \tilde{d}$, and 
	$|D_{1k} - \ED_{k}| \le \tilde{d}$. 
	From Lemma \ref{lemma:bound_EU}, for $1\leq k \leq K$, 
	we then have
	\begin{align*}
	\frac{|D_{1k}-\ED_{k}|^2}{\Var(L)} \leq \frac{\tilde{d}^2}{\E(U)}
	\le
	\frac{\tilde{d}^2}{
		p_1 p_0 \{ n\tilde{g}/(\tilde{d}+1) - 1 \}
	}
	= 
	\frac{1}{p_1 p_0} 
	\left\{
	\frac{n\tilde{g}}{\tilde{d}^2 (\tilde{d}+1)} - \frac{1}{\tilde{d}^2}
	\right\}^{-1}. 
	\end{align*}
	This then implies that
	\begin{align}\label{eq:bound_std_D}
	\max_{1\leq k\le K}\frac{|D_{1k}-\ED_{k}|^2}{\Var(L)} \leq \frac{1}{p_1 p_0} 
	\left\{
	\frac{n\tilde{g}}{\tilde{d}^2 (\tilde{d}+1)} - \frac{1}{\tilde{d}^2}
	\right\}^{-1}. 
	\end{align}
	Because
	\begin{align*}
	\frac{n\tilde{g}}{\tilde{d}^2 (\tilde{d}+1)} - \frac{1}{\tilde{d}^2}
	& \ge 
	\frac{n\tilde{g}}{\tilde{d}^2\cdot 2\tilde{d}} - 1 = \frac{1}{2}\frac{n \tilde{g}}{ \tilde{d}^3} - 1,
	\end{align*}
	under Condition \ref{cond:fp_conti}, we can know that 
	\begin{align*}
	\frac{n\tilde{g}}{\tilde{d}^2 (\tilde{d}+1)} - \frac{1}{\tilde{d}^2}
	& \converge \infty.
	\end{align*}
	From \eqref{eq:bound_std_D}, under Condition \ref{cond:fp_conti}, we have
	\begin{align*}
	\max_{1\leq k\le K}\frac{|D_{1k}-\ED_{k}|^2}{\Var(L)} \converge 0. 
	\end{align*}
	Thus, for any $\varepsilon>0$, there exists $\underline{n}$ such that when $n \ge \underline{n}$, 
	$$
	\max_{1\leq k\le K}\frac{|D_{1k}-\ED_{k}|^2}{\Var(L)} \leq \varepsilon^2.
	$$
	Consequently, when $n \ge \underline{n}$, 
	\begin{align*}
	\sum_{k=1}^{K} \E\left[
	\frac{\left(D_{1k} - \ED_{k}\right)^2}{\Var(L)} 
	\I\left\{
	\frac{|D_{1k}-\ED_{k}|}{\sqrt{\Var(L)}} > \varepsilon
	\right\}
	\mid \mathcal{F}_{k-1}
	\right] = 0.
	\end{align*}
	Therefore, the conditional Lindeberg condition holds. 
	
	Second, we consider the consistency of $U$ for $\E(U)$. Lemmas \ref{lemma:dev_U_first}--\ref{lemma:dev_U_fourth} imply 
	\begin{align}\label{eq:dev_U_order_two_terms}
	& \quad \ U - \E(U) \\
	& = \sum_{k=1}^{K-1}V_k -  \sum_{k=1}^{K-1} \E(V_k) = 
	\sum_{k=1}^{K-1}\xi_{1k} + \sum_{k=1}^{K-1}\xi_{2k} + \sum_{k=1}^{K-1}\xi_{3k} + \sum_{k=1}^{K-1}\xi_{4k} 
	\nonumber
	\\
	& =   
	\left( n\log n \right)^{1/2}
	O_\pr(1)
	+
	\tilde{d}^{1/2} \log n  \cdot
	O_\pr(1)
	+
	n^{1/2}
	O_\pr(1)
	+
	O( \log n )
	\nonumber
	\\
	& = 
	\left( n\log n \right)^{1/2}
	O_\pr (1) + 
	\tilde{d}^{1/2} \log n \cdot
	O_\pr (1).
	\nonumber
	\end{align}
	By definition, $\tilde{g} \le 1$, and thus Condition \ref{cond:fp_conti} implies that 
	\begin{align*}
	\frac{1}{\tilde{d}}
	\left(
	\frac{n}{\log n}
	\right)^{1/2} \converge \infty.
	\end{align*}
	This implies that 
	\begin{align*}
	\frac{\tilde{d}^{1/2} \log n}{\left( n\log n \right)^{1/2}}
	& = 
	\tilde{d}^{1/2} 
	\left(
	\frac{n}{\log n}
	\right)^{-1/2} 
	\leq \tilde{d}
	\left(
	\frac{n}{\log n}
	\right)^{-1/2}  \converge 0.
	\end{align*}
	Thus, from \eqref{eq:dev_U_order_two_terms}, we have
	$
	U - \E(U) = \left( n\log n \right)^{1/2}
	O_\pr (1),
	$
	and consequently, 
	\begin{align}\label{eq:dev_U_order}
	\frac{U}{\E(U)} - 1 & 
	=
	\frac{U - \E(U)}{\E(U)} = 
	\frac{
		\left( n\log n \right)^{1/2}}{\E(U)}
	O_\pr (1).
	\end{align}
	Under Condition \ref{cond:fp_conti}, from Lemma \ref{lemma:bound_EU}, we have
	\begin{align*}
	\frac{\E(U)}{
		\left( n\log n \right)^{1/2}} 
	& \ge 
	\frac{1}{
		\left( n\log n \right)^{1/2}}
	p_1 p_0 \left( \frac{n\tilde{g}}{\tilde{d}+1}  - 1 \right) \ge \frac{1}{
		\left( n\log n \right)^{1/2}} 
	p_1 p_0 \left( \frac{n\tilde{g}}{2\tilde{d}}  - 1 \right) \\
	& = p_1 p_0
	\left\{
	\frac{1}{2}
	\frac{\tilde{g}}{\tilde{d}}
	\left( \frac{n}{
		\log n} \right)^{1/2}   - \frac{1}{
		\left( n\log n \right)^{1/2}}  \right\}
	\converge \infty.
	\end{align*}
	Therefore, from \eqref{eq:dev_U_order}, we have
	\begin{align*}
	\frac{U}{\E(U)} - 1 & 
	=
	o_\pr (1),
	\end{align*}
	i.e., $U$ is consistent for $\E(U)$. 
	
	From the above and Lemma \ref{lemma:martingale_CLT}, the logrank statistic $\LR$ in \eqref{eq:logrank_randomization} is asymptotically standard Gaussian, i.e., Theorem \ref{thm:CLT_conti} holds. 
\end{proof}

\section{Asymptotic distribution of the logrank statistic with finite distinct potential event times}\label{app:asymp_finite}

\subsection{Lemmas}

To prove Theorem \ref{thm:CLT_discrete}, we need the following five lemmas. 

\begin{lemma}\label{lemma:clt_hyper}
	Let $X$ follow a hypergeometric distribution,  $\Phi(\cdot)$ be the cumulative distribution function of $\mathcal{N}(0,1)$, and $\gamma$ be the usual Berry--Esseen constant. Then
	\begin{align*}
	\sup_x \left|
	\pr\left(
	\frac{X - \E(X)}{
	\sqrt{\Var(X)}
	}
	\le x
	\right) - \Phi(x)
	\right| \le \frac{\gamma}{
		\sqrt{\Var(X)}
	}.
	\end{align*}
\end{lemma}

\begin{lemma}\label{lemma:bound_X_n_minus_X_recip}
	For any nonnegative integers $N, K, n$, and constant $p\in (0,1)$, 
	let $X$ follow a hypergeometric distribution with parameters $(N, K, n)$, and $Y$ follow a binomial distribution with parameters $(n, p)$. Then 
	\begin{align*}
	\E\left\{
	\frac{\I(0<X<n)}{X(n-X)}
	\right\} & \le \frac{4}{(K+1)(N-K+1)} 
	\frac{(N+1)(N+2)}{(n+1)(n+2)},\\
	\E\left\{
	\frac{\I(0<Y<n)}{Y(n-Y)}
	\right\} & \le \frac{4}{p(1-p)(n+1)(n+2)},
	\end{align*}
	and 
	\begin{align*}
	\E\left(
	\frac{1}{Y+1}
	\right)
	= 
	\frac{1}{(n+1)p}\left\{
	1 - 
	(1-p)^{n+1}
	\right\}.
	\end{align*}
\end{lemma}

\begin{lemma}\label{lemma:bound_mean_Vk_inv}
	Under Assumptions \ref{asmp:rbe} and \ref{asmp:noninformative} and the null hypothesis $H_0$, 
	and conditioning on all the potential event times,  
	for $1\le k \le K$,
	\begin{align*}
	\E\left( \frac{\I(V_k>0)}{V_k} \mid \bm{T}(1), \bm{T}(0) \right) & \le 
	\frac{16}{g_k \phi_k(1-\phi_k)} 
	\frac{(1+2n_k^{-1})}{(h_k+n_k^{-1})(1-h_k+n_k^{-1})} \frac{1}{n_k}. 
	\end{align*}
\end{lemma}

\begin{lemma}\label{lemma:prob_limit_Vk}
	Under Assumptions \ref{asmp:rbe} and \ref{asmp:noninformative}, the null hypothesis $H_0$, and Condition \ref{cond:fp_discrete}, 
	and conditioning on all the potential event times, 
	for $1\le k \le K$, 
	\begin{itemize}
		\item[(i)] if the limits of $d_k/n$, $G_1(t_k)$ and $G_0(t_k)$ are all positive, 
		and the limit of $h_k$ is less than $1$, 
		then 
		\begin{align*}
		\frac{\E(V_k \mid \bm{T}(1), \bm{T}(0))}{n} = g_k \phi_k (1-\phi_k) 
		\frac{n_k}{n}
		h_k (1-h_k) + o(1), 
		\end{align*}
		and 
		\begin{align*}
		\frac{V_k}{n}  
		= 
		\frac{\E( V_k \mid \bm{T}(1), \bm{T}(0) )}{n} 
		+ o_\pr (1).
		\end{align*}
		\item[(ii)] otherwise, i.e., at least one of the limits of $d_k/n$, $G_1(t_k)$ and $G_0(t_k)$ are zero, or these three limits are all positive but the limit of $h_k$ is $1$, 
		\begin{align*}
		\frac{\E(V_k \mid \bm{T}(1), \bm{T}(0) )}{n} = o(1), 
		\ \text{ and } \ 
		\frac{V_k}{n}  = o_\pr (1). 
		\end{align*}
	\end{itemize}
\end{lemma}

\begin{lemma}\label{lemma:bound_mean_Vk_inv_all}
	Under Assumptions \ref{asmp:rbe} and \ref{asmp:noninformative}, the null hypothesis $H_0$, and Condition \ref{cond:fp_discrete}, 
	and conditioning on all the potential event times, 
	define a subset of $\{1,2,\ldots, K\}$ as 
	\begin{align*}
	\mathcal{K}_0 & = \left\{k_1, \ldots, k_J \right\} \\
	& \equiv 
	\left\{k: \lim_{n\rightarrow \infty}\frac{d_k}{n} > 0,  
	\lim_{n\rightarrow \infty} h_k < 1
	\lim_{n\rightarrow \infty}G_1(t_k) > 0,  \lim_{n\rightarrow \infty} G_0(t_k) > 0 \right\},
	\end{align*}
	where $J\ge 1$ is the cardinality of $\mathcal{K}_0$. 
	For $1\leq k \leq K$, we have
	\begin{itemize}
		\item[(i)] if $k \in \mathcal{K}_0$, then 
		\begin{align*}
		\lim_{n\rightarrow \infty}\frac{\E(V_k \mid \bm{T}(1), \bm{T}(0) )}{n} & > 0, 
		&
		\frac{V_k}{n} = \frac{\E(V_k \mid \bm{T}(1), \bm{T}(0) )}{n} + o_\pr(1), 
		\\ 
		\E\left( \frac{\I(V_k>0)}{V_k^{1/2}} \mid \bm{T}(1), \bm{T}(0) \right) & = o(1),
		& 
		\pr(V_k = 0 \mid \bm{T}(1), \bm{T}(0) ) = o(1);
		\end{align*}
		\item[(ii)] otherwise, 
		\begin{align*}
		\frac{\E( V_n \mid \bm{T}(1), \bm{T}(0) )}{n} = o(1), \quad
		\frac{V_n}{n}  = o_\pr(1).
		\end{align*}
	\end{itemize}
\end{lemma}

\subsection{Proofs of the lemmas}

\begin{proof}[Proof of Lemma \ref{lemma:clt_hyper}]
	Lemma \ref{lemma:clt_hyper} follows directly from \citet[][Theorem 2.3]{kou1996asymptotics}.
\end{proof}

\begin{proof}[Proof of Lemma \ref{lemma:bound_X_n_minus_X_recip}]
	First, for $1\le x \le n-1$, we can verify that $1/x \leq 2/(x+1)$ and $1/(n-x) \leq 2/(n-x+1)$. This implies that
	\begin{align}\label{eq:bound_X_n_minus_X_recip1}
	& \quad \ \E\left\{
	\frac{\I(0<X<n)}{X(n-X)}
	\right\} \\
	& = 
	\sum_{x=\max\{1, n+K-N\}}^{\min \{n-1, K\}}
	\frac{1}{x(n-x)} 
	\frac{\binom{K}{x}\binom{N-K}{n-x}}{\binom{N}{n}} 
	\nonumber
	\\
	& \leq 
	\sum_{x=\max\{1, n+K-N\}}^{\min \{n-1, K\}}
	\frac{4}{(x+1)(n-x+1)} 
	\frac{\binom{K}{x}\binom{N-K}{n-x}}{\binom{N}{n}}. 
	\nonumber
	\end{align}
	Because 
	\begin{align*}
		& \quad \ 
		\frac{4}{(x+1)(n-x+1)} 
		\frac{\binom{K}{x}\binom{N-K}{n-x}}{\binom{N}{n}} \\
		& = 
		\frac{1}{\binom{N}{n}}
		\frac{4}{(x+1)(n-x+1)} 
		\frac{K!}{x!(K-x)!}
		\frac{(N-K)!}{(n-x)! (N-K-n+x)!}\\
		& = 
		\frac{1}{\binom{N}{n}}
		\frac{4}{(K+1)(N-K+1)} 
		\frac{(K+1)!}{(x+1)!(K-x)!}
		\frac{(N-K+1)!}{(n-x+1)! (N-K-n+x)!}\\
		& = 
		\frac{4}{(K+1)(N-K+1)} \frac{\binom{N+2}{n+2}}{\binom{N}{n}}
		\frac{\binom{K+1}{x+1}
			\binom{N-K+1}{n-x+1}}{
			\binom{N+2}{n+2}
		},
	\end{align*}
	from \eqref{eq:bound_X_n_minus_X_recip1}, we have 
	\begin{align*}
	& \quad \ \E\left\{
	\frac{\I(0<X<n)}{X(n-X)}
	\right\} \\
	& \le 
	\frac{4}{(K+1)(N-K+1)} \frac{\binom{N+2}{n+2}}{\binom{N}{n}}\sum_{x=\max\{1, n+K-N\}}^{\min \{n-1, K\}}
	\frac{\binom{K+1}{x+1}
		\binom{N-K+1}{n-x+1}}{
		\binom{N+2}{n+2}
	}\\
	& \leq \frac{4}{(K+1)(N-K+1)} \frac{\binom{N+2}{n+2}}{\binom{N}{n}}
	\\
	& = 
	\frac{4}{(K+1)(N-K+1)} 
	\frac{(N+1)(N+2)}{(n+1)(n+2)}. 
	\end{align*}

	Second, similarly, for $1\le y\le n-1$, we have $1/y \leq 2/(y+1)$ and $1/(n-y) \leq 2/(n-y+1)$. Thus,
	\begin{align*}
	\E\left\{
	\frac{\I(0<Y<n)}{Y(n-Y)}
	\right\} & = 
	\sum_{y=1}^{n-1} 
	\frac{1}{y(n-y)} 
	\binom{n}{y} p^y (1-p)^{n-y}
	\\
	& \le 
	\sum_{y=1}^{n-1} 
	\frac{4}{(y+1)(n-y+1)} 
	\binom{n}{y} p^y (1-p)^{n-y}. 
	\end{align*}
	Because 
	\begin{align}\label{eq:bound_X_n_minus_X_recip2}
	& \quad \ 
	\frac{4}{(y+1)(n-y+1)} 
	\binom{n}{y} p^y (1-p)^{n-y}\\
	& = 
	\frac{4}{(y+1)(n-y+1)} 
	\frac{n!}{y! (n-y)!} p^y (1-p)^{n-y}
	\nonumber
	\\
	& = 
	\frac{4n!}{p(1-p)(n+2)!} 
	\frac{(n+2)!}{(y+1)! (n-y+1)!} p^{y+1} (1-p)^{n-y+1}
	\nonumber
	\\
	& = 
	\frac{4}{p(1-p)(n+1)(n+2)} 
	\binom{n+2}{y+1} p^{y+1} (1-p)^{n-y+1},
	\nonumber
	\end{align}
	from \eqref{eq:bound_X_n_minus_X_recip2}, we have 
	\begin{align*}
	\E\left\{
	\frac{\I(0<Y<n)}{Y(n-Y)}
	\right\} & \le 
	\frac{4}{p(1-p)(n+1)(n+2)} \sum_{y=1}^{n-1} 
	\binom{n+2}{y+1} p^{y+1} (1-p)^{n-y+1}\\
	& \le 
	\frac{4}{p(1-p)(n+1)(n+2)}. 
	\end{align*}
	
	Third, by definition, we have
	\begin{align*}
	\E\left(
	\frac{1}{Y+1}
	\right)
	& = 
	\sum_{y=0}^n \frac{1}{y+1} \binom{n}{y} p^y (1-p)^{n-y}
	= 
	\sum_{y=0}^n \frac{1}{y+1}
	\frac{n!}{y!(n-y)!}
	p^y (1-p)^{n-y}\\
	& = 
	\frac{1}{(n+1)p} \sum_{y=0}^n 
	\frac{(n+1)!}{(y+1)!(n-y)!}
	p^{y+1} (1-p)^{n-y}\\
	& = 
	\frac{1}{(n+1)p} \sum_{y=0}^n 
	\binom{n+1}{y+1}
	p^{y+1} (1-p)^{n-y}
	\\
	& =
	\frac{1}{(n+1)p} \sum_{x=1}^{n+1} 
	\binom{n+1}{x}
	p^{x} (1-p)^{n+1-x}
	\\
	& = 
	\frac{1}{(n+1)p}\left\{
	1 - 
	(1-p)^{n+1}
	\right\}.
	\end{align*}
	
	From the above, Lemma \ref{lemma:bound_X_n_minus_X_recip} holds. 
\end{proof}

\begin{proof}[Proof of Lemma \ref{lemma:bound_mean_Vk_inv}]
	By definition, we have
	\begin{align*}
	\frac{\I(V_k>0)}{V_k} & = N_k^2(N_k-1)
	\frac{
		\I(0< D_k < N_k)
	}{D_k(N_k-D_k)}
	\frac{
		\I(0 < N_{1k} < N_k)
	}{N_{1k} (N_k - N_{1k})}
	\end{align*}
	Lemmas \ref{lemma:joint_dist_D_N1_given_N} and \ref{lemma:bound_X_n_minus_X_recip} then imply 
	\begin{align*}
	& \quad \ \E\left( \frac{\I(V_k>0)}{V_k} \mid N_k \right) \\
	& = N_k^2(N_k-1)
	\E\left\{ \frac{
		\I(0< D_k < N_k)
	}{D_k(N_k-D_k)}
	\mid N_k
	\right\}
	\E\left\{
	\frac{
		\I(0 < N_{1k} < N_k )
	}{N_{1k} (N_k - N_{1k})}
	\mid N_k
	\right\}\\
	& \le 
	\frac{N_k^2(N_k-1) \cdot 4}{(d_k+1)(n_k-d_k+1)} 
	\frac{(n_k+1)(n_k+2)}{(N_k+1)(N_k+2)}
	\cdot
	\frac{4}{\phi_k(1-\phi_k)(N_k+1)(N_k+2)}\\
	& = 
	\frac{16}{\phi_k(1-\phi_k)} 
	\frac{(n_k+1)(n_k+2)}{(d_k+1)(n_k-d_k+1)}
	\cdot
	\frac{N_k^2(N_k-1)}{(N_k+1)^2(N_k+2)^2}\\
	& \le 
	\frac{16}{\phi_k(1-\phi_k)} 
	\frac{(n_k+1)(n_k+2)}{(d_k+1)(n_k-d_k+1)}
	\cdot
	\frac{1}{N_k+1}.
	\end{align*}
	From the law of iterated expectation and Lemmas \ref{lemma:marg_dist_D_N1_N} and  \ref{lemma:bound_X_n_minus_X_recip}, 
	\begin{align*}
	\E\left( \frac{\I(V_k>0)}{V_k} \right) & = 
	\E\left\{
	\E\left( \frac{\I(V_k>0)}{V_k} \mid N_k \right)
	\right\}\\
	& \le
	\frac{16}{\phi_k(1-\phi_k)} 
	\frac{(n_k+1)(n_k+2)}{(d_k+1)(n_k-d_k+1)}
	\cdot
	\E\left( \frac{1}{N_k+1} \right)\\
	& = 
	\frac{16}{\phi_k(1-\phi_k)} 
	\frac{(n_k+1)(n_k+2)}{(d_k+1)(n_k-d_k+1)} \frac{1}{g_k(n_k+1)}
	\left\{1 - (1-g_k)^{n_k+1}\right\}\\
	& \le 
	\frac{16}{\phi_k(1-\phi_k)} 
	\frac{(n_k+2)}{(d_k+1)(n_k-d_k+1)} \frac{1}{g_k}\\
	& =  
	\frac{16}{g_k \phi_k(1-\phi_k)} 
	\frac{(1+2n_k^{-1})}{(h_k+n_k^{-1})(1-h_k+n_k^{-1})} \frac{1}{n_k}. 
	\end{align*}
	Therefore, Lemma \ref{lemma:bound_mean_Vk_inv} holds. 
\end{proof}

\begin{proof}[Proof of Lemma \ref{lemma:prob_limit_Vk}]
	We first consider case (i). 
	From Lemma \ref{lemma:marg_dist_D_N1_N}, the property of binomial distribution, and the Markov inequality, we have 
	\begin{align*}
	\frac{D_{k}}{n} & = \frac{d_{k}}{n} \frac{D_{k}}{d_{k}} =  
	\frac{d_{k}}{n}
	\left\{
	g_{k} + O_\pr \left(
	\frac{1}{d_{k}^{1/2}}
	\right)
	\right\}
	=
	\frac{d_{k}}{n}g_{k} + O_\pr \left(
	\frac{1}{n^{1/2}}
	\right), \\
	\frac{N_{k}}{n} & = 
	\frac{n_k}{n} \frac{N_k}{n_k} = \frac{n_k}{n} 
	\left\{
	g_k + 
	O_\pr \left(
	\frac{1}{n_k^{1/2}}
	\right)
	\right\}
	= \frac{n_k}{n} g_k + O_\pr \left(\frac{1}{n^{1/2}}\right),\\
	\frac{N_{1k}}{n} & = 
	\frac{n_k}{n}\frac{N_{1k}}{n_k} = 
	\frac{n_k}{n}
	\left\{
	g_k \phi_k + O_\pr \left(
	\frac{1}{n_k^{1/2}}
	\right)
	\right\} = 
	\frac{n_k}{n} g_k \phi_k + O_\pr \left(
	\frac{1}{n^{1/2}}
	\right).
	\end{align*}
	These imply that 
	\begin{align*}
	& \quad \ \frac{D_{k}(N_{k}-D_{k}) N_{1k} (N_{k} - N_{1k})}{n^4}
	\\
	& = \frac{D_k}{n} \left(
	\frac{N_k}{n} - \frac{D_k}{n}
	\right)
	\frac{N_{1k}}{n}
	\left(
	\frac{N_k}{n} - \frac{N_{1k}}{n}
	\right)\\
	& = 
	\frac{d_k}{n} g_k \left(
	\frac{n_k}{n} g_k - \frac{d_k}{n} g_k
	\right)
	\frac{n_k}{n} g_k \phi_k
	\left(
	\frac{n_k}{n} g_k - \frac{n_k}{n} g_k \phi_k
	\right) + O_\pr \left(
	\frac{1}{n^{1/2}}
	\right)\\
	& = 
	g_k^4 \phi_k (1-\phi_k) \cdot \frac{d_k}{n} \cdot
	\frac{n_k-d_k}{n} \cdot
	\left( \frac{n_k}{n} \right)^2 + O_\pr \left(
	\frac{1}{n^{1/2}}
	\right),
	\end{align*}
	and 
	\begin{align*}
	\frac{N_{k}^2(N_{k}-1)}{n^3} & = 
	\left( \frac{N_k}{n} \right)^2 
	\left(
	\frac{N_k}{n} - \frac{1}{n}
	\right)
	= 
	g_k^3
	\left( \frac{n_k}{n} \right)^3 + O_\pr \left(
	\frac{1}{n^{1/2}}
	\right). 
	\end{align*}
	Because $n_k /n \ge d_k/n$ and $g_k = p_1 G_1(t_k) + p_0 G_0(t_k)$, we can know that in case (i), 
	$g_k^3(n_k/n)^3$ has a positive limit. Thus, 
	\begin{align}\label{eq:prob_limit_Vk_case1}
	\frac{V_{k}}{n} & = \frac{D_{k}(N_{k}-D_{k}) N_{1k} (N_{k} - N_{1k})/n^4}{ N_{k}^2(N_{k}-1)/n^3}
	\\
	& = 
	\frac{
		g_k^4 \phi_k (1-\phi_k) \cdot d_k/n \cdot
		(n_k-d_k)/n \cdot
		( n_k/n )^2
	}{
		g_k^3
		( n_k/n )^3 
	} + o_\pr (1)
	\nonumber
	\\
	& = 
	g_k \phi_k (1-\phi_k) 
	\frac{n_k}{n}
	\frac{
		d_k
		(n_k-d_k) 
	}{
		n_k^2
	} + o_\pr (1) 
	\nonumber
	\\
	& = 
	g_k \phi_k (1-\phi_k) 
	\frac{n_k}{n}
	h_k(1-h_k) + o_\pr (1).
	\nonumber
	\end{align}
	From \eqref{eq:mean_Vk}, 
	\begin{align*}
	\frac{\E(V_k)}{n}
	& = 
	\frac{n_k}{n}
	\frac{n_k}{n_k-1}  h_k (1-h_k) \phi_k (1-\phi_k) 
	\left\{
	g_k - \frac{1 - (1-g_k)^{n_k}}{n_k} 
	\right\}.
	\end{align*}
	In case (i), $n_k \ge d_k \rightarrow \infty$ as $n\rightarrow \infty$, and thus
	\begin{align}\label{eq:limit_mean_Vk_case1}
	\frac{\E(V_k)}{n}
	& = 
	\frac{n_k}{n}
	h_k (1-h_k) \phi_k (1-\phi_k) 
	g_k + o(1) \\
	& = 
	g_k\phi_k (1-\phi_k)  \frac{n_k}{n}
	h_k (1-h_k) 
	+ o(1).
	\nonumber
	\end{align}
	From \eqref{eq:prob_limit_Vk_case1} and \eqref{eq:limit_mean_Vk_case1}, $V_k/n = \E(V_k)/n + o_\pr(1)$.

	We then consider case (ii). 
	By definition and from \eqref{eq:mean_Vk}, 
	\begin{align}\label{eq:bound_Vk_divided_n_0}
	\frac{\E(V_k)}{n} & = 
	\frac{1}{n}
	\frac{n_k^2}{n_k-1}  h_k (1-h_k) \phi_k (1-\phi_k) 
	\left\{
	g_k - \frac{1 - (1-g_k)^{n_k}}{n_k} 
	\right\}	
	\\
	& \leq \frac{1}{n}\frac{n_k^2}{n_k-1}  h_k (1-h_k) \phi_k (1-\phi_k) g_k
	\nonumber
	\\
	& = \frac{n_k}{n_k-1} (1-h_k) \frac{d_k}{n} 
	\frac{p_1G_1(t_k) p_0G_0(t_k)}{p_1G_1(t_k) + p_0G_0(t_k)}
	\nonumber
	\\
	& \leq 2 (1-h_k) \frac{d_k}{n} \min \left\{
	p_1G_1(t_k), p_0G_0(t_k)
	\right\}. 
	\nonumber
	\end{align}
	In case (ii), if at least one of $d_k/n, 
	p_1G_1(t_k)$ and $p_0G_0(t_k)$ converge to zero as $n\rightarrow\infty$, \eqref{eq:bound_Vk_divided_n_0} implies that
	$
	\E(V_k)/n = o(1);
	$
	otherwise, we must have $\lim_{n\rightarrow \infty} h_k = 1$, and from \eqref{eq:bound_Vk_divided_n_0}, 
	$
	\E(V_k)/n = o(1).
	$
	By the Markov inequality, 
	$$
	V_k/n = O_\pr\left(
	\frac{\E(V_k)}{n}
	\right) = o_\pr(1). 
	$$
	
	From the above, Lemma \ref{lemma:prob_limit_Vk} holds. 
\end{proof}

\begin{proof}[Proof of Lemma \ref{lemma:bound_mean_Vk_inv_all}]
	We first consider case (i) in which $k \in \mathcal{K}_0$. 
	In case (i), 
	the limits of $h_k$ and $\phi_k$ are in $(0,1),$ 
	and the limits of $n_k/n$ and $g_k$ are positive. 
	Thus, from Lemmas \ref{lemma:bound_mean_Vk_inv} and \ref{lemma:prob_limit_Vk}, we have 
	\begin{align}\label{eq:limit_meanV_V_Vinv}
	\lim_{n\rightarrow \infty} \frac{\E(V_k)}{n} > 0, \quad 
	\frac{V_k}{n} = \frac{\E(V_k)}{n} + o_\pr(1), 
	\quad
	\lim_{n\rightarrow \infty} \E\left( \frac{\I(V_k>0)}{V_k} \right) = 0.
	\end{align}
	Because 
	\begin{align*}
	\E\left( \frac{\I(V_k>0)}{V_k^{1/2}} \right) \le 
	\left\{ \E\left( \frac{\I(V_k>0)}{V_k} \right) \right\}^{1/2},
	\end{align*}
	we can know that
	\begin{align*}
	\lim_{n\rightarrow \infty} \E\left( \frac{\I(V_k>0)}{V_k^{1/2}} \right) = 0. 
	\end{align*}
	Let $\omega_k = \lim_{n\rightarrow \infty} \E(V_k)/n > 0$. 
	From \eqref{eq:limit_meanV_V_Vinv},  
	$V_k/n \convergep \omega_k$. 
	By definition, this implies that
	\begin{align*}
	\pr(V_k = 0) & \le 
	\pr\left(
	\frac{V_k}{n} < \frac{\omega_k}{2} 
	\right)
	\le 
	\pr\left(
	\left|
	\frac{V_k}{n} - \omega_k
	\right|
	\ge \frac{\omega_k}{2} 
	\right)
	= o(1).
	\end{align*}

	We then consider case (ii) in which $k \notin \mathcal{K}_0.$  
	By the definition of $\mathcal{K}_0$ and Lemma \ref{lemma:prob_limit_Vk},
	$\E(V_k)/n = o(1)$, and 
	$V_k/n = o_\pr(1).$
	
	From the above, Lemma \ref{lemma:bound_mean_Vk_inv_all} holds. 
\end{proof}

\subsection{Proof of Theorem \ref{thm:CLT_discrete}}

\begin{proof}[Proof of Theorem \ref{thm:CLT_discrete}]
	Under Condition 2, we define a subset of $\{1, 2, \ldots, K\}$ the same as in Lemma \ref{lemma:bound_mean_Vk_inv_all}: 
	\begin{align*}
	\mathcal{K}_0 & = \left\{k_1, \ldots, k_J \right\} \\
	& \equiv 
	\left\{k: \lim_{n\rightarrow \infty}\frac{d_k}{n} > 0,  
	\lim_{n\rightarrow \infty} h_k < 1
	\lim_{n\rightarrow \infty}G_1(t_k) > 0,  \lim_{n\rightarrow \infty} G_0(t_k) > 0 \right\},
	\end{align*}
	where $J\ge 1$ is the cardinality of $\mathcal{K}_0$. 
	For any positive integer $b$, let $\bm{0}_{b\times 1}$ denote a $b$ dimensional column vector with all elements being zero, and
	$\bm{I}_{b\times b}$ denote a $b \times b$ identity matrix.
	
	We first prove that 
	\begin{align}\label{eq:mult_clt}
	\left(
	\frac{D_{1k_1} - \ED_{k_1}}{V_{k_1}^{1/2}}, 
	\frac{D_{1k_2} - \ED_{k_2}}{V_{k_2}^{1/2}}, \ldots, 
	\frac{D_{1k_{J}} - \ED_{k_{J}}}{V_{k_{J}}^{1/2}}
	\right)' \converged \mathcal{N}(\bm{0}_{J\times 1}, \bm{I}_{J\times J}).
	\end{align}
	We prove \eqref{eq:mult_clt} by iteratively proving that for $b=1, , \ldots, J$, 
	\begin{align}\label{eq:asym_normal_upto_j}
	\left(
	\frac{D_{1k_1} - \ED_{k_1}}{V_{k_1}^{1/2}}, \ldots, 
	\frac{D_{1k_{b}} - \ED_{k_{b}}}{V_{k_{b}}^{1/2}}
	\right)' \converged \mathcal{N}(\bm{0}_{b \times 1}, \bm{I}_{b\times b}).
	\end{align}
	Consider first \eqref{eq:asym_normal_upto_j} with $b=1$. From Theorem \ref{thm:mds} and Lemma \ref{lemma:clt_hyper}, for any $x_1\in \mathbb{R}$, 
	\begin{align*}
	\left|
	\pr\left(
	\frac{D_{1k_1} - \ED_{k_1}}{V_{k_1} ^{1/2}}
	\le x_1 \mid \mathcal{F}_0
	\right) - \Phi(x_1)
	\right| \le \gamma 
	\frac{\I(V_{k_1}>0)}{V_{k_1}^{1/2}}
	+ \I(V_{k_1} = 0). 
	\end{align*}
	By the law of iterated expectation, 
	\begin{align}\label{eq:clt_first_one}
	& \quad \ \left|
	\pr\left(
	\frac{D_{1k_1} - \ED_{k_1}}{V_{k_1} ^{1/2}} 
	\le x_1
	\right) - \Phi(x_1)
	\right|
	\\
	& = 
	\left|
	\E\left\{\pr\left(
	\frac{D_{1{k_1}} - \ED_{k_1}}{V_{k_1} ^{1/2}} 
	\le x_1 \mid \mathcal{F}_0
	\right)\right\} - \Phi(x_1)
	\right|
	\nonumber
	\\
	& 
	\le 
	\E\left\{
	\left|
	\pr\left(
	\frac{D_{1k_1} - \ED_{k_1}}{V_{k_1} ^{1/2}} 
	\le x_1 \mid \mathcal{F}_0
	\right) - \Phi(x_1)
	\right|\right\}
	\nonumber
	\\
	& \le 
	\gamma 
	\E\left( 
	\frac{\I(V_{k_1}>0)}{V_{k_1}^{1/2} }
	\right) + \pr(V_{k_1} = 0).
	\nonumber
	\end{align}
	From Lemma \ref{lemma:bound_mean_Vk_inv_all} and  \eqref{eq:clt_first_one}, for any $x\in \mathbb{R}$, 
	\begin{align*}
	\pr\left(
	\frac{D_{1k_1} - \ED_{k_1}}{V_{k_1} ^{1/2}} 
	\le x_1 
	\right) - \Phi(x_1) = o(1),
	\end{align*}
	i.e., \eqref{eq:asym_normal_upto_j} holds for $b=1.$ 
	Assume that \eqref{eq:asym_normal_upto_j} holds for $b = j < J$, 
	we prove that \eqref{eq:asym_normal_upto_j} also holds for $b = j+1$. 
	Let $\mathcal{G}_j$ denote the event that $V_{k_{q}}^{-1/2} (D_{1k_{q}} - \ED_{k_{q}}) \le x_q$ for $q=1,\ldots, j$. 
	By the law of iterated expectation, 
	\begin{align*}
	& \quad \  \pr\left(
	\frac{D_{1k_{j+1}} - \ED_{k_{j+1}}}{V_{k_{j+1}}^{1/2}} \le x_{j+1} \mid 
	\mathcal{G}_j
	\right) \\
	& = 
	\E\left\{
	\pr\left(
	\frac{D_{1k_{j+1}} - \ED_{k_{j+1}}}{V_{k_{j+1}}^{1/2}}
	\le x_{j+1} \mid 
	\mathcal{F}_{k_{j+1}-1}, \mathcal{G}_j
	\right) \mid 
	\mathcal{G}_j
	\right\}\\
	& = \E\left\{
	\pr\left(
	\frac{D_{1k_{j+1}} - \ED_{k_{j+1}}}{V_{k_{j+1}}^{1/2}}
	\le x_{j+1} \mid 
	\mathcal{F}_{k_{j+1}-1}
	\right) \mid 
	\mathcal{G}_j
	\right\},
	\end{align*}
	where the last equality holds because $\mathcal{F}_{k_{j+1}-1}$ contains more information than $\mathcal{G}_j$ by definition. 
	Theorem \ref{thm:mds} and Lemma \ref{lemma:clt_hyper} imply that for any $x_{j+1} \in \mathbb{R}$, 
	\begin{align*}
	& \quad \ 
	\left| 
	\pr\left(
	\frac{D_{1k_{j+1}} - \ED_{k_{j+1}}}{V_{k_{j+1}}^{1/2}} \le x_{j+1} \mid 
	\mathcal{F}_{k_{j+1}-1}
	\right) - \Phi(x_{j+1})
	\right| \\
	& \le 
	\gamma 
	\frac{\I(V_{k_{j+1}}>0)}{V_{k_{j+1}}^{1/2}}
	+ \I(V_{k_{j+1}} = 0). 
	\end{align*}
	By the law of iterated expectation, 
	\begin{align}\label{eq:bound_cond_dist}
	& \quad \ 
	\left|
	\pr\left(
	\frac{D_{1k_{j+1}} - \ED_{k_{j+1}}}{V_{k_{j+1}}^{1/2}} \le x_{j+1} \mid 
	\mathcal{G}_j
	\right)
	- \Phi(x_{j+1})
	\right|
	\\
	& = 
	\left| \E\left\{
	\pr\left(
	\frac{D_{1k_{j+1}} - \ED_{k_{j+1}}}{V_{k_{j+1}}^{1/2}} \le x_{j+1} \mid 
	\mathcal{F}_{k_{j+1}-1}
	\right) \mid 
	\mathcal{G}_j
	\right\} - \Phi(x_{j+1})
	\right|
	\nonumber
	\\
	& \le 
	\E\left\{
	\left| 
	\pr\left(
	\frac{D_{1k_{j+1}} - \ED_{k_{j+1}}}{V_{k_{j+1}}^{1/2}} \le x_{j+1} \mid 
	\mathcal{F}_{k_{j+1}-1}
	\right) - \Phi(x_{j+1})
	\right|
	\mid 
	\mathcal{G}_j
	\right\} 
	\nonumber
	\\
	& \le 
	\E\left(
	\gamma 
	\frac{\I(V_{k_{j+1}}>0)}{V_{k_{j+1}}^{1/2}}
	+ \I(V_{k_{j+1}} = 0)
	\mid 
	\mathcal{G}_j
	\right).
	\nonumber
	\end{align}
	By the law of total expectation, we can further bound \eqref{eq:bound_cond_dist} by 
	\begin{align}\label{eq:bound_cond_dist_2}
	& \quad \ 
	\left|
	\pr\left(
	\frac{D_{1k_{j+1}} - \ED_{k_{j+1}}}{V_{k_{j+1}}^{1/2}} \le x_{j+1} \mid 
	\mathcal{G}_j
	\right)
	- \Phi(x_{j+1})
	\right|
	\nonumber
	\\
	& \le 
	\E\left(
	\gamma 
	\frac{\I(V_{k_{j+1}}>0)}{V_{k_{j+1}}^{1/2}}
	+ \I(V_{k_{j+1}} = 0)
	\mid 
	\mathcal{G}_j
	\right)
	\nonumber
	\\
	& \leq 
	\frac{1}{\pr(\mathcal{G}_j)}
	\left\{
	\gamma 
	\E\left(
	\frac{\I(V_{k_{j+1}}>0)}{V_{k_{j+1}}^{1/2}}
	\right) + 
	\pr\left(
	V_{k_{j+1}} = 0
	\right) 
	\right\}.
	\end{align}
	Because 
	\begin{align*}
	& \quad \ \pr\left(
	\mathcal{G}_j, 
	\frac{D_{1k_{j+1}} - \ED_{k_{j+1}}}{V_{k_{j+1}}^{1/2}} \le x_{j+1}
	\right) - \prod_{q=1}^{j+1} \Phi(x_q)\\
	& = 
	\pr\left( 
	\mathcal{G}_j
	\right)
	\left\{
	\pr\left(
	\frac{D_{1k_{j+1}} - \ED_{k_{j+1}}}{V_{k_{j+1}}^{1/2}}
	\le x_{j+1} \mid \mathcal{G}_j
	\right) - \Phi(x_{j+1})
	\right\} \\
	& \quad \ + 
	\left\{\pr\left( 
	\mathcal{G}_j
	\right) -  \prod_{q=1}^{j} \Phi(x_q) 
	\right\}
	\Phi(x_{j+1}),
	\end{align*}
	from \eqref{eq:bound_cond_dist_2}, 
	we have
	\begin{align}\label{eq:bound_joint_dist}
	& \quad \ 
	\left| \pr\left(
	\mathcal{G}_j, 
	\frac{D_{1k_{j+1}} - \ED_{k_{j+1}}}{V_{k_{j+1}}^{1/2}} \le x_{j+1}
	\right) - \prod_{q=1}^{j+1} \Phi(x_q)
	\right|
	\\
	& \le 
	\pr\left( 
	\mathcal{G}_j
	\right)
	\left|
	\pr\left(
	\frac{D_{1k_{j+1}} - \ED_{k_{j+1}}}{V_{k_{j+1}}^{1/2}}
	\le x_{j+1} \mid \mathcal{G}_j
	\right) - \Phi(x_{j+1})
	\right| 
	\nonumber
	\\
	& \quad \ 
	+ 
	\Phi(x_{j+1})
	\left|\pr\left( 
	\mathcal{G}_j
	\right) -  \prod_{q=1}^{j} \Phi(x_q) 
	\right|
	\nonumber
	\\
	& \le 
	\gamma 
	\E\left(
	\frac{\I(V_{k_{j+1}}>0)}{V_{k_{j+1}}^{1/2}}
	\right) + 
	\pr\left(
	V_{k_{j+1}} = 0
	\right)  + \left|\pr\left( 
	\mathcal{G}_j
	\right) -  \prod_{q=1}^{j} \Phi(x_q) 
	\right|. 
	\nonumber
	\end{align}
	From Lemma \ref{lemma:bound_mean_Vk_inv_all}, the first two terms in \eqref{eq:bound_joint_dist} are both $o(1)$.
	Because \eqref{eq:asym_normal_upto_j} holds for $b = j$, the third term in \eqref{eq:bound_joint_dist} is also $o(1)$. 
	Therefore, 
	\begin{align*}
	& \quad \ 
	\pr\left(
	\frac{D_{1k_1} - \ED_{k_1}}{V_{k_1}^{1/2}}
	\le x_1, \ldots, 
	\frac{D_{1k_{j+1}} - \ED_{k_{j+1}}}{V_{k_{j+1}}^{1/2}}
	\le x_j
	\right) - \prod_{q=1}^{j+1} \Phi(x_q)\\
	& =
	\pr\left(
	\mathcal{G}_j, 
	\frac{D_{1k_{j+1}} - E_{1k_{j+1}}}{V_{k_{j+1}}^{1/2}} \le x_{j+1}
	\right) - \prod_{q=1}^{j+1} \Phi(x_q) \\
	& = o(1),
	\end{align*}
	i.e., 
	\eqref{eq:asym_normal_upto_j} holds for $b = j+1$. 
	From the above, 
	we can know that 
	\eqref{eq:asym_normal_upto_j} holds for $b=J$, i.e., \eqref{eq:mult_clt} holds. 
	
	Second, we prove the asymptotic Gaussianity of the logrank statistic. 
	By definition, 
	\begin{align}\label{eq:logrank_two_parts}
	\frac{L}{U^{1/2}} & = 
	\frac{\sum_{k=1}^K (D_{1k} - \ED_{k})}{U^{1/2}} = 
	\frac{\sum_{j=1}^{J} (D_{1k_j} - \ED_{k_j})}{U^{1/2}} + 
	\frac{\sum_{k\notin \mathcal{K}_0 } (D_{1k} - \ED_{k})}{U^{1/2}} 
	\nonumber
	\\
	& = 
	\sum_{j=1}^{J} \left( \frac{V_{k_j}}{U} \right)^{1/2}  \cdot
	\frac{D_{1k_j} - \ED_{k_j}}{V_{k_j}^{1/2}}
	+ 
	\sum_{k\notin \mathcal{K}_0}
	\frac{n^{-1/2}(D_{1k} - \ED_{k})}{n^{-1/2} U^{1/2}}.
	\end{align}
	We consider the two terms in \eqref{eq:logrank_two_parts} separately. 
	For the first term in \eqref{eq:logrank_two_parts}, from Lemma \ref{lemma:bound_mean_Vk_inv_all}, 
	\begin{align}\label{eq:limit_U_divided_by_n}
	\frac{U}{n} = 
	\sum_{k=1}^{K}\frac{V_k}{n}
	= \sum_{k=1}^{K}\frac{\E(V_k)}{n} + o_\pr(1)
	= \sum_{j=1}^{J} \frac{\E(V_{k_j})}{n} + o_\pr(1), 
	\end{align}
	and thus
	\begin{align*}
	\frac{V_{k_j}}{U} & =\left(
	\frac{U}{n}
	\right)^{-1} \cdot \frac{V_{k_j}}{n} 
	= 
	\left\{
	\sum_{j=1}^{J} \frac{\E(V_{k_j})}{n} + o_\pr(1)
	\right\}^{-1} \cdot 
	\left\{\frac{\E(V_{k_j})}{n} + o_\pr(1)
	\right\}\\
	& = 
	\left\{
	\sum_{j=1}^{J} \frac{\E(V_{k_j})}{n} 
	\right\}^{-1} \cdot 
	\frac{\E(V_{k_j})}{n} + o_\pr(1).
	\end{align*}
	From \eqref{eq:mult_clt}, by Slutsky's theorem, we have  
	\begin{align}\label{eq:logrank_fist_part}
	\sum_{j=1}^{J} \left( \frac{V_{k_j}}{U} \right)^{1/2}  \cdot
	\frac{D_{1k_j} - \ED_{k_j}}{V_{k_j}^{1/2}} \converged \mathcal{N}(0,1). 
	\end{align}
	For the second term in \eqref{eq:logrank_two_parts}, 
	from Lemma \ref{lemma:bound_mean_Vk_inv_all}, for any $k \notin \mathcal{K}_0$, 
	\begin{align*}
	\E\left\{\left( \frac{D_{1k} - \ED_{k}}{n^{1/2}} \right)^2\right\} & = 
	\frac{1}{n}
	\E\left[
	\E\left\{\left( D_{1k} - \ED_{k} \right)^2 \mid \mathcal{F}_{k-1}\right\}
	\right] = \frac{\E(V_k)}{n} = o(1).
	\end{align*}
	Thus, by the Markov inequality, $n^{-1/2}(D_{1k} - \ED_{k}) = o_\pr(1)$. 
	\eqref{eq:limit_U_divided_by_n} then implies 
	\begin{align}\label{eq:logrank_second_part}
	\sum_{k\notin \mathcal{K}_0}
	\frac{n^{-1/2}(D_{1k} - \ED_{k})}{n^{-1/2} U^{1/2}} = o_\pr(1). 
	\end{align}
	From \eqref{eq:logrank_two_parts}, \eqref{eq:logrank_fist_part}, \eqref{eq:logrank_second_part} and Slutsky's theorem, we have
	\begin{align*}
	\frac{L}{U^{1/2}} & = 
	\sum_{j=1}^{J} \left( \frac{V_{k_j}}{U} \right)^{1/2}  \cdot
	\frac{D_{1k_j} - \ED_{k_j}}{V_{k_j}^{1/2}}
	+ 
	\sum_{k\notin \mathcal{K}_0}
	\frac{n^{-1/2}(D_{1k} - \ED_{k})}{n^{-1/2} U^{1/2}} \converged \mathcal{N}(0,1).
	\end{align*}
	Moreover, \eqref{eq:limit_U_divided_by_n} implies that $U/\E(U) = 1 + o_\pr(1)$, i.e., $U$ is a consistent estimator for $\E(U) = \Var(L)$. 
	
	From the above, Theorem \ref{thm:CLT_discrete} holds. 
\end{proof}

\section{Asymptotic distribution of the stratified logrank statistic}\label{app:logrank_strata}

\subsection{Lemmas}

To prove Theorem \ref{thm:CLT_SLR}, we need the following two lemmas. 

\begin{lemma}\label{lemma:sum_indep_normal_varying_coef}
	Let $\{(X_{n1}, \ldots, X_{nS}): n = 1, 2, \ldots, \infty\}$ be a sequence of random vectors satisfying that (a) $(X_{n1}, X_{n2}, \ldots, X_{nS})$ are mutually independent, and (b) for $1\le s\le S$, $X_{ns} \converged \mathcal{N}(0,1)$ as $n\rightarrow \infty$. 
	Let $\{(a_{n1}, \ldots, a_{nS}): n = 1, 2, \ldots, \infty\}$ be a sequence of constant vectors satisfying that $\sum_{s=1}^S a_{ns}^2 = 1$ for all $n\ge 1$. 
	Then $\sum_{s=1}^S a_{ns} X_{ns} \converged \mathcal{N}(0,1)$ as $n\rightarrow \infty.$ 
\end{lemma}

\begin{lemma}\label{lemma:sum_converge_in_prob}
	Let $\{(Y_{n1}, \ldots, Y_{nS}): n = 1, 2, \ldots, \infty\}$ be a sequence of random vectors, and $\{(b_{n1}, \ldots, b_{nS}): n = 1, 2, \ldots, \infty\}$ be a sequence of constant vectors satisfying that $b_{ns}>0$ for all $n\ge 1$ and $1\le s\le S$. 
	As $n\rightarrow \infty,$
	if  $b_{ns}^{-1}Y_{ns} \convergep 1$ for $1\le s\le S$, then 
	\begin{align}\label{eq:sum_converge_in_prob}
	\frac{\sum_{s=1}^S Y_{ns}}{\sum_{s=1}^S b_{ns}}
	\convergep 1.
	\end{align}
\end{lemma}

\subsection{Proofs of the lemmas}

\begin{proof}[Proof of Lemma \ref{lemma:sum_indep_normal_varying_coef}]
	For $1\le s\le S$ and $n \ge 1$, let $F_{ns}(x) \equiv \pr(X_{ns}\le x)$ be the distribution function of $X_{ns}$, 
	and $F_{ns}^{-1}(p) = \inf \{x: F_{ns}(x) \ge p\}$ be the corresponding quantile function. 
	Let $(U_1, \ldots, U_S)$ be i.i.d.\ random variables uniformly distributed on $[0,1]$, 
	$\Phi(\cdot)$ be the distribution function of $\mathcal{N}(0,1)$, 
	and $\Phi^{-1}(\cdot)$ be the quantile function of $\mathcal{N}(0,1)$. 
	From \citet[][Lemma 21.2]{van2000asymptotic}, 
	$F_{ns}^{-1}(U_s) \convergeas \Phi^{-1}(U_s)$, for $1\le s\le S$. 
	This implies 
	\begin{align*}
	& \quad \ \left| \sum_{s=1}^S a_{ns} F_{ns}^{-1}(U_s) - \sum_{s=1}^S a_{ns} \Phi^{-1}(U_s)
	\right|\\
	& = \sum_{s=1}^S \left| a_{ns} \right|
	\left| F_{ns}^{-1}(U_s) - \Phi^{-1}(U_s) \right|
	\le 
	\left(\sum_{s=1}^S a_{ns}^2 \right)^{1/2} \cdot
	\sum_{s=1}^S 
	\left| F_{ns}^{-1}(U_s) - \Phi^{-1}(U_s) \right|\\
	& = \sum_{s=1}^S 
	\left| F_{ns}^{-1}(U_s) - \Phi^{-1}(U_s) \right| \convergeas 0.
	\end{align*}
	Because $\sum_{s=1}^S a_{ns} \Phi^{-1}(U_s)\sim \mathcal{N}(0,1)$ for all $n\ge 1$, 
	by Slutsky's theorem, 
	\begin{align*}
	\sum_{s=1}^S a_{ns} F_{ns}^{-1}(U_s) & = 
	\left( \sum_{s=1}^S a_{ns} F_{ns}^{-1}(U_s) - \sum_{s=1}^S a_{ns} \Phi^{-1}(U_s) \right) + 
	\sum_{s=1}^S a_{ns} \Phi^{-1}(U_s)
	\\
	& \converged \mathcal{N}(0,1).
	\end{align*}
	From the property of quantile functions and the mutually independence of $U_s$'s, 
	$(X_{n1}, \ldots, X_{nS}) \sim (F_{n1}^{-1}(U_1), \ldots, F_{nS}^{-1}(U_S))$, and thus
	$
	\sum_{s=1}^S a_{ns} X_{ns} \converged \mathcal{N}(0,1).
	$
	Therefore, Lemma \ref{lemma:sum_indep_normal_varying_coef} holds. 
\end{proof}

\begin{proof}[Proof of Lemma \ref{lemma:sum_converge_in_prob}]
	The difference between the two sides of \eqref{eq:sum_converge_in_prob} 
	has the following equivalent forms:
	\begin{align*}
	\frac{\sum_{s=1}^S Y_{ns}}{\sum_{j=1}^S b_{nj}} - 1
	& = \frac{\sum_{s=1}^S (Y_{ns}-b_{ns})}{\sum_{j=1}^S b_{nj}} 
	= \sum_{s=1}^S 
	\frac{b_{ns}}{\sum_{j=1}^S b_{nj}}
	\frac{ (Y_{ns}-b_{ns})}{b_{ns}}
	\\
	& = \sum_{s=1}^S 
	\frac{b_{ns}}{\sum_{j=1}^S b_{nj}}
	\left( \frac{Y_{ns}}{b_{ns}} - 1 \right)
	\end{align*}
	The positivity of the $b_{ns}$'s then implies 
	\begin{align*}
	\left| \frac{\sum_{s=1}^S Y_{ns}}{\sum_{j=1}^S b_{nj}} - 1\right|
	& 
	\le \sum_{s=1}^S 
	\frac{b_{ns}}{\sum_{j=1}^S b_{nj}}
	\left| \frac{Y_{ns}}{b_{ns}} - 1 \right|
	\le
	\sum_{s=1}^S 
	\left| \frac{Y_{ns}}{b_{ns}} - 1 \right| = o_\pr(1).
	\end{align*}
	Therefore, Lemma \ref{lemma:sum_converge_in_prob} holds.
\end{proof}

\subsection{Proof of Theorem \ref{thm:CLT_SLR}}

\begin{proof}[Proof of Theorem \ref{thm:CLT_SLR}]
	Because the treatment assignments $Z_i$'s and the potential censoring times $(C_i(1), C_i(0))$'s are mutually independent for $1\le i \le n$, 
	$(U_{[1]}^{-1/2} L_{[1]}, \ldots, U_{[S]}^{-1/2} L_{[S]})$ are mutually independent. 
	From Theorems \ref{thm:CLT_conti} and \ref{thm:CLT_discrete} and Slutsky's theorem, 
	as $n\rightarrow \infty,$ for $1\le s\le S$, 
	$$
	\frac{L_{[s]}}{\sqrt{E(U_{[s]}) }}
	\converged \mathcal{N}(0,1). 
	$$
	Using Lemma \ref{lemma:sum_indep_normal_varying_coef}, we then have
	\begin{align}\label{eq:strata_CLT_comp1}
	\sum_{s=1}^S 
	\left\{
	\frac{\E(U_{[s]})}{\sum_{j=1}^S \E(U_{[j]})}
	\right\}^{1/2}
	\frac{L_{[s]}}{\sqrt{E(U_{[s]}) }} 
	= 
	\frac{\sum_{s=1}^S  L_{[s]}}{
		\sqrt{ \sum_{s=1}^S \E(U_{[s]}) }
	} 
	\converged \mathcal{N}(0,1).
	\end{align}
	From Theorems \ref{thm:CLT_conti} and \ref{thm:CLT_discrete} and Lemma \ref{lemma:sum_converge_in_prob}, we have
	\begin{align}\label{eq:strata_CLT_comp2}
	\frac{\sum_{s=1}^S U_{[s]}}{\sum_{s=1}^S \E(U_{[s]})}
	\convergep 1.
	\end{align}
	\eqref{eq:strata_CLT_comp1}, \eqref{eq:strata_CLT_comp2}, and Slutsky's theorem imply that 
	\begin{align*}
	\SLR = \frac{\sum_{s=1}^{S} L_{[s]}}{
		\sqrt{\sum_{s=1}^{S} U_{[s]}} 
	}
	& = 
	\left\{ \frac{\sum_{s=1}^S U_{[s]}}{\sum_{s=1}^S \E(U_{[s]})} \right\}^{-1/2}
	\frac{\sum_{s=1}^{S} L_{[s]}}{ \sqrt{ \sum_{s=1}^S \E(U_{[s]}) } }
	\converged \mathcal{N}(0,1).
	\end{align*}
	Therefore, Theorem \ref{thm:CLT_SLR} holds. 
\end{proof}

\end{document}